\documentclass[11pt]{article}
\usepackage[a4paper, margin=2cm, top=1cm, bottom=1.5cm]{geometry}
\usepackage{amsmath, amssymb, amsthm, bm}
\usepackage{enumitem}
\usepackage{bbm}
\usepackage[hidelinks]{hyperref}
\usepackage{color}
\usepackage{comment}
\theoremstyle{plain}

\usepackage{authblk}

\newtheorem{theorem}{Theorem}

\newtheorem{proposition}{Proposition}

\newtheorem{lemma}{Lemma}
\newtheorem*{lemma*}{Lemma}

\theoremstyle{remark}
\newtheorem{remark}{Remark}

\newcommand{\D}{\mathcal{D}}

\newcommand{\pa}[1]{\left(#1\right)}
\newcommand{\cro}[1]{\left[#1\right]}
\newcommand{\ac}[1]{\left\{#1\right\}}

\newcommand{\fr}[1]{\textcolor{red}{Fr: #1}}

\title{On the privacy cost for dependent Gaussian data: \\
spectral density estimation under local differential privacy}

\author[1]{Yann Issartel%
\thanks{\texttt{yann.issartel@telecom-paris.fr}}}
\author[1]{Fran\c{c}ois Roueff%
\thanks{\texttt{francois.roueff@telecom-paris.fr}}}

\affil[1]{LTCI, T\'el\'ecom Paris, Institut Polytechnique de Paris}

\date{}

\begin{document}

\maketitle

\begin{abstract}
  We study the fundamental problem of estimating the dependence
  structure of a centered stationary Gaussian process under local
  differential privacy (LDP). In this setting, the spectral density characterizes the dependence structure of the data and is the quantity to be estimated.
  Our main contribution is to close the open \(\alpha^2\)-versus-\(\alpha^4\)
  gap between the previously known lower and upper bounds on the
  minimax rate.  Specifically, we establish a minimax lower bound
  showing that, over Sobolev-type classes of spectral densities, the
  effective sample size in the high-privacy regime is \(N\alpha^4\),
  rather than the usual \(N\alpha^2\) arising for independent
  observations. This additional privacy cost is caused by the temporal
  dependence between the observations rather than by their marginal
  distributions. The proof relies on a contraction bound for
  privatized dependent Gaussian observations.  Our second contribution
  is a matching upper bound, free of the polylogarithmic losses
  present in previous work.  Rather than applying a generic
  privatization scheme to classical estimators, we construct a
  problem-specific procedure attaining the rate identified by our
  lower bound.  Beyond closing the gaps in spectral density
  estimation, we apply the tools developed for this problem to several
  related questions.  We (i) close the logarithmic gap for fixed-lag
  autocovariance estimation, (ii) show that the \(\alpha^4\) cost
  arises locally around every spectral density bounded away from zero,
  and (iii) establish that classical asymptotic equivalence with an
  independent Gaussian experiment generally fails under LDP.
\end{abstract}

\medskip
\noindent\textbf{Keywords:}
local differential privacy; spectral density estimation; stationary
Gaussian processes; minimax rates.

%\medskip
%\noindent\textbf{MSC 2020:}
%62G05; 62M15; 62C20.

%\tableofcontents

\section{Introduction}  \label{section-intro}

\textit{Applications and privacy motivation.}  Dependent observations
arise in various settings, for example when data are collected
repeatedly over time or across neighboring locations.  Familiar
real-world examples include human mobility trajectories and household
electricity-consumption recorded over time, as well as measurements
collected by spatial sensor networks. In survey data, dependence may
also be induced by the sampling design: random-walk surveys select
households within geographic clusters, sometimes along contiguous
routes~\cite{Ben91,LR85}, whereas snowball sampling recruits
participants successively through social links~\cite{Goodman1961}.

These applications may involve highly sensitive individual data:
random-walk surveys have been used to collect self-reported and
biological information concerning mental health, substance use,
violence, genetic vulnerabilities, and stress~\cite{Flynn13}. Concrete
privacy risks have also been documented for mobility and
energy-consumption data: four spatio-temporal points were sufficient to
uniquely characterize \(95\%\) of the individuals in a large mobility
database~\cite{deMontjoyeEtAl2013}; even aggregate location time series
can be vulnerable to membership-inference attacks%
~\cite{PyrgelisEtAl2018}; and fine-grained smart-meter records can allow
households to be re-identified~\cite{VoyezEtAl2025}. These examples
motivate the development of statistical methods for learning dependence
structures from privatized observations.

\smallskip 

\textit{Local differential privacy.} 
A standard framework for formalizing such privacy requirements is
differential privacy~\cite{Dwo06}.
In the central model, a trusted
curator has access to the original data \(X_1,\ldots,X_N\), and
releases a privatized output.  Under local differential privacy, by
contrast, each observation \(X_i\) is privatized before it is
communicated to the statistician~\cite{duchi2018minimax}.  In other
words, each data holder \(i\) releases a privatized version \(Z_i\) of
their data \(X_i\), so that the statistician has access only to the
privatized sequence \(Z_1,\ldots,Z_N\), rather than to the original
sequence \(X_1,\ldots,X_N\).  The local model therefore removes the
need for a trusted curator, typically at the price of a larger
statistical error.  We only consider the local model in this paper.

\smallskip 

\textit{Inference for dependent data.}  
The statistical literature on local differential privacy has focused
primarily on i.i.d. samples \(X_1,\ldots,X_N\), addressing, among
other problems, density estimation~\cite{duchi2018minimax,But20},
functional estimation~\cite{RS20,butucea2021locally,BRS23}, hypothesis
testing~\cite{BB20,LLL22,BRS23}, and classification~\cite{BB19}.  By
comparison, the statistical theory of local privacy for dependent data
\(X_1,\ldots,X_N\) remains much less developed, although a few recent
contributions have begun to address specific dependent
settings~\cite{AmorinoGloterHalconruy2025,RothAvellaMedina2025},
raising the general and difficult question of how local privacy
affects the estimation of dependence structures.

In this paper, we address this question by studying arguably one of the
simplest and most fundamental problems in dependence estimation:
estimating the spectral density of a centered stationary Gaussian process.
For such a process, the spectral density completely characterizes its
covariance structure, which in turn determines its law.
This makes the problem a natural benchmark for inference on dependence.
Although this problem is classical and its non-private theory is well
understood~\cite{Gol93,Comte2001}, to the best of our knowledge, the precise
statistical cost imposed by local privacy remains unresolved.
We fill this gap for Gaussian time series.

%%%%%%%%%%%%%%%

\paragraph{Spectral density estimation problem.}

Let \(X=(X_t)_{t\in\mathbb Z}\) be a centered real-valued stationary
Gaussian process. Its law is fully determined by the autocovariance
function \(\gamma\colon\mathbb Z\to\mathbb R\), defined by
\(\gamma(h)=\operatorname{Cov}(X_0,X_h)\), \(h\in\mathbb Z\).
The corresponding spectral density \(f\) admits the Fourier
representation
\begin{equation}
\label{eq:f-fourier-representation}
f(\lambda)
=
\frac{1}{2\pi}
\sum_{h\in\mathbb Z}\gamma(h)e^{-ih\lambda},
\qquad \lambda\in[-\pi,\pi],
\end{equation}
whenever the series is well defined; for example, in
\(L^2([-\pi,\pi])\) if \(\sum_{h\in\mathbb Z}\gamma(h)^2<\infty\).
Although \(\gamma\) and \(f\) contain the same information,
nonparametric models for stationary time series are typically
parameterized by spectral densities, partly because it is easier to
verify that a candidate \(f\) is a valid spectral density for
a real-valued stationary Gaussian process.
Indeed, a necessary and sufficient condition for \(f\) is to be
nonnegative, integrable, and even on \([-\pi,\pi]\).
We refer to \cite{brockwell1991time} for
background on this topic.

\textit{Parameter class.}
Given an unknown spectral density \(f\), the standard statistical
problem is to estimate \(f\) from the first \(N\) observations
\(X_1,\ldots,X_N\), with estimation error measured in
\(L^2([-\pi,\pi])\).
The autocovariance function associated with \(f\) is denoted by
\(\gamma_f\) and is given by Fourier inversion:
\begin{equation}
\label{eq:gamma-fourier-inversion}
\gamma_f(h)
=
\int_{-\pi}^{\pi}f(\lambda)e^{ih\lambda}\,d\lambda,
\qquad h\in\mathbb Z.
\end{equation}
Its value at zero, \(\gamma_f(0)\), is the marginal variance and it
measures the overall scale of the process. When \(\gamma_f(0)>0\), the
normalized dependence structure is described by the autocorrelation
function \(\rho_f(h)=\gamma_f(h)/\gamma_f(0)\), \(h\in\mathbb Z\).
This leads to the following parameter class, which controls these two
components separately: for fixed constants \(C_0\geq 1\), \(C_1>0\),
and \(s>1/2\),
\begin{equation}
\label{defi-parameter-class-F}
\mathcal F(C_0,C_1,s)
=
\left\{
f \text{ spectral density}:
\gamma_f(0)\in[C_0^{-1},C_0]
\ \text{ and }\
\sum_{h\in\mathbb Z\setminus\{0\}}
|h|^{2s}|\rho_f(h)|^2\le C_1^2
\right\}.
\end{equation}
In this class, the variance \(\gamma_f(0)\) is bounded from below and
above, while the autocorrelation sequence satisfies a Sobolev-type
decay condition.  Note that, since \(s>1/2\),
\((\gamma_f(h))_{h\in\mathbb Z}\) is absolutely summable for every
\(f\in\mathcal F(C_0,C_1,s)\), and hence the Fourier series in
\eqref{eq:f-fourier-representation} is well defined pointwise.

%%%%%

\paragraph{Local differential privacy (LDP).} 
 
Before introducing the statistical model, we recall two standard
measure-theoretic notions used below: Markov kernels and standard
Borel spaces. 
Throughout, the \(\sigma\)-algebra of each measurable space is omitted
from the notation.
For two measurable spaces \(\mathcal X\) and \(\mathcal Z\), we represent a Markov kernel \(K\) from \(\mathcal X\) to \(\mathcal Z\) as a map \((x,A)\mapsto K(A\mid x)\), defined for all \(x\in\mathcal X\) and all measurable sets \(A\subseteq\mathcal Z\), with the usual conditions that, for \(x\) fixed, \(A\mapsto K(A\mid x)\) is a probability measure on \(\mathcal Z\) and, for \(A\) fixed, \(x\mapsto K(A\mid x)\) is measurable; see, for instance,~\cite{polyanskiy2025information}. 
In the following, the term standard Borel space refers to any measurable
space isomorphic to a Borel subset of \([0,1]\). 
Examples of standard Borel spaces include all topological spaces, equipped with their Borel \(\sigma\)-algebras, that admit complete and separable metrizations; see~\cite{kallenberg2021foundations}.

\textit{Observation model.}   Recall that, in the
locally private model, the statistician receives a privatized variable
\(Z_i\) in place of each \(X_i\); write \(Z=(Z_1,\ldots,Z_N)\) for the
resulting observations.  
For each \(i\in [N]\), where \([N]:=\{1,\ldots,N\}\), let \(K_i\) be the Markov kernel such that \(K_i(\cdot\mid x)\) is the distribution of \(Z_i\) given \(X_i=x\).
Conditionally on the whole process \(X\), the variables \(Z_1,\ldots,Z_N\) are independent and satisfy \(Z_i\sim K_i(\cdot\mid X_i)\), \(i\in[N]\).
Thus, \(K_{1:N}=(K_1,\ldots,K_N)\) acts coordinatewise: each \(Z_i\)
depends on the original sample only through the corresponding
observation \(X_i\), via \(K_i\). 
In the literature, such a mechanism \(K_{1:N}\) is called non-interactive.  
We next impose the local privacy constraint on each coordinate Markov kernel \(K_i\). 

\textit{\(\alpha\)-LDP Markov kernels.} 
For \(\alpha>0\), let \(\mathcal M_\alpha\) denote the class of all
Markov kernels \(K\),
from \(\mathbb R\) to an arbitrary standard Borel space,
satisfying
\begin{equation}
  \label{eq:alpha-ldp-cond}
K(S\mid x)
\le e^\alpha K(S\mid x')
 \quad\text{for all $x,x'\in\mathbb R$ and all measurable sets $S$ of the output space of $K$}.  
\end{equation}
A Markov kernel satisfying this inequality is called \(\alpha\)-locally differentially private (\(\alpha\)-LDP).

\textit{Privacy mechanisms.}
Returning to the observation model above, we assume that
\(K_{1:N}\in\mathcal M_\alpha^N\), where
\(\mathcal M_\alpha^N :=\mathcal M_\alpha\times\cdots\times\mathcal M_\alpha\), meaning that \(K_i\in\mathcal M_\alpha\) for each \(i\in[N]\). 
We refer to such a tuple \(K_{1:N}\) as an \(\alpha\)-LDP mechanism,
or simply as a privacy mechanism.
For a spectral density \(f\) and \(K_{1:N}\in\mathcal M_\alpha^N\), we write \(P_{f,K_{1:N}}\) for the joint law of \((X,Z)\), and
\(\mathbb E_{f,K_{1:N}}\) for the corresponding expectation.

%%%%%

\paragraph{Minimax risk.}

For a given privacy mechanism \(K_{1:N}\), let \(\mathcal E_{L^2}(K_{1:N})\) denote the set of all measurable functions from the output space of \(K_{1:N}\) to \(L^2([-\pi,\pi])\).
Throughout, \(\|\cdot\|\) denotes the norm of \(L^2([-\pi,\pi])\). 
The minimax risk under local differential privacy is defined by
\begin{equation}
\label{eq:minimax-risk}
\mathcal R_{N,\alpha}(C_0,C_1,s)
\ = \ 
\inf_{K_{1:N}\in\mathcal M_\alpha^N} \ 
\inf_{\hat f\in\mathcal E_{L^2}(K_{1:N})} \ 
\sup_{f\in\mathcal F(C_0,C_1,s)} \
\mathbb E_{f,K_{1:N}}
\left[
\|\hat f(Z)-f\|^2
\right],
\end{equation}
where the first infimum is over all  \(\alpha\)-LDP mechanisms on \(N\) coordinates, the second over measurable functions of the corresponding
privatized sample \(Z=(Z_1,\ldots,Z_N)\), and the supremum over the parameter class of
Sobolev-type spectral densities defined in~(\ref{defi-parameter-class-F}).

%%%%%%

\paragraph{Benchmarks and contributions.}

\textit{Classical spectral density estimation.}
In the absence of privacy constraints, spectral density estimation from
\(X_1,\ldots,X_N\) has a long history, in both
parametric~\cite{Dah89,Dav73,FT86,Tan87} and
nonparametric~\cite{Comte2001,Efr98,Gol93,Neu96,Sou00} settings.
Over Sobolev-type spectral density classes, the classical minimax rate
is \(N^{-2s/(2s+1)}\), attained by standard smoothed-periodogram
estimators; see, for instance,~\cite{Gol93,Comte2001}.

\smallskip 

\textit{Estimation under local differential privacy.}
A central problem in the literature reviewed above is nonparametric
\(L^2\)-recovery of the spectral density.
Kroll~\cite{Kroll2024} initiated the study of this problem under LDP
and established the first lower and upper bounds on its minimax rate,
leaving a substantial and puzzling gap between them, which we discuss
below.
A more recent study by Butucea et al.~\cite{butucea2025nonparametric} continues this line of work by considering closely related variants of the original problem, either by studying the same \(L^2\)-recovery problem under a relaxed LDP constraint, or by studying related estimation problems under LDP.
Therefore, the gap for \(L^2\)-recovery of the spectral density under
LDP remains open.

The main focus of the present paper is this open \(L^2\)-recovery problem.
To better understand it, we next present Kroll's
construction~\cite{Kroll2024}, which combines two standard ingredients.
On the statistical side, it starts from the classical periodogram-based
approach used in non-private spectral density estimation.
On the privacy side, it applies the usual clipping-and-Laplace strategy
for real-valued data: the observations \(X_i\) are first clipped at a
threshold \(\tau\) to make them bounded, and Laplace noise calibrated to
this threshold is then added.
With \(\tau^2\asymp(\log N)^{1+\delta}\), for any fixed \(\delta>0\), the resulting private periodogram is projected onto a finite-dimensional approximation space.
For Sobolev-type spectral density classes, this leads to an upper bound
of order
\[
\left(
\frac{1}{N}
\vee
\frac{(\log N)^{2+2\delta}}{N\alpha^4}
\right)^{2s/(2s+1)}
\vee
\frac{(\log N)^{2+2\delta}}{N},
\]
up to constants.
A more detailed comparison with our results is given in
Section~\ref{section-estimators-upper-bounds}.

\medskip 

\textit{Open questions.} 
This result raises two natural questions.
First, the dependence on \(\alpha^4\) is surprising.
In much of the local differential privacy literature for i.i.d. data,
the privacy constraint leads instead to an \(\alpha^2\)-type loss of
information~\cite{duchi2018minimax,RS20}.
This comparison, however, does not rule out a different privacy cost
for dependent observations.
Thus, in the present dependent time-series setting, it is unclear
whether the stronger \(\alpha^4\) dependence is intrinsic to the
problem, or merely an artifact of the standard approach analyzed
in~\cite{Kroll2024}.
The lower bound in~\cite{Kroll2024} does not resolve this issue: it
only exhibits an \(\alpha^2\)-type dependence.

Second, the upper bound contains polylogarithmic factors.
It is also unclear whether these logarithmic factors are unavoidable
under local privacy, or whether they come from the clipping step used
to privatize unbounded Gaussian observations in a generic way.

\medskip 

\textit{Contributions.}
The present paper resolves both questions.
On the lower-bound side, we develop a new proof strategy for locally
private dependent Gaussian time series, yielding a lower bound with the
sharp privacy dependence.
On the upper-bound side, we construct an estimator tailored to the
Gaussian structure, based on bounded transformations of the
observations rather than on the standard clipping-based periodogram
strategy.
Together, these lower and upper bounds characterize the minimax rate
under local privacy up to multiplicative constants,
closing the gap between the previous upper and lower bounds.

Along the way, we develop several tools that may be of independent
interest, some specific to Gaussian time series and others applicable
more broadly to lower bounds under LDP, such as reductions to simpler
laws and KL-contraction arguments.

Beyond their role in establishing our main results, we apply these
tools to three related questions.
First, for every fixed lag \(h\), we establish matching upper and lower
bounds for autocovariance estimation, thereby removing the
polylogarithmic loss in the previously known upper bound~\cite{butucea2025nonparametric}.
Second, through a local testing problem, we show that the
\(\alpha^4\) privacy cost persists around every spectral density bounded
away from zero and is therefore not caused by a pathological least
favorable spectral density.
Third, we show that the classical asymptotic equivalence between the
Gaussian time-series experiment and its independent Gaussian
counterpart does not generally extend to local differential privacy.

%%%%%%%%%%%%%

\paragraph{Further related literature.}

\textit{Closest works.}
Beyond the rate comparison above, Kroll's upper-bound result for a
prescribed projection space~\cite{Kroll2024} applies beyond the Gaussian setting, to stationary linear processes satisfying suitable
summability conditions and having sub-Gaussian marginals.
He further adapts the penalized model-selection approach of
\cite{Comte2001} to the locally private setting: the projection space
is selected through a penalized contrast criterion, and an oracle
inequality is established for the resulting adaptive estimator under
a Gaussian assumption.

While Kroll and the present paper focus primarily on global
\(L^2\)-recovery under local privacy, Butucea et
al.~\cite{butucea2025nonparametric} consider two related problems
for centered stationary Gaussian processes under LDP:
the estimation of individual autocovariances and of the spectral density at a fixed frequency. 
Their results on these problems are discussed and compared with consequences of our methods in Section~\ref{sec:related-problems}. 
They also study a relaxation of the LDP constraint, namely sequentially
interactive local privacy, introducing mechanisms for these two problems and for global \(L^2\)-recovery. 
For the resulting interactive procedures, they state \(\alpha^2\)-based rates for autocovariance and pointwise estimation, together with a different, slower rate for global recovery.
By contrast, we do not consider interactive local privacy in this paper.

\smallskip 

\textit{Local privacy and dependent observations.}
Beyond spectral density estimation, few statistical results are available for locally private inference with dependent observations. 
Locally private drift estimation from independently replicated diffusion trajectories is studied
in~\cite{AmorinoGloterHalconruy2025}, where consistency and asymptotic normality are established for a private pseudo-likelihood estimator under within-trajectory temporal dependence.
Mean estimation and regression under dependence characterized by log-Sobolev inequalities are considered
in~\cite{RothAvellaMedina2025}, which develops central, user-level, and locally private procedures for models including repeated measurements and longitudinal regression. 

Related algorithmic work considers transition-matrix estimation from multiple Markov sequences \cite{GunerGursoy2024}, private sharing of smart-meter and smart-home data~\cite{ShanmugarasaEtAl2024,WaheedEtAl2023}, longitudinal monitoring and trajectory collection~\cite{JosephEtAl2018,OhrimenkoWirthWu2021,ArcoleziEtAl2023,ZhangEtAl2023Trajectory}, and privacy notions adapted to correlated genomic data~\cite{YilmazEtAl2022}.
These settings concern multiple trajectories, finite-dimensional or release-oriented targets, or privacy notions modified to account for dependence.

\smallskip

\textit{Central and aggregate-release mechanisms.}
Outside the coordinatewise local model, a broader literature considers the private release of dependent or sequential signals, either through a trusted curator or through distributed aggregation. 
Transform-based methods perturb a truncated collection of Fourier or wavelet coefficients~\cite{RastogiNath2010,LyuEtAl2017,LeukamLakoEtAl2021}, while filtering-based methods construct private approximations to dynamic filters~\cite{LeNyPappas2014}. 
Prediction-based methods for user-level private real-time release of autoregressive aggregate-query sequences are developed in~\cite{ZhangKhaliliLiu2023}. 
These works target the release of a signal or query sequence itself, rather than the estimation of an unknown spectral density from coordinatewise locally privatized observations.

%%%%%

\paragraph{Organization of the paper.}

Section~\ref{section-estimators-upper-bounds} states the main results
of the paper: the matching lower and upper bounds for
spectral density estimation under LDP.
Section~\ref{sec:related-problems} applies our methods to three related
questions and discusses several open problems.
Section~\ref{section:element_proofs} presents the main technical tools used in the proofs, including several results that may be of independent
interest. The lower and upper bounds are proved in
Sections~\ref{subsection:proof-lower-bound-thm}
and~\ref{section:proof-sketch-upper-bound-thm}, respectively.
Technical proofs are deferred to the appendices for convenience.

%%% Local Variables:
%%% mode: latex
%%% TeX-master: "main.tex"
%%% ispell-local-dictionary: "american"
%%% End:

%%%%%%%%%%%%%%%%%%

\section{Minimax rates for spectral density estimation}  \label{section-estimators-upper-bounds}

Throughout the paper, we fix the parameter class
\(\mathcal F(C_0,C_1,s)\), with constants \(C_0\geq 1\), \(C_1>0\),
and \(s>1/2\).
This section studies the minimax difficulty of spectral density
estimation over this class under local differential privacy.
We first establish a lower bound that provides the benchmark for this
problem.
We then construct an estimator attaining this benchmark rate, and
examine which part of the estimation problem determines the privacy
cost.
We use the notation \(u\wedge v:=\min\{u,v\}\) and \(u\vee v:=\max\{u,v\}\).

%%%%%%%%%%
%Section 2.1
%%%%%%%%%%%%%

\subsection{The \(\alpha^4\) privacy cost: minimax lower bound}
\label{section_minimax_lower_bounds}

Our first result establishes a lower bound on the minimax risk introduced in Section~\ref{section-intro}.

\begin{theorem}
\label{thm-borne-inf:cas-non-interactif}
There exists a constant \(c_{C_1,s}>0\), depending only on
\((C_1,s)\), such that, for every \(N\geq 2\) and every \(\alpha>0\),
\begin{equation*}
\mathcal R_{N,\alpha}(C_0,C_1,s)
\ \geq\
c_{C_1,s}
\left(
1\wedge
\left[(\alpha^4\wedge 1)N\right]^{-\frac{2s}{2s+1}}
\right).
\end{equation*}
\end{theorem}

%%%%%%

A central message of this paper is that spectral density estimation incurs an \(\alpha^4\) privacy cost.
This contrasts sharply with the usual \(\alpha^2\) privacy dependence encountered in much of the local differential privacy literature, which has mainly been developed for problems with i.i.d. observations.
Spectral density estimation presents an additional challenge because the observations form a dependent time series.
Previous lower bounds for this problem had only exhibited an \(\alpha^2\)-type dependence~\cite{Kroll2024}, leaving the sharp privacy cost unresolved.
Theorem~\ref{thm-borne-inf:cas-non-interactif} resolves this question by providing the first lower bound exhibiting the \(\alpha^4\) dependence.

The lower bound exhibits three regimes.
When \(\alpha^4N\le1\), the minimax risk is bounded away from zero, so consistent estimation is impossible.
When \(\alpha^4N>1\) and in the strong privacy regime \(\alpha\le1\), the lower bound is of order \((\alpha^4N)^{-2s/(2s+1)}\).
Finally, in the weak privacy regime \(\alpha>1\), the lower bound reduces to the classical non-private rate \(N^{-2s/(2s+1)}\) for spectral density estimation.
In summary, local privacy reduces the effective sample size from \(N\) to
\(N\wedge\bigl(1\vee(\alpha^4N)\bigr)\).
The matching upper bound established in Theorem~\ref{thm:f-estimator}
shows that these rates are minimax optimal, up to constants, in all three regimes.

Another interesting question concerns the origin of the \(\alpha^4\) privacy cost.
Theorem~\ref{thm-borne-inf:cas-non-interactif} identifies the benchmark rate, but does not by itself explain which feature of the problem is responsible for this unusual privacy cost.
Does the \(\alpha^4\) dependence arise from estimating the unknown scale of the process, or from estimating its normalized dependence structure?
To address this question, Section~\ref{section-upper-bounds} studies these two estimation problems separately to determine which component drives the \(\alpha^4\) privacy cost.

%%%%%%%%

\paragraph{Proof strategy for Theorem~\ref{thm-borne-inf:cas-non-interactif}.}
The proof of the privacy-dependent term proceeds in three steps.
First, the nonparametric problem of estimating \(f\) is restricted to
a finite hypercube of Gaussian models indexed by the signs of finitely
many Fourier coefficients, equivalently by the signs of finitely many
autocovariance coefficients \(\gamma_f(h)\).  Second, Assouad's lemma
reduces the lower bound to controlling the privatized KL divergence
between two neighboring vertices of this hypercube, which differ only
in the sign of one coefficient.  The third and final step is devoted
to controlling this divergence.  This is delicate because the
privatized observations \(Z=(Z_1,\ldots,Z_N)\) inherit a nontrivial
dependence structure from the underlying Gaussian process
\(X=(X_t)_{t\in\mathbb Z}\).  This control is provided by
Theorem~\ref{thm:KL-bound-gaussian} below. We first introduce the
notation needed to state it.

\textit{Notation for KL divergence.}
For two probability measures \(P\) and \(Q\) on a common measurable
space, \(P\ll Q\) means that \(P\) is dominated by \(Q\).
The Kullback--Leibler divergence from \(P\) to \(Q\) is denoted by
\(\D(P\|Q)\) and defined by
\[
\D(P\|Q)
\ := \
\int \log\!\left(\frac{dP}{dQ}\right)\,dP,
\qquad\text{if }P\ll Q,
\]
and \(\D(P\|Q):=+\infty\) otherwise.
Throughout the paper, \(\log\) denotes the natural logarithm on
\((0,\infty)\), extended by \(\log(0)=-\infty\).

Recall from Section~\ref{section-intro} that \(P_{f,K_{1:N}}\) denotes
the joint law of \((X,Z)\), where \(Z=(Z_1,\ldots,Z_N)\) is the vector
of privatized observations.  We denote the marginal law of \(Z\) under
\(P_{f,K_{1:N}}\) by \(P_{f,K_{1:N}}^Z\).

\begin{theorem}
\label{thm:KL-bound-gaussian}
There exists a universal constant \(C\) such that, for all integers
$h,N\geq1$, all real numbers \(\alpha,a>0\) and all \(\delta\in[-1,1]\), if
\(f_+\) and \(f_-\) are two spectral densities satisfying, for every
\(\lambda\in[-\pi,\pi]\),
\begin{equation}
\label{eq:min-f_pm}
\min_{\pm\in\ac{-,+}}f_\pm(\lambda)
\ \geq \  \frac{3a}{4\pi},
\qquad \qquad 
f_+(\lambda)-f_-(\lambda)
\ = \
\frac{a\delta}{2\pi}\cos(h\lambda),
\end{equation}
then we have
\[
\sup_{K_{1:N}\in\mathcal M_\alpha^N}
\D\!\left(
P_{f_+,K_{1:N}}^Z
\,\middle\|\,
P_{f_-,K_{1:N}}^Z
\right)
\ \leq\
C(\alpha^4\wedge1)(N-h)_+\delta^2.
\]
\end{theorem}

Here, \(x_+:=\max\{x,0\}\).
The two conditions in~\eqref{eq:min-f_pm} ensure, respectively, that \(f_+\) and \(f_-\) are uniformly bounded from below, and that they differ only in their lag-\(h\) component in the Fourier representation~\eqref{eq:f-fourier-representation-even-cosine}.
To interpret the bound in Theorem~\ref{thm:KL-bound-gaussian}, consider the strong privacy regime \(\alpha\leq1\), in the case where \(h\) and \(a\) are fixed, while \(N\) is large.
In this regime, \(f_+\) and \(f_-\) are separated by order \(\delta^2\) in squared \(L^2\)-norm, and Theorem~\ref{thm:KL-bound-gaussian} shows that, uniformly over all
\(\alpha\)-LDP mechanisms, the KL divergence between the corresponding privatized laws is at most of order \(\alpha^4N\delta^2\).

It is this \(\alpha^4\) contraction of the privatized KL divergence
that yields the \(\alpha^4\) privacy cost in the minimax lower bound
of Theorem~\ref{thm-borne-inf:cas-non-interactif}. Thus, Theorem~\ref{thm:KL-bound-gaussian} forms the crux of our lower-bound argument. To the best of our
knowledge, this result is new. We next comment on the different
terms appearing in this KL bound, and then briefly outline the main
steps in the proof of Theorem~\ref{thm:KL-bound-gaussian}.

%%%%%
\smallskip 

\begin{remark}
Items~(1)--(2) cover the \((N,h)\)-dependence, and items~(3)--(4) the \(\alpha\)-dependence.
\begin{enumerate}[label=(\arabic*)]
\item
The factor \((N-h)_+\) vanishes when \(h\geq N\).
This is expected: when \(h\geq N\), the sample \(X_{1:N}\) contains no
pair of the form \((X_i,X_{i+h})\), so it has the same distribution
under \(f_+\) and \(f_-\).
Consequently, the corresponding privatized laws coincide and their KL
divergence is zero. 

\item 
For fixed \(h\) and large \(N\), the linear dependence of the bound on
\(N\) is natural.
Similar linear growth occurs in many stationary ergodic models.
The simplest exact example is the i.i.d.\ setting, where the KL
divergence between the joint laws is exactly \(N\) times the KL
divergence between the marginal distributions.

\item
When \(\alpha\geq1\), the factor \(\alpha^4\wedge1\) equals \(1\), and
the upper bound has the same order as the KL divergence for a direct
observation of \(X_{1:N}\).

As a side note, it would be interesting to identify the exact leading
constant of this KL divergence as \(N\to\infty\).
This question lies outside our analysis of privatized data, so we do
not pursue it.

\item
When \(\alpha\leq1\), the factor \(\alpha^4\wedge1\) reduces to
\(\alpha^4\).
To interpret this factor, consider first a standard i.i.d.\ setting in
which the two marginal distributions of \(X\) differ.
Under local differential privacy, the privatized
observations \(Z_1,\dots,Z_N\) remain independent, and a single
application of the privacy-contraction bound in
Lemma~\ref{lem:oriented-kl-contraction-ldp} produces the usual
\(\alpha^2\) factor.

In the present dependent setting, however, the one-dimensional
marginals under \(f_+\) and \(f_-\) are identical.
Consequently, each \(Z_i\) also has the same marginal distribution
under the two models, and the sum of the corresponding marginal KL
divergences is zero.
Therefore, this marginal calculation does not capture the joint KL
divergence: all the information separating the two models lies in
their dependence structure.
This suggests interpreting the factor \(\alpha^4\) as a contribution
of the dependence structure, while the absence of an \(\alpha^2\)
term may be viewed as a consequence of the equal marginal
distributions.
This will become clear in the proof.

\end{enumerate}
\end{remark}

\smallskip 

The proof of Theorem~\ref{thm:KL-bound-gaussian} combines three main
ingredients.
First, the privatized KL divergence is reduced to that of a privatized
\(1\)-dependent Gaussian process.
Although this simplifies the dependence structure, controlling the
effect of privatization on the KL divergence remains challenging.
Second, we reduce this divergence to a sum of conditional KL
divergences involving only the underlying process~\(X\).
This reduction is obtained by iterating a conditional KL-contraction
argument, with the \(\alpha^4\) factor arising from two successive
privacy contractions.
Finally, these conditional KL divergences are controlled using the
stationary Gaussian structure of~\(X\).

Section~\ref{section:element_proofs} presents the main tools developed
in this paper for these steps, some of which may be of independent
interest beyond the present problem.
The proofs of Theorems~\ref{thm-borne-inf:cas-non-interactif}
and~\ref{thm:KL-bound-gaussian} are given in
Section~\ref{subsection:proof-lower-bound-thm}.

%%%%%%%%%%%
%%%%%%%%%%%%%%%

%%%%%%%%%%
%Section 2.2
%%%%%%%%%%%%%

\subsection{Matching upper bound: construction of the estimator}
\label{section-upper-bounds}

Recall that \(X=(X_t)_{t\in\mathbb Z}\) is a centered real-valued stationary Gaussian process with unknown spectral density
\(f\in\mathcal F(C_0,C_1,s)\), where \(\mathcal F(C_0,C_1,s)\) is defined in~\eqref{defi-parameter-class-F}.
We now construct an estimator of \(f\) attaining the rate in the lower bound of Theorem~\ref{thm-borne-inf:cas-non-interactif}.
To identify the origin of the \(\alpha^4\) privacy cost, the construction mirrors the structure of the parameter class \(\mathcal F(C_0,C_1,s)\): we estimate separately the scale \(\gamma_f(0)\) and the normalized dependence structure.

%%%%%%%

\paragraph{Scale and reduced spectral density.}
Since \(X\) is real-valued, the autocovariance sequence \(\gamma_f\) is even, and the Fourier representation~\eqref{eq:f-fourier-representation} of \(f\) can be written as
\begin{equation}
\label{eq:f-fourier-representation-even-cosine}
f(\lambda)
=
\frac{1}{2\pi}
\left(
\gamma_f(0)+2\sum_{h=1}^{\infty}\gamma_f(h)\cos(h\lambda)
\right),
\qquad \lambda\in[-\pi,\pi].
\end{equation}
Factoring out the variance \(\gamma_f(0)=\operatorname{Var}(X_i)\) yields the decomposition
\begin{equation*}
%\label{eq:fred-factorization}
f
=
\gamma_f(0)f_{\mathrm{red}},
\end{equation*}
where the \textit{reduced spectral density} \(f_{\mathrm{red}}\) is
\begin{equation}
\label{eq:fred-fourier-representation}
f_{\mathrm{red}}(\lambda)
=
\frac{1}{2\pi}
\left(
1+2\sum_{h=1}^{\infty}\rho_f(h)\cos(h\lambda)
\right),
\qquad
\lambda\in[-\pi,\pi].
\end{equation}
Here \(\rho_f(h)=\gamma_f(h)/\gamma_f(0)\) is the \textit{lag-\(h\) autocorrelation coefficient} introduced in Section~\ref{section-intro}.

Thus, estimating \(f\) reduces to estimating the variance
\(\gamma_f(0)\) and the reduced spectral density \(f_{\mathrm{red}}\).
We analyze these two tasks separately and then combine the resulting
estimators to estimate \(f\). The \(\alpha^4\) privacy cost will be
inherited from the reduced spectral density estimation error term.

%%%%%%%%%%%%%%%

\paragraph{Privatization mechanism.}
For any \(\alpha>0\), define the privatized observations
\(Z_i=(Z_{i,1},Z_{i,2})\), \(i\in[N]\), by
\begin{equation}
\label{defi-K_lap}
Z_{i,1}
=
\mathbf 1_{\{|X_i|> C_0\}}+\frac{3}{\alpha}W_{i,1},
\qquad \qquad
Z_{i,2}
=
\operatorname{sgn}(X_i)+\frac{3}{\alpha}W_{i,2},
\end{equation}
where the \(W_{i,j}\)'s are i.i.d.\ centered standard Laplace random
variables, independent of \(X\).
We denote by \(K^{\mathrm{Lap},\alpha}\) the corresponding one-observation
privatization kernel, mapping \(X_i\) to \(Z_i=(Z_{i,1},Z_{i,2})\).
We write
\[
K_{1:N}^{\mathrm{Lap},\alpha}
=
\bigl(K^{\mathrm{Lap},\alpha},\ldots,K^{\mathrm{Lap},\alpha}\bigr)
\]
for the corresponding coordinatewise mechanism, obtained by applying
\(K^{\mathrm{Lap},\alpha}\) independently to \(X_1,\ldots,X_N\).
It is straightforward to check that
\(K^{\mathrm{Lap},\alpha}\in\mathcal M_\alpha\), and hence
\(K_{1:N}^{\mathrm{Lap},\alpha}\in\mathcal M_\alpha^N\).

%%%%
%%%

\subsubsection{The scale component: variance estimation}

We estimate the variance \(\gamma_f(0)\) from the privatized data \((Z_{i,1})_{i\in[N]}\) defined in~\eqref{defi-K_lap}.
We first introduce the probability \(p_f:=\mathbb P(|X_i|>C_0)\).
Since \(X_i\sim\mathcal N(0,\gamma_f(0))\), this probability can be written as a function of the variance \(\gamma_f(0)\) as follows:
\begin{equation*}
p_f
=2\pa{1-
\Phi\!\left(\frac{C_0}{\sqrt{\gamma_f(0)}}\right)},
\end{equation*}
where \(\Phi\) denotes the standard Gaussian distribution function. Hence, \(p_f\) determines the variance \(\gamma_f(0)\) through the inverse relation
\begin{equation}
\label{eq:inverse-formula-motivatin-estimator}
\gamma_f(0)
=
\left(\frac{C_0}{\Phi^{-1}(1-p_f/2)}\right)^2.
\end{equation}

Since the variables \(Z_{i,1}\) are unbiased for \(p_f\), namely \(\mathbb E[Z_{i,1}]=\mathbb P(|X_i|>C_0)=p_f\), a natural estimator of \(p_f\) is the empirical mean \(\frac1N\sum_{i=1}^N Z_{i,1}\).
However, since \(\gamma_f(0)\in[C_0^{-1},C_0]\), the probability parameter \(p_f\) belongs to
\(I_{C_0}:=\cro{2\pa{1-\Phi(C_0^{3/2})},2\pa{1-\Phi(C_0^{1/2})}}\); we therefore set
\begin{equation}
\label{clipping}
\hat p_f=\Pi_{I_{C_0}}\pa{\frac1N\sum_{i=1}^N Z_{i,1}},
\end{equation}
where \(\Pi_{[a,b]}(x)=(x\wedge b)\vee a\) denotes the projection onto \([a,b]\).
Replacing \(p_f\) by \(\hat p_f\) in the inverse formula~\eqref{eq:inverse-formula-motivatin-estimator} yields the \textit{variance estimator}
\begin{equation}
\label{defi-variance-estimator}
\hat\gamma(0)
=
\left(
\frac{C_0}{\Phi^{-1}(1-\hat p_f/2)}
\right)^2.
\end{equation}

\begin{remark}
The projection step~\eqref{clipping} ensures that \(\hat p_f \in I_{C_0}\).
On this interval, the map \(u\mapsto \bigl(C_0/\Phi^{-1}(1-u/2)\bigr)^2\) is Lipschitz.
Consequently, the estimation error of \(\hat\gamma(0)\) is controlled by that of \(\hat p_f\).
\end{remark}

%%%%%%%%%%
\medskip 
%%%%%%%%%%

The next proposition provides an upper bound on the quadratic risk of the estimator \(\hat\gamma(0)\).

\begin{proposition}
\label{lem:variance-estimator}
For every \(N\ge1\) and every \(\alpha>0\), under the privatization mechanism \(K_{1:N}^{\mathrm{Lap},\alpha}\) in~\eqref{defi-K_lap}, the estimator \(\hat\gamma(0)\) satisfies
\begin{equation*}
\sup_{f\in\mathcal F(C_0,C_1,s)}
\mathbb E_{f,K_{1:N}^{\mathrm{Lap},\alpha}}
\left[
|\hat\gamma(0)-\gamma_f(0)|^2
\right]
\ \le\
C_{C_0,C_1,s}\,\pa{
1\wedge\cro{\pa{\alpha^2\wedge1}N}^{-1}},
\end{equation*}
where \(C_{C_0,C_1,s}\) depends only on \((C_0,C_1,s)\).
\end{proposition}

The bound above shows that local privacy reduces the effective sample size
for estimating the variance \(\gamma_f(0)\) from \(N\) to
\(1\vee\cro{(\alpha^2\wedge1)N}\).
Thus, the scale component can be estimated with an \(\alpha^2\) rather than
an \(\alpha^4\) privacy cost.
Therefore, it does not by itself account for the \(\alpha^4\) privacy cost
in Theorem~\ref{thm-borne-inf:cas-non-interactif}.

\begin{remark}
When \(C_0>1\), the rate in Proposition~\ref{lem:variance-estimator} is also minimax optimal up to constants. 
Unlike the \(\alpha^4\) lower bound of Theorem~\ref{thm-borne-inf:cas-non-interactif}, the corresponding lower bound follows from standard arguments for i.i.d.\ locally private models.
\end{remark}

%%%%%%%%%%%%%%%%%%

%%%%%%%%%%%%%%%%%

\subsubsection{The dependence component: reduced spectral density estimation}

We estimate the dependence component in two standard steps.
First, we estimate the autocorrelation coefficients \(\rho_f(h)\) up
to a suitably chosen cutoff lag \(H\).
Then, we plug these estimates into the Fourier
representation~\eqref{eq:fred-fourier-representation} of
\(f_{\mathrm{red}}\), setting the coefficients at lags larger than
\(H\) to zero.
The less standard ingredient is the use of the privatized signs
\((Z_{i,2})_{i\in[N]}\) to estimate the autocorrelation coefficients.
This construction is based on the following elementary result, whose
proof is omitted.

\begin{lemma}
\label{lem:sign-gaussian}
Let \(X=(X_t)_{t\in\mathbb Z}\) be a centered stationary Gaussian
process with positive variance and autocorrelation function \(\rho\).
Then \(\left(\operatorname{sgn}(X_t)\right)_{t\in\mathbb Z}\) is a centered
stationary process with unit variance and autocorrelation function
\(h\mapsto \frac{2}{\pi}\arcsin(\rho(h))\).
\end{lemma}

Since the Laplace noises are centered i.i.d.\ random variables independent of
\(X\), the covariance function of
\(Z^{(2)}:=(Z_{i,2})_{i\in[N]}\) coincides, at nonzero lags, with that of
\((\operatorname{sgn}(X_t))_{t\in\mathbb Z}\).
Lemma~\ref{lem:sign-gaussian} therefore suggests applying the inverse
map \(u\mapsto\sin(\pi u/2)\) to the empirical covariance of
\(Z^{(2)}\).
Accordingly, we define
\begin{equation}
\label{eq:def-rho-hat}
\hat\rho(h)
\ := \
\sin\!\left(
\frac{\pi}{2N}
\sum_{j=1}^{N-h} Z_{j,2}Z_{j+h,2}
\right),
\qquad h\in[N-1].
\end{equation}
Then, for a cutoff lag \(H\in\{0,\ldots,N-1\}\), the
\textit{reduced spectral density estimator} is defined by
\begin{equation}
\label{eq:def-fred-hat}
\hat f_{\mathrm{red},H}(\lambda)
\ := \
\frac{1}{2\pi}
\left(
1+2\sum_{h=1}^{H}\hat\rho(h)\cos(h\lambda)
\right),
\qquad \lambda\in[-\pi,\pi],
\end{equation}
with the convention \(\sum_{h=1}^{H}\hat\rho(h)\cos(h\lambda)=0\) if \(H=0\).
For \(N\geq2\), we further set
\[
H^*
\ := \
\left\lfloor
\bigl((1\wedge\alpha^4)N\bigr)^{1/(2s+1)}
\right\rfloor.
\]
This choice balances two errors: the estimation error for the first \(H^*\) autocorrelation coefficients, and the truncation bias due to the remaining unestimated coefficients, namely those at lags \(h>H^*\).

%%%%%%%%%%
\medskip 
%%%%%%%%%%

The next proposition gives the corresponding \(L^2\)-risk bound for \(\hat f_{\mathrm{red},H^*}\).

\begin{proposition}
\label{prop:fred-estimator}
For every \(N\ge2\) and every \(\alpha>0\), under the privatization mechanism
\(K_{1:N}^{\mathrm{Lap},\alpha}\) in~\eqref{defi-K_lap}, the estimator
\(\hat f_{\mathrm{red},H^*}\) satisfies
\begin{equation*}
\sup_{f\in\mathcal F(C_0,C_1,s)}
\mathbb E_{f,K_{1:N}^{\mathrm{Lap},\alpha}}
\left[
\|\hat f_{\mathrm{red},H^*}-f_{\mathrm{red}}\|^2
\right]
\ \le\
C_{C_1,s}\,
\pa{
1\wedge\cro{\pa{\alpha^4\wedge1}N}^{-\frac{2s}{2s+1}}},
\end{equation*}
where \(C_{C_1,s}\) depends only on \(C_1\) and \(s\).
\end{proposition}

Proposition~\ref{prop:fred-estimator} shows that estimation of the reduced spectral density incurs an \(\alpha^4\) privacy cost.
By contrast, Proposition~\ref{lem:variance-estimator} shows that estimation of the variance incurs only an \(\alpha^2\) privacy cost.
Therefore, the normalized dependence structure \(f_{\mathrm{red}}\), rather than the unknown scale \(\gamma_f(0)\), is the bottleneck in our estimation strategy for \(f\) and the source of
the \(\alpha^4\) privacy cost.

This bottleneck is not merely an artifact of our estimation strategy.
A closer look at the proof of Theorem~\ref{thm-borne-inf:cas-non-interactif} shows that the lower bound is established on the unit-variance subclass \(\mathcal F(1,C_1,s)\), where \(\gamma_f(0)=1\) and hence \(f=f_{\mathrm{red}}\).
Therefore, the same lower bound applies to the estimation of \(f_{\mathrm{red}}\).
Together with Proposition~\ref{prop:fred-estimator}, this establishes the minimax rate for estimating the normalized dependence structure \(f_{\mathrm{red}}\).
In particular, it shows that the \(\alpha^4\) privacy cost is intrinsic to this component of the problem.

%%%%%%%%%%%%%%

%%%%%%%%%%%%%

\subsubsection{Combining the scale and dependence components: spectral density estimation}

We now combine the variance estimator \(\hat\gamma(0)\) with the reduced
spectral density estimator \(\hat f_{\mathrm{red},H^*}\).
We define the \textit{spectral density estimator} by
\begin{equation}
\label{eq:def-f-hat}
\hat f_{H^*}(\lambda)
\ := \
\hat\gamma(0)\,\hat f_{\mathrm{red},H^*}(\lambda),
\qquad \quad  \lambda\in[-\pi,\pi],
\end{equation}
where \(\hat\gamma(0)\) is defined in~\eqref{defi-variance-estimator} and
\(\hat f_{\mathrm{red},H^*}\) is defined in~\eqref{eq:def-fred-hat}.

Combining the preceding risk bounds for the scale and dependence components yields the following upper bound for spectral density estimation.

\begin{theorem}
\label{thm:f-estimator}
For every \(N\ge2\) and every \(\alpha>0\), under the privatization mechanism
\(K_{1:N}^{\mathrm{Lap},\alpha}\) in~\eqref{defi-K_lap}, the estimator
\(\hat f_{H^*}\) satisfies
\begin{equation*}
\sup_{f\in\mathcal F(C_0,C_1,s)}
\mathbb E_{f,K_{1:N}^{\mathrm{Lap},\alpha}}
\left[
\|\hat f_{H^*}-f\|^2
\right]
\ \le\
C_{C_0,C_1,s}\,
\pa{
1\wedge\cro{\pa{\alpha^4\wedge1}N}^{-\frac{2s}{2s+1}}},
\end{equation*}
where \(C_{C_0,C_1,s}\) depends only on \((C_0,C_1,s)\).
\end{theorem}

Together with Theorem~\ref{thm-borne-inf:cas-non-interactif}, Theorem~\ref{thm:f-estimator} establishes that \(\hat f_{H^*}\) is minimax optimal over \(\mathcal F(C_0,C_1,s)\), up to multiplicative constants.
The detailed proofs of Propositions~\ref{lem:variance-estimator} and~\ref{prop:fred-estimator}, and of Theorem~\ref{thm:f-estimator}, are provided in Section~\ref{section:proof-sketch-upper-bound-thm}.

%%%%%%%
%%%%%%%%

\paragraph{Comparison with previous upper bounds.}

To the best of our knowledge, the only previous upper bound for
\(L^2\)-recovery of the spectral density under LDP is due to
Kroll~\cite{Kroll2024}.
A more recent work by Butucea et
al.~\cite{butucea2025nonparametric} also studies centered stationary
Gaussian processes, but its results concern either related estimation problems under the same
LDP model, or the same \(L^2\)-recovery problem under a
relaxation of the LDP model; see
Section~\ref{sec:related-problems}.

Kroll's upper bound applies to Sobolev classes.
His estimator achieves the minimax rate up to extra polylogarithmic
factors, including \(\log^{2+2\delta}(N)\), for some \(\delta>0\).
Theorem~\ref{thm:f-estimator} removes these extra factors and therefore
attains the minimax risk up to constants.

Beyond the rate improvement, the construction is substantially different.
The approach of \cite{Kroll2024} relies on a privatized version of the classical periodogram: the raw observations \(X_i\) are first clipped at a deterministic threshold \(\tau\), and the resulting bounded variables are then privatized with Laplace noise calibrated to this threshold.
This generic clipping reduction is common in local privacy, as it allows one to work, with high probability, in a setting close to the classical non-private one, so that much of the standard spectral analysis can be adapted.
However, the threshold \(\tau\) has to grow with \(N\) in order to control the clipping bias for Gaussian observations, which introduces the additional polylogarithmic factor in the final rate.

By contrast, our estimator avoids this clipping-calibration tradeoff.
It uses bounded transformations of the Gaussian observations, namely threshold and sign transformations, that retain the information needed to estimate the scale and dependence components, while remaining directly compatible with local privacy.
In particular, the sign transformation preserves the autocorrelation signal through the Gaussian arcsine identity, which allows us to reconstruct the reduced spectral density without introducing a growing clipping threshold.

This advantage comes with new technical difficulties. Our estimators
are based on non-linear functionals of dependent Gaussian sequences,
so the usual analysis for linear processes no longer applies.
We handle this by bringing into the privacy analysis covariance
inequalities for Gaussian functionals, in particular the results of
Gebelein~\cite{Gebelein1941} and
Kolmogorov--Rozanov~\cite{KolmogorovRozanov1960}, which are key to
proving Propositions~\ref{lem:variance-estimator}
and~\ref{prop:fred-estimator}. For the latter, the main argument is
summarized by a general covariance bound for bounded pairwise
transformations of a Gaussian process; see
Proposition~\ref{thm:cov-bound-signs}.

In conclusion, our results show that, under LDP, minimax analysis can
go beyond the generic strategy of clipping the observations and then
adapting the classical non-private analysis.
Problem-specific transformations, applied before privatization, can
yield sharper rates, at the cost of a problem-specific theoretical
analysis.
%In our setting, this approach avoids the polylogarithmic losses caused by generic clipping-based methods.

%Such tools are not standard in the local privacy literature, and their use here is essential to control the dependence between non-linear transformations of the Gaussian process.  
%More broadly, the comparison illustrates the benefit of using transformations tailored to the statistical structure of the problem before privatization.  Combined with covariance tools adapted to Gaussian dependence,

%%% Local Variables:
%%% mode: latex
%%% TeX-master: "main.tex"
%%% ispell-local-dictionary: "american"
%%% End:

%%%%%%%%%%%%%

\section{Related problems and open questions}  \label{sec:related-problems}

In this section, we examine several problems related to our main results.
We first consider lag-\(h\) autocovariance estimation, for which the
tools developed above yield sharp bounds. 
We then investigate the local nature of the \(\alpha^4\) privacy cost through a testing problem. 
We next examine the interplay between asymptotic equivalence and local
differential privacy. 
Finally, we discuss several questions left open by the present work.

%%%%%%%%%%%

\subsection{Lag-\(h\) autocovariance estimation}

Let \(f\in\mathcal F(C_0,C_1,s)\) be the unknown spectral density of
the process \(X\), and fix a lag \(h\in[N-1]\).
We consider the estimation of the autocovariance coefficient \(\gamma_f(h)\),
a problem that was recently studied in~\cite{butucea2025nonparametric}. 
On the Sobolev class \(\mathcal F(C_0,C_1,s)\), for
\(h\leq\sqrt N\) and \(N\) sufficiently large, Proposition~1
of~\cite{butucea2025nonparametric} gives an upper bound on the
quadratic risk of order
\begin{equation}\label{literature-bound}
\log^{2+2\delta}(N)
\left(
\frac{1}{N} + \frac{1}{N\alpha^4}
\right),
\end{equation}
for some \(\delta>0\).
We show that the estimators developed in Section~\ref{section-upper-bounds} yield the corresponding bound without a polylogarithmic loss. 
Moreover, the bound holds for every \(N\geq2\) and
every \(h\in[N-1]\), and naturally remains bounded as
\(N\alpha^4\to0\).

Recall the estimator \(\hat\rho(h)\) of the lag-\(h\) autocorrelation
coefficient \(\rho_f(h)\), defined in~\eqref{eq:def-rho-hat}.
Since \(\gamma_f(h)=\gamma_f(0)\rho_f(h)\), a natural estimator of the
autocovariance coefficient \(\gamma_f(h)\) is
\(\hat\gamma(0)\hat\rho(h)\), where the variance estimator
\(\hat\gamma(0)\) is defined in~\eqref{defi-variance-estimator}.
However, here we use a debiased version: we first debias the empirical
covariance before applying the map \(u\mapsto\sin(\pi u/2)\), and then
multiply by \(\hat\gamma(0)\). 
This yields
\begin{equation}
\label{eq:def-gamma-tilde}
\widetilde\gamma(h)
\ := \
\hat\gamma(0)\,
\sin\!\left(
\frac{\pi}{2(N-h)}
\sum_{j=1}^{N-h} Z_{j,2}Z_{j+h,2}
\right).
\end{equation}
The following result provides a risk bound for \(\widetilde\gamma(h)\).

\begin{proposition}
\label{prop:lag-h-autocovariance-upper-bound}
For every \(N\geq2\), every \(\alpha>0\), and every \(h\in[N-1]\),
under the privatization mechanism
\(K_{1:N}^{\mathrm{Lap},\alpha}\) in~\eqref{defi-K_lap}, the estimator
\(\widetilde\gamma(h)\) satisfies
\begin{equation*}
\sup_{f\in\mathcal F(C_0,C_1,s)}
\mathbb E_{f,K_{1:N}^{\mathrm{Lap},\alpha}}
\left[
|\widetilde\gamma(h)-\gamma_f(h)|^2
\right]
\ \leq \
C_{C_0,C_1,s}\,
\pa{
1\wedge
\cro{\pa{\alpha^4\wedge1}(N-h)}^{-1}
},
\end{equation*}
where \(C_{C_0,C_1,s}\) depends only on \((C_0,C_1,s)\).
\end{proposition}

For \(\alpha\leq1\) and \(h\leq cN\), where \(c\in(0,1)\) is fixed,
Proposition~\ref{prop:lag-h-autocovariance-upper-bound} gives the rate
\(1\wedge(N\alpha^4)^{-1}\) for lag-\(h\) autocovariance estimation,
thus improving on the previously known bound~\eqref{literature-bound}.
Proposition~\ref{prop:lag-h-autocovariance-upper-bound} follows by
combining the variance bound in
Proposition~\ref{lem:variance-estimator} with an adaptation of the
lag-\(h\) correlation argument used in the proof of
Proposition~\ref{prop:fred-estimator}.
A proof is given in
Appendix~\ref{appendix:lag-h-autocovariance}.

%%%%%%
\medskip 
%%%%%%%

%\textit{Matching lower bound.}
The next result shows that the privacy cost \(1/(N\alpha^4)\) is unavoidable for lag-\(h\) autocovariance estimation.
For a privacy mechanism \(K_{1:N}\), let
\(\mathcal E_{\mathbb R}(K_{1:N})\) denote the set of all measurable
functions from the output space of \(K_{1:N}\) to \(\mathbb R\).

\begin{proposition}
\label{prop:lag-h-autocovariance-lower-bound}
For every \(h\geq1\), there exists a constant
\(c_{C_1,s,h}>0\), depending only on \((C_1,s,h)\), such that, for
every \(N\geq h+1\) and every \(\alpha>0\),
\begin{equation*}
\inf_{K_{1:N}\in\mathcal M_\alpha^N} \
\inf_{\hat g\in\mathcal E_{\mathbb R}(K_{1:N})} \ 
\sup_{f\in\mathcal F(C_0,C_1,s)}
\mathbb E_{f,K_{1:N}}
\left[
|\hat g(Z)-\gamma_f(h)|^2
\right]
\ \geq \
c_{C_1,s,h}\,
\pa{
1\wedge
\cro{\pa{\alpha^4\wedge1}(N-h)}^{-1}
},
\end{equation*}
where the first infimum is over all \(\alpha\)-LDP mechanisms on \(N\) coordinates, and
the second over all real-valued measurable functions of the corresponding privatized sample \(Z=(Z_1,\ldots,Z_N)\).
\end{proposition}

For every fixed \(h\geq1\),
Propositions~\ref{prop:lag-h-autocovariance-upper-bound}
and~\ref{prop:lag-h-autocovariance-lower-bound} identify the minimax
rate for lag-\(h\) autocovariance estimation on
\(\mathcal F(C_0,C_1,s)\), up to constants.
The proof of the lower bound, given in
Appendix~\ref{appendix:lag-h-autocovariance}, follows from the local
lower bound stated in
Proposition~\ref{prop:local-lag-h-autocovariance-lower-bound-test}
below.

A lower bound exhibiting a \(1/(N\alpha^4)\) privacy term for lag-\(h\) autocovariance estimation was also obtained in~\cite{butucea2025nonparametric}.
Their approach is based on Fisher-information calculations, whereas our proof relies on the KL-contraction method developed in Sections~\ref{section_minimax_lower_bounds} and~\ref{section:element_proofs}.
Our approach provides a dependent-data counterpart to the KL-contraction methodology that, in the i.i.d.\ setting, has proved highly successful~\cite{duchi2018minimax}.
The flexibility of this method is illustrated by the following stronger result, which gives a local version of the above lower bound.

%%%%%%%

%%%%%%%%%%

\subsection{Local nature of the \(\alpha^4\) privacy cost}
\label{section:local-nature-alpha4-testing}

Given \(N\geq2\), \(\alpha>0\), a lag \(h\in[N-1]\), and a spectral
density \(f\), we ask for the smallest perturbation size
\(\varepsilon\) for which an \(\alpha\)-LDP test can distinguish
between the following two alternatives 
\begin{equation}
\label{eq:local-lag-h-epsilon-sp-density-choice}
f_\pm(\lambda)
\ := \
f(\lambda)
\pm
\frac{\varepsilon}{\pi}\cos(h\lambda),
\qquad
\lambda\in[-\pi,\pi].
\end{equation}
The two functions \(f_\pm\) are thus obtained by perturbing only the lag-\(h\)
autocovariance coefficient of \(f\), that is,
\(\gamma_{f_\pm}(h)=\gamma_f(h)\pm\varepsilon\), while the coefficients
at all other nonnegative lags remain unchanged.

Formally, for a privacy mechanism \(K_{1:N}\), let
\(\mathcal E_{[0,1]}(K_{1:N})\) denote the set of all measurable
functions from the output space of \(K_{1:N}\) to \([0,1]\).
A randomized test with test function \(\phi\) rejects the null
hypothesis \(f=f_-\) in favor of the alternative \(f=f_+\),
conditionally on \(Z\), with probability \(\phi(Z)\).
That is, its decision is obtained by
drawing a Bernoulli random variable with parameter \(\phi(Z)\), where
the outcomes \(1\) and \(0\) correspond to choosing \(f_+\) and \(f_-\),
respectively.
The maximal risk associated with \(\phi\in\mathcal E_{[0,1]}(K_{1:N})\) is then
\begin{equation}
\label{eq:statistical_max_risk}
\mathcal T_{N,\alpha}
\bigl(\phi,f,\varepsilon,K_{1:N}\bigr)
\ := \
\max\pa{
\mathbb E_{f_-,K_{1:N}}\cro{\phi(Z)},
\mathbb E_{f_+,K_{1:N}}\cro{1-\phi(Z)}
}\;.
\end{equation}
The maximal risk \(1/2\) is trivial: it is attained by guessing completely at random, corresponding to the constant test function \(\phi\equiv1/2\). 
For a nontrivial risk level \(r\in(0,1/2)\), the problem is to determine the smallest order of \(\varepsilon\) for which
there exist \(K_{1:N}\in\mathcal M_\alpha^N\) and
\(\phi\in\mathcal E_{[0,1]}(K_{1:N})\) satisfying
\(\mathcal T_{N,\alpha}(\phi,f,\varepsilon,K_{1:N})\leq r\).

A natural test, based on the estimator
\(\widetilde\gamma(h)\) and the
privatization mechanism \(K_{1:N}^{\mathrm{Lap},\alpha}\) from~\eqref{eq:def-gamma-tilde} and~\eqref{defi-K_lap}, respectively, is
\begin{equation}
\label{eq:estimator_to_test}
\phi_f(Z)
:=
\begin{cases}
1 & \text{if \(\widetilde\gamma(h)\geq\gamma_f(h)\),}\\
0 & \text{otherwise.}
\end{cases}
\end{equation}
Here, \(\widetilde\gamma(h)\) is computed from the privatized sample
\(Z\). 
For any \(r\in(0,1/2)\), set
\begin{equation}
\label{eq:local-testing-separation-rate}
\varepsilon_{\mathrm{up}}
\ := \
C_{C_0,C_1,s,r}
\left(
1\wedge
\left[(\alpha^4\wedge1)(N-h)\right]^{-1/2}
\right),
\end{equation}
where \(C_{C_0,C_1,s,r}>0\) depends only on \((C_0,C_1,s,r)\).
Choosing \(C_{C_0,C_1,s,r}\) sufficiently large and using
Proposition~\ref{prop:lag-h-autocovariance-upper-bound}, we prove that
\begin{equation}
\label{eq:local-testing-upper-bound}
\mathcal T_{N,\alpha}
\bigl(
\phi_f,f,\varepsilon_{\mathrm{up}},
K_{1:N}^{\mathrm{Lap},\alpha}
\bigr)
\ \leq \
r,
\end{equation}
provided that \(f\in\mathcal F(C_0,C_1,s)\) is bounded away from zero
and satisfies the Sobolev constraint strictly, and that the value of
\(\varepsilon_{\mathrm{up}}\)
in~\eqref{eq:local-testing-separation-rate} is sufficiently small.
A detailed proof of~\eqref{eq:local-testing-upper-bound} is
given in Appendix~\ref{appendix:local-testing}.

%%%%%%
\medskip
%%%%%%

The following proposition shows that the separation rate
in~\eqref{eq:local-testing-separation-rate} cannot be improved.

\begin{proposition}
\label{prop:local-lag-h-autocovariance-lower-bound-test}
Let \(r\in(0,1/2)\), \(h\geq1\), \(f_{\min}>0\), and
\(f\) be a spectral density satisfying \(f\geq f_{\min}\).
Then \(f_+\) and \(f_-\) defined
by~\eqref{eq:local-lag-h-epsilon-sp-density-choice} are spectral
densities for every \(\varepsilon\in(0,\pi f_{\min}]\),
and there exists a constant
\(c_{r,f_{\min}}\in(0,\pi f_{\min}]\), depending only on
\((r,f_{\min})\), such that, for every \(N\geq h+1\) and every
\(\alpha>0\), 
we have
\[
\inf_{K_{1:N}\in\mathcal M_\alpha^N} \ 
\inf_{\phi\in\mathcal E_{[0,1]}(K_{1:N})} \ 
\mathcal T_{N,\alpha}
\bigl(\phi,f,\varepsilon_{\mathrm{low}},K_{1:N}\bigr)
\ \geq \
r,
\]
where 
\(\varepsilon_{\mathrm{low}}:=c_{r,f_{\min}}
\bigl(1\wedge\left[(\alpha^4\wedge1)(N-h)\right]^{-1/2}\bigr)\).
\end{proposition}

Proposition~\ref{prop:local-lag-h-autocovariance-lower-bound-test}
establishes a local two-point lower bound around a reference spectral
density \(f\), rather than a global worst-case lower bound over a
large nonparametric class. It therefore shows that the \(\alpha^4\)
privacy cost is not caused by a pathological least favorable spectral
density, but persists locally around every spectral density bounded
away from zero. 
Its proof relies on Theorem~\ref{thm:KL-bound-gaussian} and is given
in Appendix~\ref{appendix:local-testing}.

%%%%%%%%%%%%%%%%%%%%%

%%%%%%%%%%%%%%%%%%%

\subsection{Asymptotic equivalence and local differential privacy}
\label{sect3.3:asympto-equiv}

In the non-private setting, \cite[Theorem~1.1]{golubev10} shows that
spectral density estimation for a stationary Gaussian time series is
asymptotically equivalent to an experiment with independent Gaussian
observations. Under local differential privacy, however, this
equivalence appears to fail. Indeed, recall that the \(\alpha^4\)
privacy cost in our problem originates from the dependence between the
observations rather than from their marginal distributions. This
suggests that an \(\alpha\)-LDP procedure for the independent model
cannot, in general, be transferred to the time series model with
asymptotically equivalent risks. The purpose of this subsection is to
formalize and establish this incompatibility.

We begin by describing the two experiments more precisely and recalling
the classical notion of asymptotic equivalence.
In the first experiment, one observes
\(X_{1:N}=(X_1,\ldots,X_N)\), consisting of \(N\) consecutive values
of a centered stationary Gaussian process with spectral density \(f\); its distribution is denoted by \(P_f^N\). In the second, one observes
an \(N\)-dimensional centered Gaussian vector
\(Y_{1:N}=(Y_1,\ldots,Y_N)\) with independent components, whose
distribution \(Q_f^N\) is characterized by
\[
Y_j\sim\mathcal N\bigl(0,J_{j,N}(f)\bigr),
\qquad \quad
J_{j,N}(f)
=
\frac{N}{2\pi}
\int_{2(j-1)\pi/N}^{2j\pi/N}f(\lambda)\,d\lambda,
\qquad j\in[N],
\]
where \(f\) is extended to \(\mathbb R\) by \(2\pi\)-periodicity.
For any class \(\mathcal G\) of spectral densities, we refer to the
experiments
\(\mathcal P_N:=(P_f^N)_{f\in\mathcal G}\) and
\(\mathcal Q_N:=(Q_f^N)_{f\in\mathcal G}\) as the
\emph{time series model} and the \emph{independent model},
respectively.
When \(\mathcal G=\widetilde{\mathcal F}(M,s)\), where
\(\widetilde{\mathcal F}(M,s)\) is the Sobolev-type class defined
below, \cite[Theorem~1.1]{golubev10} shows that the sequences
\(\mathcal P=(\mathcal P_N)_{N\geq1}\) and
\(\mathcal Q=(\mathcal Q_N)_{N\geq1}\) are asymptotically equivalent in the following sense.

The Le Cam deficiency of \(\mathcal P_N\) with respect to
\(\mathcal Q_N\) is
\begin{equation*}
\mathcal D_{\mathrm{LC}}
\bigl(\mathcal P_N\,\big\|\,\mathcal Q_N\bigr)
\ := \ 
\inf_L\sup_{f\in\mathcal G}
\left\|L\circ P_f^N-Q_f^N\right\|_{\mathrm{TV}},
\end{equation*}
where the infimum is over all Markov kernels \(L\) from
\(\mathbb R^N\) to \(\mathbb R^N\).
The sequence \(\mathcal P\) is asymptotically more informative than
\(\mathcal Q\) if this deficiency tends to zero as \(N\to\infty\), and
the two sequences \(\mathcal P\) and \(\mathcal Q\) are asymptotically equivalent if the deficiencies
in both directions tend to zero, that is, if
\(\mathcal D_{\mathrm{LC}}(\mathcal P_N\|\mathcal Q_N)\to0\) and
\(\mathcal D_{\mathrm{LC}}(\mathcal Q_N\|\mathcal P_N)\to0\).
For bounded losses, asymptotic equivalence means that statistical
procedures can be transferred from one model to the other, with the
difference between the corresponding risks tending to zero uniformly
over \(f\in\mathcal G\).

To determine whether this transfer property remains valid under local
differential privacy, we adapt Le Cam's deficiency to account for the
constraint that observations can only be accessed through privacy
mechanisms. Let \((\alpha_N)_{N\geq1}\) be a sequence of positive
privacy parameters and suppose that, for each \(N\), observations from
\(\mathcal P_N\) and \(\mathcal Q_N\) can only be accessed through
\(\alpha_N\)-LDP mechanisms. We define 
\begin{equation}
  \label{eq:le-cam-deficiency-ldp}
\mathcal D_{\mathrm{LC},\alpha_N}
\bigl(\mathcal P_N\,\big\|\,\mathcal Q_N\bigr)
\ := \ 
\sup_{K_{1:N}\in\mathcal M_{\alpha_N}^N}\
\inf_{\substack{K'_{1:N}\in\mathcal M_{\alpha_N}^N\\ L'}}\
\sup_{f\in\mathcal G}
\left\|
L'\circ K'_{1:N}\circ P_f^N
-
K_{1:N}\circ Q_f^N
\right\|_{\mathrm{TV}},
\end{equation}
where \(L'\) ranges over all Markov kernels from the output space of
\(K'_{1:N}\) to that of \(K_{1:N}\).
Thus, for each \(\alpha_N\)-LDP mechanism \(K_{1:N}\) applied to
\(\mathcal Q_N\), the infimum seeks an \(\alpha_N\)-LDP mechanism
\(K'_{1:N}\) applied to \(\mathcal P_N\) and a post-processing \(L'\)
such that \(L'\circ K'_{1:N}\circ P_f^N\) approximates
\(K_{1:N}\circ Q_f^N\), uniformly over \(f\in\mathcal G\).
Consequently, if this deficiency is smaller than \(\varepsilon\),
every statistical procedure based on \(\alpha_N\)-LDP data from
\(\mathcal Q_N\) has a counterpart based on \(\alpha_N\)-LDP data from
\(\mathcal P_N\) whose risk differs by less than \(\varepsilon\),
uniformly over \(f\in\mathcal G\), for any loss bounded by one.
As in the classical setting, we say that \(\mathcal P\) is
asymptotically more informative than \(\mathcal Q\) under LDP if
\(\mathcal D_{\mathrm{LC},\alpha_N}
(\mathcal P_N\|\mathcal Q_N)\to0\), and that \(\mathcal P\) and \(\mathcal Q\)  are asymptotically equivalent under LDP if
\(\mathcal D_{\mathrm{LC},\alpha_N}
(\mathcal P_N\|\mathcal Q_N)\to0\) and
\(\mathcal D_{\mathrm{LC},\alpha_N}
(\mathcal Q_N\|\mathcal P_N)\to0\).

To state our result precisely, we first introduce the
Sobolev-type class of spectral densities considered in
\cite[Theorem~1.1]{golubev10}. For \(M>0\) and \(s>1/2\), define
\[
\widetilde{\mathcal F}(M,s)
:=
\left\{
f\text{ spectral density}:
f\geq\frac1M,\quad
\gamma_f^2(0)
\left(
1+\sum_{h\in\mathbb Z\setminus\{0\}}
|h|^{2s}|\rho_f(h)|^2
\right)
\leq M
\right\}.
\]
Up to changes in the constants, this class is comparable to uniformly
lower-bounded subclasses of \(\mathcal F(C_0,C_1,s)\); precise
inclusions are given in Appendix \ref{appendix:asymptotic-equivalence-ldp}.
To exclude empty or degenerate parameter classes, we assume
\(M>(2\pi)^{2/3}\).

\begin{proposition}
\label{prop:golubev_etal-upm}
Let \(M>(2\pi)^{2/3}\) and \(s>1/2\). 
Set
\(\mathcal P_N:=(P_f^N)_{f\in\widetilde{\mathcal F}(M,s)}\) and
\(\mathcal Q_N:=(Q_f^N)_{f\in\widetilde{\mathcal F}(M,s)}\).
Let \((\alpha_N)_{N\geq1}\) be a sequence of positive privacy
parameters such that
\(\alpha_N\longrightarrow0\),
and
\(\alpha_N^2N\longrightarrow\infty\).
Then
\[
\liminf_{N\to\infty}
\mathcal D_{\mathrm{LC},\alpha_N}
\bigl(\mathcal P_N\,\big\|\,\mathcal Q_N\bigr)
\ \geq\ 
\frac12.
\]
\end{proposition}

Proposition~\ref{prop:golubev_etal-upm} shows that the classical
asymptotic equivalence between the two models does not extend to LDP.
More precisely, throughout the privacy regime considered above, the
time series model is not asymptotically more informative than the
independent model under LDP.
This result follows by comparing the difficulty of the local testing
problem between \(f_-\) and \(f_+\), introduced in
Section~\ref{section:local-nature-alpha4-testing}, in the two models
\(\mathcal P_N\) and \(\mathcal Q_N\).
The proof is given in Appendix~\ref{appendix:asymptotic-equivalence-ldp}.

\medskip 

The preceding negative result raises a natural complementary question:
under what conditions can a classical asymptotic equivalence remain valid under LDP? 
A simple positive case arises when the Markov kernel mapping
\(\mathcal P_N\) to \(\mathcal Q_N\) acts coordinatewise, that is,
\(L_{1:N}=(L_1,\ldots,L_N)\), where each \(L_i\) is a Markov kernel from \(\mathbb R\) to \(\mathbb R\).
Indeed, the following lemma bounds the LDP deficiency by the classical Le Cam
deficiency restricted to coordinatewise Markov kernels.
It is stated for an arbitrary parameter space \(\Theta\) (for spectral density models, simply take \(\theta=f\)).

\begin{lemma}
\label{lem:TV-compatible}
Let \(N\geq1\) and \(\alpha>0\), let \(\Theta\) be a parameter space,
and consider two statistical models
\(\mathcal P_N=(P_\theta^N)_{\theta\in\Theta}\) and
\(\mathcal Q_N=(Q_\theta^N)_{\theta\in\Theta}\) on
\(\mathbb R^N\). Then
\[
\mathcal D_{\mathrm{LC},\alpha}
\bigl(\mathcal P_N\,\big\|\,\mathcal Q_N\bigr)
\ \leq \
\inf_{L_{1:N}}\
\sup_{\theta\in\Theta}\
\left\|
L_{1:N}\circ P_\theta^N-Q_\theta^N
\right\|_{\mathrm{TV}},
\]
where the infimum is over all Markov kernels
\(L_{1:N}=(L_1,\ldots,L_N)\) acting coordinatewise, with each \(L_i\)
a Markov kernel from \(\mathbb R\) to \(\mathbb R\).
\end{lemma}

The proof is given in Appendix~\ref{appendix:asymptotic-equivalence-ldp}.

Lemma~\ref{lem:TV-compatible} shows that whenever a classical
asymptotic equivalence is achieved by coordinatewise Markov kernels,
it remains valid under LDP, for any sequence of privacy parameters.
The classical asymptotic equivalence of \cite{golubev10}, however,
cannot be achieved by coordinatewise Markov kernels, a fact that can
be proved directly by comparing the models \(\mathcal P_N\) and
\(\mathcal Q_N\), but that also follows immediately from
Proposition~\ref{prop:golubev_etal-upm} and
Lemma~\ref{lem:TV-compatible}.  Thus, coordinatewise Markov kernels
form a class within which classical asymptotic equivalence is
preserved under LDP, but this class is too restrictive to include the
transformations underlying the classical equivalence of
\cite{golubev10}.  This raises the question of whether classical
asymptotic equivalence can remain compatible with LDP beyond the
coordinatewise setting.  We return to related questions in the next
subsection.

\subsection{Open questions}
\label{subsection:open-pbms}

We conclude this section by discussing several questions left open by the present work.

\paragraph{Pointwise spectral density estimation.}

Consider pointwise estimation of the spectral density at a fixed frequency \(\omega_0\in[-\pi,\pi]\), where the goal is to estimate the scalar quantity \(f(\omega_0)\) from the privatized observations. 
This problem was considered in~\cite{butucea2025nonparametric}. 
On the Sobolev class \(\mathcal F(C_0,C_1,s)\), the result of~\cite{butucea2025nonparametric} gives an upper bound of order
\begin{equation*}
\left(
\frac1N
+
\frac{\log^{2+2\delta}(N)}{N\alpha^4}
\right)^{\frac{2s-1}{2s}}
+
\frac{\log^{2+2\delta}(N)}{N}.
\end{equation*}

The construction developed in Section~\ref{section-upper-bounds}
suggests that this polylogarithmic loss should not be intrinsic. 
Indeed, one can consider the pointwise plug-in estimator
\(\hat f_H(\omega_0)
=
\hat\gamma(0)\,\,\hat f_{\mathrm{red},H}(\omega_0)\),
where \(\hat\gamma(0)\) is defined in~\eqref{defi-variance-estimator} and \(\hat f_{\mathrm{red},H}\) is defined in~\eqref{eq:def-fred-hat}. 
The pointwise bias-variance tradeoff suggests the truncation level
\(H_p
:=
\left\lfloor ((1\wedge \alpha^4)N)^{1/(2s)} \right\rfloor\),
which would lead to the conjectural pointwise rate
\begin{equation*}
\left[
\frac1N
+
\left(
1\wedge\frac1{\alpha^4N}
\right)
\right]^{\frac{2s-1}{2s}},
\end{equation*}
up to constants.
Proving this upper bound should amount to adapting the risk analysis of Section~\ref{section-upper-bounds} to the pointwise loss. 

Establishing a matching lower bound appears more delicate.
The lower-bound techniques developed in Sections~\ref{section_minimax_lower_bounds} and~\ref{section:element_proofs} ultimately reduce the KL analysis to perturbations of a single lag coefficient, through a reduction to a lag-\(h\) model and then to the lag-\(1\) case.
By contrast, natural two-point constructions for pointwise spectral density estimation would involve simultaneous perturbations of many autocovariance coefficients. 
A sharp lower bound for this pointwise problem therefore remains open.

%%%%%%%%%%

\paragraph{Asymptotic equivalence under LDP.}

In Section~\ref{sect3.3:asympto-equiv}, we adapted Le Cam's notion of
deficiency to the LDP constraint in order to formalize the gap between
the privacy costs in two sequences of models that are asymptotically
equivalent in the classical sense.
However, beyond the simple case covered by Lemma~\ref{lem:TV-compatible}, it remains unclear under what conditions this LDP deficiency tends to zero.

Moreover, classical asymptotic equivalence often involves models with
different observation spaces;
for instance, \cite[Theorems~1.1 and~1.2]{golubev10} show that the time-series models considered in Section~\ref{sect3.3:asympto-equiv} are
asymptotically equivalent to a sequence of Gaussian white-noise
experiments with mean  \(\log f\) and noise level proportional to
\(N^{-1/2}\) (under additional restrictions on the parameter space).
To allow for different observation spaces, we extend
the deficiency~\eqref{eq:le-cam-deficiency-ldp} to two potentially
different classes of privacy mechanisms.
For two statistical models
\(\mathcal P=(P_\theta)_{\theta\in\Theta}\) and
\(\mathcal Q=(Q_\theta)_{\theta\in\Theta}\) with the same parameter
space \(\Theta\), and two classes \(\mathcal M\) and \(\mathcal N\) of
privacy mechanisms for \(\mathcal P\) and \(\mathcal Q\), respectively,
we define
\[
\mathcal D_{\mathrm{LC},\mathcal M,\mathcal N}
\bigl(\mathcal P\,\big\|\,\mathcal Q\bigr)
\ := \ 
\sup_{K\in\mathcal N}\ 
\inf_{\substack{K'\in\mathcal M\\ L'}}\ 
\sup_{\theta\in\Theta}
\left\|
L'\circ K'\circ P_\theta
-
K\circ Q_\theta
\right\|_{\mathrm{TV}},
\]
where \(L'\) ranges over all Markov kernels from the output space of
\(K'\) to that of \(K\).
Defining suitable classes \(\mathcal M\) and \(\mathcal N\) of privacy
mechanisms for abstract observation spaces, such as the observation spaces of
Gaussian white-noise experiments, and studying the associated deficiency
\(\mathcal D_{\mathrm{LC},\mathcal M,\mathcal N}\),
could help to better understand the impact of privacy constraints on
other classical asymptotic equivalence results such as in~\cite{Reiss2008}.

We now return to the time-series and independent models
\(\mathcal P_N\) and \(\mathcal Q_N\) of Proposition~\ref{prop:golubev_etal-upm}.
To account for their different privacy costs, it is natural to study
the behavior of
\[
\mathcal D_{\mathrm{LC},
\mathcal M_{\alpha_N}^N,\mathcal M_{\alpha'_N}^N}
\bigl(\mathcal P_N\,\big\|\,\mathcal Q_N\bigr)
\]
for two possibly different sequences
\((\alpha_N)_{N\geq1}\) and \((\alpha'_N)_{N\geq1}\).
Proposition~\ref{prop:golubev_etal-upm} shows that, if
\(\alpha_N=\alpha'_N\to0\) and \(\alpha_N^2N\to\infty\), this
deficiency does not vanish as \(N\to\infty\).
The proof exploits the different privacy costs \(\alpha_N^4\) and
\(\alpha_N^2\) in the two models.
This raises the question of whether the deficiency vanishes when
\(\alpha'_N=\alpha_N^2\), that is, when the privacy levels in the two
models are chosen so that their respective privacy costs are on the
same scale.

%%%%%

\paragraph{Beyond Gaussian time series.}
Our upper and lower bounds are established for stationary Gaussian
processes, and Gaussianity plays an essential role in both parts of the
analysis. 
For the lower bounds, we use Gaussian representations that
create useful independence structures. 
For the upper bounds, the construction of the estimators uses explicit
identities for Gaussian distributions, and
the risk analysis relies on covariance inequalities for bounded functionals of Gaussian processes.
By contrast, the estimators developed in~\cite{Kroll2024} incur logarithmic losses but also apply in a non-Gaussian setting, namely to linear processes with sub-Gaussian noise whose distribution need not be precisely known. 
The Gaussian-specific relations used in our upper-bound analysis are
generally unavailable in this broader setting, so our method for
removing the logarithmic losses does not extend directly to it.

It is natural to ask whether the additional logarithmic factors are
intrinsic to settings with an unknown noise distribution, or whether
one can construct estimators that are agnostic to this distribution
and still attain the Gaussian rates without additional logarithmic
factors. It is also natural to ask what rates can be achieved under
heavier tails or when the tail behavior is unknown, or under more
general dependence assumptions such as mixing conditions.

\paragraph{Interactive privacy.}
Finally, the present paper focuses on non-interactive mechanisms, whereas several related estimation problems have been studied under interactive local privacy~\cite{butucea2025nonparametric}.
This raises the question of characterizing the sharp interactive rates for the time-series estimation problems considered here. 
For further questions concerning interactive local privacy for dependent data, we refer to~\cite{butucea2025nonparametric}.

%%%%%
%%%%%%%
%%% Local Variables:
%%% mode: latex
%%% TeX-master: "main.tex"
%%% ispell-local-dictionary: "american"
%%% End:

%%%%%%%

\section{Main technical tools}  \label{section:element_proofs}

This section gathers the main tools used in the proofs of
Theorems~\ref{thm:KL-bound-gaussian} and~\ref{thm:f-estimator}.  
It is organized around three types of results.
Section~\ref{subsection:reduction-results} establishes reduction
principles for privatized KL divergences.
Section~\ref{subsection:KL-contraction-results} studies how KL
divergences are contracted when probability measures are transformed
by Markov kernels.  Section~\ref{subsection:gaussian-bounds} collects
bounds specific to stationary Gaussian processes.  Since some of these
results may be of independent interest, they are presented in general
settings whenever possible.  Each result is stated in its own setting,
with assumptions local to that setting.

\textit{Notation: kernel composition.}  
Recall from Section~\ref{section-intro} that a Markov kernel \(K\)
from a measurable space \(\mathcal X\) to a measurable space
\(\mathcal Z\) is a map \(K:(x,A)\mapsto K(A\mid x)\).
For a probability measure \(P\) on \(\mathcal X\), we write
\(K\circ P\) for the probability measure on \(\mathcal Z\) defined
by
\[
(K\circ P)(A)
:=
\int_{\mathcal X} K(A\mid x)\,P(dx),
\qquad A\subseteq\mathcal Z
\text{ measurable};
\]
see, for instance,~\cite{polyanskiy2025information}.
This operation naturally extends to coordinatewise kernels:
For \(K_{1:N}=(K_1,\ldots,K_N)\), where \(K_i\) is a Markov kernel from
\(\mathcal X_i\) to \(\mathcal Z_i\), and \(P\) is a probability measure
on \(\mathcal X_1\times\cdots\times\mathcal X_N\), we define
\(K_{1:N}\circ P\) by
\((K_{1:N}\circ P)(A_1\times\cdots\times A_N) 
:=\int \prod_{i=1}^N K_i(A_i\mid x_i)\,P(dx)\),
for measurable \(A_i\subseteq\mathcal Z_i\).

\textit{\(\alpha\)-LDP Markov kernels.} The notation \(\mathcal M_\alpha\), introduced in
Section~\ref{section-intro} for \(\alpha\)-LDP Markov kernels from
\(\mathbb R\), extends to any measurable input space \(\mathcal X\). 
That is, \(\mathcal M_\alpha(\mathcal X)\) denotes the class of Markov kernels from
\(\mathcal X\) to an arbitrary standard Borel space satisfying the
\(\alpha\)-LDP condition~(\ref{eq:alpha-ldp-cond}), with
\(x,x'\in\mathcal X\).
Similarly, if \(\mathcal X=\mathcal X_1\times\cdots\times\mathcal X_N\), we write
\(\mathcal M_\alpha^N(\mathcal X)\) for the class of tuples
\(K_{1:N}=(K_1,\ldots,K_N)\) such that
\(K_i\in\mathcal M_\alpha(\mathcal X_i)\) for every \(i\in[N]\).

%%%%%%%%%%

\subsection{Reduction principles for privatized KL divergence}
\label{subsection:reduction-results}

This subsection presents a reduction principle for experiments with a
shared latent component.  Let \(P\) and \(Q\) be two probability
measures defined on a common measurable space, and let \(U\) and \(V\)
be random variables defined on this common space, taking values in
\(\mathcal U\) and \(\mathcal V\), respectively.  We write \(P^Y\) and
\(Q^Y\) for the law of any variable \(Y\) under \(P\) and \(Q\).
Assume that, under both \(P\) and \(Q\), the variables \(U\) and \(V\)
are independent, and that \(V\) has the same distribution; that is:
\begin{equation}
  \label{eq:independence_U-V_under_P-Q}
P^{U,V}=P^U\otimes P^V,
\qquad
Q^{U,V}=Q^U\otimes P^V.  
\end{equation}
Thus, the two laws \(P^{U,V}\) and \(Q^{U,V}\) differ only through the
distribution of \(U\), while the distribution of \(V\) is shared.

Let \(g:\mathcal U\times\mathcal V\to\mathcal X\) be measurable.
The variable \(g(U,V)\) represents a deterministic transformation of the
informative component \(U\) and the shared component \(V\). 
The next proposition compares the supremum, over private mechanisms, of the privatized KL divergence for \(g(U,V)\) with the corresponding quantity for \(U\).

\begin{proposition}
\label{prop:shared-latent-reduction}
Under the above assumptions, for every \(\alpha>0\),
\[
\sup_{K\in\mathcal M_\alpha(\mathcal X)}
\D\!\left(
K\circ P^{g(U,V)}
\,\middle\|\,
K\circ Q^{g(U,V)}
\right)
\ \le \
\sup_{\tilde K\in\mathcal M_\alpha(\mathcal U)}
\D\!\left(
\tilde K\circ P^U
\,\middle\|\,
\tilde K\circ Q^U
\right).
\]

Moreover, the same principle holds for mechanisms acting on \(N\)
coordinates. Let \(N\ge1\), and assume that
\(\mathcal U=\mathcal U_1\times\cdots\times\mathcal U_N\),
and
\(\mathcal X=\mathcal X_1\times\cdots\times\mathcal X_N\),
and that \(g\) acts coordinatewise in the first variable:
\[
g(u,v)=\bigl(g_1(u_1,v),\dots,g_N(u_N,v)\bigr),
\qquad
u=(u_1,\dots,u_N)\in\mathcal U,\quad v\in\mathcal V,
\]
where each \(g_i:\mathcal U_i\times\mathcal V\to\mathcal X_i\) is
measurable. Then, for every \(\alpha>0\),
\[
\sup_{K_{1:N}\in\mathcal M_\alpha^N(\mathcal X)}
\D\!\left(
K_{1:N}\circ P^{g(U,V)}
\,\middle\|\,
K_{1:N}\circ Q^{g(U,V)}
\right)
\ \le \
\sup_{\tilde K_{1:N}\in\mathcal M_\alpha^N(\mathcal U)}
\D\!\left(
\tilde K_{1:N}\circ P^U
\,\middle\|\,
\tilde K_{1:N}\circ Q^U
\right).
\]
\end{proposition}

The proof idea is that any privatization of \(g(U,V)\) can be represented as a privatization of \(U\) alone, thus removing the shared component \(V\).
The proof is deferred to Appendix~\ref{sec:proof-reduction-general}.

%%%%%%%%%
%%%%%%%%

\subsection{KL contraction for Markov kernels}
\label{subsection:KL-contraction-results}

This subsection presents several results on how KL divergences behave when probability measures are transformed by Markov kernels.
We first introduce the notion of a \(\beta\)-contracting kernel.

\paragraph{Definition.}
For \(\beta\in(0,1]\), we say that a Markov kernel \(K\) is
\textit{\(\beta\)-contracting} if, for every pair of probability measures
\(P\) and \(Q\) on the input space of \(K\),
\begin{equation}
\label{eq:def:contracting-kernel}
\D\!\left(K\circ P\,\middle\|\,K\circ Q\right)
\ \le \ 
\beta\,\D\!\left(P\,\middle\|\,Q\right).
\end{equation}

By the data processing inequality (see
Assertion~\ref{item:lem:contraction-key3} in Lemma~\ref{lem:contraction-key}), every Markov kernel is \(1\)-contracting.
Thus, the notion of a \(\beta\)-contracting kernel refines the usual data processing bound when \(\beta<1\).

%%%%%%%%%%%

\subsubsection{Local privacy implies KL contraction}

In the \(\alpha\)-locally private setting, the following lemma gives a
contraction factor that is smaller than~\(1\) when \(\alpha\) is
small.  Such KL contraction inequalities were popularized in the
influential work of Duchi et al.~\cite{duchi2018minimax} and are now
standard tools in the local privacy literature.  The contraction
phenomenon itself is not new; our contribution here is to provide a
rigorous and self-contained measure-theoretic proof of this
contraction result, formulated in terms of Markov kernels, without
relying on conditional densities.  The proof is given in
Appendix~\ref{appendix:proof-kl-contraction-ldp} and remains valid
when the output space of \(K\) is an arbitrary measurable space (the
following lemma thus continues to hold if the output space of $K$ is
not standard Borel).

\begin{lemma}[KL contraction under \(\alpha\)-LDP]
\label{lem:oriented-kl-contraction-ldp}
Let \(\alpha>0\), and set
\[
\beta_\alpha := \frac12\left(e^\alpha-e^{-\alpha}\right)^2\wedge1.
\]
Then, for any measurable space \(\mathcal X\), every kernel \(K\in\mathcal M_\alpha(\mathcal X)\) is \(\beta_\alpha\)-contracting. 
\end{lemma}

Lemma~\ref{lem:oriented-kl-contraction-ldp} shows that, under
\(\alpha\)-local differential privacy, the KL divergence is contracted
by a factor \(\beta_\alpha\), which is equivalent to \(2\alpha^2\) as
\(\alpha\downarrow0\).
Although \(\beta_\alpha\) is only an upper bound on the optimal
contraction factor, it is sufficient for our purpose, as it still provides a useful
measure of how much \(\alpha\)-LDP reduces the information for
distinguishing two input laws \(P\) and \(Q\).

%%%%

\subsubsection{KL contraction implies conditional KL contraction}
\label{subsection:condtional-KL-contraction}

We now state a conditional form of the KL contraction property~\eqref{eq:def:contracting-kernel}. 
To make the connection with~\eqref{eq:def:contracting-kernel} explicit, observe that if a \(\beta\)-contracting kernel \(K\) maps a variable  \(X\) to a variable \(Z\) under both probability measures \(P\) and \(Q\), in the sense that \(P^{Z\mid X}=Q^{Z\mid X}=K\), then~\eqref{eq:def:contracting-kernel} gives \(\D(P^Z\|Q^Z)\le \beta\,\D(P^X\|Q^X)\).
The next lemma can be viewed as a conditional counterpart of this statement: the same contraction remains valid after conditioning on an auxiliary variable \(Y\).
To formulate it, we first introduce some notation.

\textit{Notation: conditional laws and conditional KL divergences.}
The notation below follows~\cite{polyanskiy2025information}.  Given
any joint law \(P^{X,Z}\), we write \(P^X\) for its first marginal and
denote by \(P^{Z\mid X}\) a Markov kernel representing the conditional
distribution of \(Z\) given \(X\), which is called a regular version
of the conditional distribution. For every \(x\) in the state space of
\(X\), we write \(P^{Z\mid X=x}:=P^{Z\mid X}(\cdot\mid x)\) for the
probability measure obtained by evaluating this kernel at \(x\).  When
the state space of \(Z\) is standard Borel, such a regular version
exists and is unique \(P^X\)-almost surely; see, for instance,
\cite{kallenberg2021foundations}. For two joint laws \(P^{X,Z}\) and
\(Q^{X,Z}\) satisfying \(P^X\ll Q^X\), we define the conditional
Kullback--Leibler divergence between \(P^{Z\mid X}\) and
\(Q^{Z\mid X}\), with respect to \(P^X\), by
\begin{equation*}
\D\!\left(P^{Z\mid X}\,\middle\|\,Q^{Z\mid X}\mid P^X\right)
\ := \
\int
\D\!\left(
P^{Z\mid X=x}
\,\middle\|\,
Q^{Z\mid X=x}
\right)
\,P^X(dx).
\end{equation*}
If the state space \(\mathcal Z\) of \(Z\) is not standard Borel or if
\(P^X\not\ll Q^X\), this quantity need not be uniquely defined.
Hence, in the following, we define the above conditional KL divergence
only when \(\mathcal Z\) is standard Borel and \(P^X\ll Q^X\), and
regard it as not well defined otherwise.

\begin{lemma}
\label{lem:general-contraction-with-contiionning}
Let \(P_+\) and \(P_-\) be two probability measures defined on a common measurable
space. 
Let \(X,Y,Z\) be variables defined on this common space, taking
values in \(\mathcal X,\mathcal Y,\mathcal Z\), respectively, with
\(\mathcal X\) and \(\mathcal Z\) standard Borel spaces. Let
\(\beta\in(0,1]\), and let \(K\) be a \(\beta\)-contracting Markov kernel
from \(\mathcal X\) to \(\mathcal Z\). 
If \(P_+^Y\ll P_-^Y\) and
\[
P_\pm^{Z\mid (X,Y)=(x,y)}
=
K(\cdot\mid x),
\qquad \text{ for }
P_\pm^{(X,Y)}\text{-almost every }(x,y),
\]
then
\[
\D\!\left(
P_+^{Z\mid Y}
\,\middle\|\,
P_-^{Z\mid Y}
\,\middle|\,
P_+^Y
\right)
\ \le \
\beta\,
\D\!\left(
P_+^{X\mid Y}
\,\middle\|\,
P_-^{X\mid Y}
\,\middle|\,
P_+^Y
\right).
\]
\end{lemma}

The proof is deferred to Appendix~\ref{appendix:proof-kl-iteration-1-dependent-process}.

%%%%%

\subsubsection{Iterating conditional KL contraction for \(1\)-dependent processes}

\noindent\textit{Definition of 1-dependent processes.}
We say that a process \(X=(X_t)_{t\in\mathbb Z}\) is 1-dependent if
for every \(k\in\mathbb Z\), the two collections \((X_t)_{t\leq k}\)
and \((X_t)_{t\ge k+2}\) are independent.

We now use the conditional KL contraction (Lemma~\ref{lem:general-contraction-with-contiionning}) to obtain a KL bound for a general class of dependent processes, namely those obtained by applying \(\beta\)-contracting Markov kernels to a stationary \(1\)-dependent process.

Let \(P_+\) and \(P_-\) be two probability measures on a common measurable space,
and let \(X=(X_t)_{t\in\mathbb Z}\) be a process defined on this space
and valued in a standard Borel space \(\mathcal X\).
Let \(N\ge1\). As in the observation model of
Section~\ref{section-intro}, let \(Z_{1:N}=(Z_1,\ldots,Z_N)\) be
generated from \(X\) by \(K_{1:N}=(K_1,\ldots,K_N)\), where each \(K_i\)
is a Markov kernel from \(\mathcal X\) to a standard Borel output space.
More precisely, \(K_{1:N}(\cdot\mid X_{1:N})\) is a regular version of
\(P_\pm^{Z_{1:N}\mid X}\).
Under the assumption that \(X\) is stationary and
\(1\)-dependent, the following result bounds the KL divergence between the laws of \(Z_{1:N}\) in terms of conditional KL quantities involving only the
process \(X\).

\begin{proposition}
\label{prop:abstract-kl-iteration}
In the setting above, assume that, under both \(P_+\) and
\(P_-\), the process \(X\) is stationary and \(1\)-dependent. Suppose
moreover that there exists \(\beta\in(0,1]\) such
that, for each \(i\in[N]\), the kernel \(K_i\) is \(\beta\)-contracting.
Then 
\begin{equation*}
\D\!\left(
P_+^{Z_{1:N}}
\,\middle\|\,
P_-^{Z_{1:N}}
\right)
\ \le \
\beta
\sum_{k=0}^{N-1}
\beta^k (N-k)\,
\D\!\left(
P_+^{X_0\mid X_{1:k}}
\,\middle\|\,
P_-^{X_0\mid X_{1:k}}
\,\middle|\,
P_+^{X_{1:k}}
\right),
\end{equation*}
where we assume that all the conditional KL divergences appearing in the right-hand side are well defined (i.e. \(P_+^{X_{1:k}}\ll P_-^{X_{1:k}}\) for all \(k\in[N-1]\)).
Here, for \(i\le j\), we write \(X_{i:j}:=(X_i,\ldots,X_j)\), with the
convention that \(X_{i:j}\) is empty when \(i>j\).
\end{proposition}

In the proposition above, conditioning on an empty set of variables is
understood as no conditioning. 
Thus the conditional KL term with index \(k=0\) in the sum is simply \(\D(P_+^{X_0}\|P_-^{X_0})\).

The proof of this proposition uses Lemma~\ref{lem:general-contraction-with-contiionning} iteratively along a chain-rule decomposition of the KL divergence, together with the \(1\)-dependence of the latent process \(X\). 
The proof of Proposition~\ref{prop:abstract-kl-iteration} is deferred to Appendix~\ref{appendix:proof-kl-iteration-1-dependent-process}.

%%%%
%%%%
%%%%

\subsection{Bounds for Gaussian processes}
\label{subsection:gaussian-bounds}

This subsection collects two types of bounds for Gaussian processes. The first concerns conditional KL divergences, while the second gives covariance bounds for transformations of Gaussian variables.

%%%%%

\subsubsection{Conditional KL bound for Gaussian processes}

We state a Gaussian bound for the conditional KL quantities appearing in Proposition~\ref{prop:abstract-kl-iteration}.

Let \(P_+\) and \(P_-\) be two probability measures on a common measurable space, and let \(X=(X_t)_{t\in\mathbb Z}\) be a real-valued process defined on this space. 
Assume that, under \(P_\pm\), the process \(X\) is a centered stationary Gaussian process with spectral density \(f_{\pm}\) and unit variance.
Assume that there exist \(0<a\le 1/(2\pi)\le b\) such that
\[
a \ \le \ f_{\pm}(\lambda) \ \le \ b,
\qquad \lambda\in[-\pi,\pi].
\]
Recall that \(\|\cdot\|\) denotes the \(L^2([-\pi,\pi])\)-norm.

\begin{lemma}
\label{lem:lemma-conditional-KL-gaussian-general-sp-densities}
Under the above assumptions, there exists a constant \(C_{a,b}\), depending only on \((a,b)\), such that, for all \(k\geq1\),
\[
\D\!\left(
P_+^{X_0\mid X_{1:k}}
\,\middle\|\,
P_-^{X_0\mid X_{1:k}}
\,\middle|\,
P_+^{X_{1:k}}
\right)
\ \le \
C_{a,b}\,
\max_{\pm\in\{+,-\}}
\left\|f_{\pm}-\frac1{2\pi}\right\| \ \ 
\sup_{\lambda\in[-\pi,\pi]} 
\Big{|}f_{+}(\lambda)-f_{-}(\lambda)\Big{|}.
\]
\end{lemma}

The bound is uniform in the conditioning length \(k\). It is controlled by the \(L^2\)-distance of \(f_\pm\) from the white-noise density \(1/(2\pi)\), multiplied by the uniform distance between \(f_+\) and \(f_-\).
The proof is deferred to Appendix~\ref{appendix:gaussian-process-bounds}.

%%%%%%%%%%%%%%

\subsubsection{Covariance of bounded pairwise transformations}
\label{subsection:covariance-results}

While all preceding results in Section~\ref{section:element_proofs} are used in the proof of Theorem~\ref{thm:KL-bound-gaussian}, the covariance bound below is used in the proof of Theorem~\ref{thm:f-estimator}.

\begin{proposition}
\label{thm:cov-bound-signs}
Let \(X=(X_t)_{t\in\mathbb Z}\) be a real-valued stationary Gaussian process with positive variance and autocorrelation function \(\rho\).
Let \(g_1,g_2\) be two measurable functions from \(\mathbb R^2\) to \([-1,1]\).
Suppose that there exists \(\epsilon>0\) such that
\[
|\rho(k)| \ \le \ 1-\epsilon\quad\text{for all}\quad k\in\mathbb Z\setminus\{0\}.
\]
Then, for all \(s,t,u,v\in\mathbb Z\),
\begin{equation*}
\left|
\operatorname{Cov}\left(
g_1(X_s,X_t),
g_2(X_u,X_v)
\right)
\right|
\ \le \
\frac{
|\rho(s-u)|\vee|\rho(t-v)|
+
|\rho(s-v)|\vee|\rho(t-u)|
}{\epsilon}.
\end{equation*}
\end{proposition}

The proof is deferred to Appendix~\ref{proof-kolmogorov}.

%%% Local Variables:
%%% mode: latex
%%% TeX-master: "main.tex"
%%% ispell-local-dictionary: "american"
%%% End:

%%%%%%%%%%

\section{Proof of the lower bound} \label{subsection:proof-lower-bound-thm}

In this section, we first prove Theorem~\ref{thm:KL-bound-gaussian},
which controls the privatized KL divergence, and then use this result
to prove the minimax lower bound of
Theorem~\ref{thm-borne-inf:cas-non-interactif}.

%%%%%%%
%Proof Thm 2
%%%%

\subsection{Proof of Theorem~\ref{thm:KL-bound-gaussian}}
%\label{sec:proof-theor-KL}

As outlined at the end of Section~\ref{section_minimax_lower_bounds}, the proof combines three main ingredients: a reduction to a \(1\)-dependent model, a KL contraction, and a Gaussian KL bound.
These are based on the tools developed in Section~\ref{section:element_proofs}.

Concretely, the proof is organized into five steps.
Step~1 uses Proposition~\ref{prop:shared-latent-reduction} to reduce the model to a simpler one, which is further reduced to a \(1\)-dependent process in Step~2.
Steps~3 and~4 then implement the KL-contraction part of the argument, providing two bounds in terms of conditional KL divergences of the original process \(X\): 
Step~3 gives a bound for large values of \(\alpha\), 
while Step~4 uses Proposition~\ref{prop:abstract-kl-iteration} and Lemma~\ref{lem:oriented-kl-contraction-ldp} to obtain a sharper bound for small values of \(\alpha\).
Step~5 then controls these conditional KL divergences using Lemma~\ref{lem:lemma-conditional-KL-gaussian-general-sp-densities}.

%%%%%%
\medskip 
%%%%%%

\noindent\textit{Notation.} Recall that, in the statistical model of Section~\ref{section-intro},
for any spectral density \(f\), any \(N\geq1\), \(\alpha>0\), and
\(K_{1:N}\in\mathcal M_\alpha^N\), we denote by \(P_{f,K_{1:N}}\) the
joint law of \((X,Z)\), where \(X\) has spectral density \(f\).
In particular, by the definition of kernel composition given at the
beginning of Section~\ref{section:element_proofs}, the marginal law of \(Z\) under \(P_{f,K_{1:N}}\) satisfies
\begin{equation}
\label{eq:privatized-law-composition}
P_{f,K_{1:N}}^Z
\ = \
K_{1:N}\circ P_{f,K_{1:N}}^{X_{1:N}}.
\end{equation}
For \(j\leq k\), we write \(X_{j:k}:=(X_j,\ldots,X_k)\), with the
convention that \(X_{j:k}\) is empty when \(j>k\).

%%%%%%
\bigskip 
%%%%%%

Fix \(N,h\geq1\), \(a>0\), \(\delta\in[-1,1]\), and two spectral
densities \(f_+\) and \(f_-\) satisfying~\eqref{eq:min-f_pm}. 
Set
\[\Delta_h(\lambda)
\ := \
\frac{1}{\pi}\cos(h\lambda),
\qquad \lambda\in[-\pi,\pi].\]

\noindent\textit{Step 1: Reduction to a lag-\(h\)-only model.}
By~\eqref{eq:min-f_pm}, we have
\[f 
\ := \
f_+-\frac{a\delta}{4}\Delta_h
\ = \ 
f_-+\frac{a\delta}{4}\Delta_h \;.\]
We now aim to simplify the problem by replacing the common part \(f\) of
\(f_\pm\) by \(1/(2\pi)\), the spectral density of a unit-variance
Gaussian white noise.

Using~\eqref{eq:min-f_pm}, the fact that \(\Delta_h\) takes values in
\([-1/\pi,1/\pi]\), and \(|\delta|\leq1\), we have
\(f\geq 3a/(4\pi)-a/(4\pi)=a/(2\pi)\),
and can therefore decompose \(f_\pm\) into a common nonnegative
component and a normalized component:
\begin{equation}
\label{eq:decomp-lagh-reduction-new}
f_\pm
\ = \ 
f\pm\frac{a\delta}{4}\Delta_h
\ = \
\underbrace{\left(f-\frac{a}{2\pi}\right)}_{\geq0}
\ +\
a\left(\frac1{2\pi}\pm\frac{\delta}{4}\Delta_h\right).
\end{equation}

We can thus apply Proposition~\ref{prop:shared-latent-reduction} with
\[\mathcal U=\mathcal V=\mathcal X=\mathbb R^N
\qquad\text{and}\qquad
g(u,v)=v+\sqrt a\,u,\]
where \(P^U\), \(Q^U\), and \(P^V=Q^V\) are the laws of the first \(N\)
observations of centered stationary Gaussian processes with spectral
densities
\[
\frac1{2\pi}+\frac{\delta}{4}\Delta_h,
\qquad \qquad 
\frac1{2\pi}-\frac{\delta}{4}\Delta_h,
\qquad \qquad 
f-\frac{a}{2\pi},
\]
respectively.
Under \(P^{U,V}=P^U\otimes P^V\) and
\(Q^{U,V}=Q^U\otimes P^V\), the vectors \(U\) and \(V\) are independent.
By~\eqref{eq:decomp-lagh-reduction-new}, \(g(U,V)\) has the law of the
first \(N\) observations of the process with spectral density \(f_+\)
under \(P\), and \(f_-\) under \(Q\).

Hence, using~\eqref{eq:privatized-law-composition}, the \(N\)-coordinate
version of Proposition~\ref{prop:shared-latent-reduction} gives, for any \(\alpha>0\),
\begin{equation}
\label{eq:kl-reduction-lagh-only}
\sup_{K_{1:N}\in\mathcal M_\alpha^N}
\D\!\left(
P_{f_+,K_{1:N}}^Z
\,\middle\|\,
P_{f_-,K_{1:N}}^Z
\right)
\ \leq \
\sup_{\widetilde K_{1:N}\in\mathcal M_\alpha^N}
\D\!\left(
P_{\frac1{2\pi}+\frac{\delta}{4}\Delta_h,\widetilde K_{1:N}}^Z
\,\middle\|\,
P_{\frac1{2\pi}-\frac{\delta}{4}\Delta_h,\widetilde K_{1:N}}^Z
\right).
\end{equation}
We call a lag-\(h\)-only model a model of centered stationary Gaussian
processes whose covariances vanish at all lags except \(0\) and \(\pm h\).
Thus, the two spectral densities \(\frac1{2\pi}\pm\frac{\delta}{4}\Delta_h\) define such a model.

%%%%%%
\bigskip
%%%%%%

\noindent\textit{Step 2: Reduction to a lag-\(1\)-only model.}
For each \(r\in[h\wedge N]\), let
\[I_r:=\{r,r+h,r+2h,\ldots\}\cap[N],\]
and denote by \(m_r\) the cardinality of \(I_r\).
The sets \((I_r)_{r\in[h\wedge N]}\) form a partition of \([N]\), and hence
\[\sum_{r\in[h\wedge N]}m_r=N.\]
For either of the spectral densities
\(\frac1{2\pi}\pm\frac{\delta}{4}\Delta_h\), the Gaussian blocks
\((X_i)_{i\in I_r}\), \(r\in[h\wedge N]\), 
are mutually independent: the cross-covariances between distinct blocks
vanish, and joint Gaussianity therefore implies independence.
Moreover, for each \(r\in[h\wedge N]\), the block
\((X_i)_{i\in I_r}\) has the same law as the first \(m_r\) observations
of a centered stationary Gaussian process with spectral density
\(\frac1{2\pi}\pm\frac{\delta}{4}\Delta_1\), that is, of a
lag-\(1\)-only model.

In addition, any mechanism \(K_{1:N}\in\mathcal M_\alpha^N\) can be
regrouped according to the partition \((I_r)_{r\in[h\wedge N]}\), with
the mechanism on each block belonging to \(\mathcal M_\alpha^{m_r}\).
Since \(K_{1:N}\) acts coordinatewise, the corresponding privatized
blocks are independent.

Then, for any \(\alpha>0\) and
\(K_{1:N}\in\mathcal M_\alpha^N\), the tensorization property~\ref{item:tensor-kl-va} in
Lemma~\ref{lem:standard-properties-kl}, applied to the independent
privatized blocks, gives
\begin{align*}
\D\!\left(
P_{\frac1{2\pi}+\frac{\delta}{4}\Delta_h,K_{1:N}}^Z
\,\middle\|\,
P_{\frac1{2\pi}-\frac{\delta}{4}\Delta_h,K_{1:N}}^Z
\right)
\ \leq \
\sum_{r\in[h\wedge N]}
\sup_{\widetilde K_{1:m_r}\in\mathcal M_\alpha^{m_r}}
\D\!\left(
P_{\frac1{2\pi}+\frac{\delta}{4}\Delta_1,\widetilde K_{1:m_r}}^Z
\,\middle\|\,
P_{\frac1{2\pi}-\frac{\delta}{4}\Delta_1,\widetilde K_{1:m_r}}^Z
\right).
\end{align*}

%%%%%
\bigskip
%%%%%%%%

It therefore remains to prove that, for some universal constant \(C>0\),
for all \(M\geq1\) and \(\alpha>0\),
\begin{equation}
\label{eq:rmains-to-do-theorem-kl-private-gaussian}
\sup_{K_{1:M}\in\mathcal M_\alpha^M}
\D\!\left(
P_{\frac1{2\pi}+\frac{\delta}{4}\Delta_1,K_{1:M}}^Z
\,\middle\|\,
P_{\frac1{2\pi}-\frac{\delta}{4}\Delta_1,K_{1:M}}^Z
\right)
\ \leq \
C(\alpha^4\wedge1)(M-1)\delta^2.
\end{equation}
Indeed, applying~\eqref{eq:rmains-to-do-theorem-kl-private-gaussian}
with \(M=m_r\) in the preceding bound, and using
\[
\sum_{r\in[h\wedge N]}(m_r-1)
=
N-(h\wedge N)
=
(N-h)_+ \;,
\]
together with~\eqref{eq:kl-reduction-lagh-only}, yields the claimed inequality of the theorem.

To prove~\eqref{eq:rmains-to-do-theorem-kl-private-gaussian}, we fix \(M\geq1\), \(\alpha>0\), and
\(K_{1:M}\in\mathcal M_\alpha^M\), and use the shorthand notation
\[P_\pm
\ := \
P_{\frac1{2\pi}\pm\frac{\delta}{4}\Delta_1,K_{1:M}}.\]

%%%%%%
\bigskip
%%%%%%

\noindent\textit{Step 3: A bound for the lag-\(1\) model via the data processing inequality.}
By~\eqref{eq:privatized-law-composition} and the data processing
inequality (see Lemma~\ref{lem:contraction-key}~\ref{item:lem:contraction-key3}),
we have
\[
\begin{aligned}
\D\!\left(
P_+^Z
\,\middle\|\,
P_-^Z
\right)
\ = \
\D\!\left(
K_{1:M}\circ P_+^{X_{1:M}}
\,\middle\|\,
K_{1:M}\circ P_-^{X_{1:M}}
\right)
\ \leq \
\D\!\left(
P_+^{X_{1:M}}
\,\middle\|\,
P_-^{X_{1:M}}
\right).
\end{aligned}
\]
Moreover, by repeated application of the chain rule in
Lemma~\ref{lem:standard-properties-kl}\ref{item:chain-rule-va}, 
\begin{align*}
\D\!\left(
P_+^{X_{1:M}}
\,\middle\|\,
P_-^{X_{1:M}}
\right)
&\ = \
\sum_{j=1}^M
\D\!\left(
P_+^{X_j\mid X_{(j+1):M}}
\,\middle\|\,
P_-^{X_j\mid X_{(j+1):M}}
\,\middle|\,
P_+^{X_{(j+1):M}}
\right)
\\
&\ = \
\sum_{k=0}^{M-1}
\D\!\left(
P_+^{X_0\mid X_{1:k}}
\,\middle\|\,
P_-^{X_0\mid X_{1:k}}
\,\middle|\,
P_+^{X_{1:k}}
\right),
\end{align*}
where the last equality follows from stationarity and the change of
variable \(k=M-j\).
Note that all the conditional KL divergences above are well defined, since
\(P_+^{X_{1:k}}\ll P_-^{X_{1:k}}\) for every \(k\geq1\), as both are nondegenerate Gaussian laws.

Combining the preceding two displays gives
\[
\D\!\left(
P_+^Z
\,\middle\|\,
P_-^Z
\right)
\ \leq \
\sum_{k=0}^{M-1}
\D\!\left(
P_+^{X_0\mid X_{1:k}}
\,\middle\|\,
P_-^{X_0\mid X_{1:k}}
\,\middle|\,
P_+^{X_{1:k}}
\right).
\]
Since \(X_0\) has the same distribution under \(P_+\) and \(P_-\), the term \(k=0\) vanishes.
We therefore obtain
\begin{equation}
\label{eq:dpi-X-Z_1:N-with-chain-rule}
\D\!\left(
P_+^Z
\,\middle\|\,
P_-^Z
\right)
\ \leq \
(M-1)\;
\sup_{k\geq1}\;
\D\!\left(
P_+^{X_0\mid X_{1:k}}
\,\middle\|\,
P_-^{X_0\mid X_{1:k}}
\,\middle|\,
P_+^{X_{1:k}}
\right).
\end{equation}

%%%%%%
\bigskip
%%%%%%

\noindent\textit{Step 4: A bound for the lag-\(1\) model for small
\(\alpha\).}
By Lemma~\ref{lem:oriented-kl-contraction-ldp}, each kernel \(K_i\) is
\(\beta_\alpha\)-contracting. Since \(X\) is stationary and
\(1\)-dependent under \(P_+\) and \(P_-\),
Proposition~\ref{prop:abstract-kl-iteration} yields
\[
\D\!\left(
P_+^Z
\,\middle\|\,
P_-^Z
\right)
\ \leq \
\beta_\alpha
\sum_{k=0}^{M-1}
\beta_\alpha^k(M-k)
\D\!\left(
P_+^{X_0\mid X_{1:k}}
\,\middle\|\,
P_-^{X_0\mid X_{1:k}}
\,\middle|\,
P_+^{X_{1:k}}
\right).
\]
As in the previous step, the term \(k=0\) vanishes. Hence, whenever
\(\beta_\alpha<1\),
\[
\D\!\left(
P_+^Z
\,\middle\|\,
P_-^Z
\right)
\ \leq \
\frac{\beta_\alpha^2}{1-\beta_\alpha}\;
(M-1)\;
\sup_{k\geq1}
\D\!\left(
P_+^{X_0\mid X_{1:k}}
\,\middle\|\,
P_-^{X_0\mid X_{1:k}}
\,\middle|\,
P_+^{X_{1:k}}
\right).
\]
Combining this with~\eqref{eq:dpi-X-Z_1:N-with-chain-rule}, and with the
convention that
\(\beta_\alpha^2/(1-\beta_\alpha):=+\infty\) when
\(\beta_\alpha=1\), we obtain
\begin{equation}
\label{eq:X-Z_1:N-general-case}
\D\!\left(
P_+^Z
\,\middle\|\,
P_-^Z
\right)
\ \leq \
\left(
1\wedge\frac{\beta_\alpha^2}{1-\beta_\alpha}
\right)
(M-1)
\sup_{k\geq1}
\D\!\left(
P_+^{X_0\mid X_{1:k}}
\,\middle\|\,
P_-^{X_0\mid X_{1:k}}
\,\middle|\,
P_+^{X_{1:k}}
\right).
\end{equation}

%%%%%
\bigskip 
%%%%%%

\noindent\textit{Step 5: Uniform Gaussian bound for the conditional KL divergences.}
We apply Lemma~\ref{lem:lemma-conditional-KL-gaussian-general-sp-densities}
with \(f_\pm=\frac1{2\pi}\pm\frac{\delta}{4}\Delta_1\).
Since \(|\delta|\leq1\), these spectral densities satisfy
\(\frac1{4\pi}\leq f_\pm\leq\frac3{4\pi}\).
Moreover,
\[
\max_{\pm\in\{+,-\}}
\left\|
f_\pm-\frac1{2\pi}
\right\|
=
\frac{|\delta|}{4}\|\Delta_1\|,
\qquad \qquad 
\sup_{\lambda\in[-\pi,\pi]}
|f_+(\lambda)-f_-(\lambda)|
=
\frac{|\delta|}{2\pi}.
\]
Hence, Lemma~\ref{lem:lemma-conditional-KL-gaussian-general-sp-densities}
gives a universal constant \(C>0\) such that
\begin{equation}
\label{eq:X-Z_1:N-general-case-supbound}
\sup_{k\geq1}
\D\!\left(
P_+^{X_0\mid X_{1:k}}
\,\middle\|\,
P_-^{X_0\mid X_{1:k}}
\,\middle|\,
P_+^{X_{1:k}}
\right)
\ \leq \
C\delta^2.
\end{equation}

%%%%
\bigskip
%%%%%

\noindent\textit{Conclusion.}
Recall from Lemma~\ref{lem:oriented-kl-contraction-ldp} that
\(\beta_\alpha=\frac12(e^\alpha-e^{-\alpha})^2\wedge1\).
In particular, \(\beta_\alpha\sim2\alpha^2\) as \(\alpha\downarrow0\),
and therefore
\(\beta_\alpha^2/(1-\beta_\alpha)\sim4\alpha^4\).
Consequently, there exists a universal constant \(C'\) such that, for every
\(\alpha>0\), 
\[
1\wedge\frac{\beta_\alpha^2}{1-\beta_\alpha}
\ \leq \
C'(1 \wedge \alpha^4).
\]
Combining this with~\eqref{eq:X-Z_1:N-general-case} and~\eqref{eq:X-Z_1:N-general-case-supbound} therefore gives 
\[
\D\!\left(
P_+^Z
\,\middle\|\,
P_-^Z
\right)
\ \leq \
\tilde C(\alpha^4\wedge1)(M-1)\delta^2,
\]
for some universal constant \(\tilde C\).
Since \(P_\pm = P_{\frac1{2\pi}\pm\frac{\delta}{4}\Delta_1,K_{1:M}}\)
for an arbitrary \(K_{1:M}\in\mathcal M_\alpha^M\), taking the
supremum over \(K_{1:M}\) proves~\eqref{eq:rmains-to-do-theorem-kl-private-gaussian}, and concludes the proof of Theorem~\ref{thm:KL-bound-gaussian}.

%%%%%
%%%%%

%%%%%%%
%Proof Thm 1
%%%%

\subsection{Proof of Theorem~\ref{thm-borne-inf:cas-non-interactif}}

The proof follows the strategy outlined in Section~\ref{section_minimax_lower_bounds}, organized into three steps: a reduction of the estimation problem to the hypercube \(\{-\varepsilon,\varepsilon\}^H\), an Assouad reduction to neighboring privatized KL divergences, and a control of these divergences using
Theorem~\ref{thm:KL-bound-gaussian}.
We then choose \(\varepsilon\) and \(H\) to conclude.

%%%%%%%%
\bigskip 
%%%%

\noindent\textit{Step 1: Hypercube reduction.}
Since \(C_0\ge 1\), we have \(\mathcal F(1,C_1,s)\subset \mathcal F(C_0,C_1,s)\).
It is therefore enough to prove a lower bound over the subclass \(\mathcal F(1,C_1,s)\), corresponding to spectral densities with unit variance.
We next restrict the problem to estimating a finite number of autocovariance coefficients, all with the same amplitude but with possibly different signs.

Let \(H\ge 1\) and \(\varepsilon>0\), whose values will be chosen later.
For each \(\boldsymbol{\omega}\in\{-1,1\}^H\), define
\begin{equation}
\label{eq:hypercube-spectral-densities}
f_{\boldsymbol{\omega}}(\lambda)
\ = \
\frac{1}{2\pi}
\left(
1+2\varepsilon\sum_{h=1}^H \omega_h\cos(h\lambda)
\right),
\qquad \lambda\in[-\pi,\pi].
\end{equation}
This is precisely the Fourier representation
\eqref{eq:f-fourier-representation-even-cosine} associated with an
autocovariance sequence equal to \(1\) at lag \(0\), to
\(\varepsilon\omega_h\) at lags \(\pm h\), \(1\le h\le H\), and to \(0\)
at all remaining lags \(|h|\ge H+1\).
Thus, the vector \(\varepsilon\boldsymbol{\omega}\) is the vector of the
first \(H\) nonzero autocovariance coefficients beyond the variance term.

On this finite-dimensional subclass, since \(\varepsilon\) is fixed and known, estimating the spectral density is equivalent, up to constants, to estimating the sign vector \(\boldsymbol{\omega}\).
The following lemma makes this reduction precise. 
Its proof is deferred to Appendix~\ref{appendix:lower-bound-thm:proofs-lemmas}. 

\begin{lemma}
\label{lem:hypercube-reduction}
For any \(H\ge1\) and \(\varepsilon>0\) such that
\begin{equation} \label{eq:lemma:condition-class-validity}
    \varepsilon H\le \frac14 \qquad\text{and}\qquad \varepsilon^2(2H)^{2s+1}\le C_1^2,
\end{equation}
we have \(f_{\boldsymbol{\omega}}\in\mathcal F(1,C_1,s)\) for every
\(\boldsymbol{\omega}\in\{-1,1\}^H\). 
Moreover,
\[
\mathcal R_{N,\alpha}(1,C_1,s)
\ \ge \
\frac{\varepsilon^2}{\pi} \ 
\inf_{K_{1:N}\in\mathcal M_\alpha^N} \ 
\inf_{\hat{\boldsymbol{\omega}}} \ 
\sup_{\boldsymbol{\omega}\in\{-1,1\}^H} \ 
\mathbb E_{f_{\boldsymbol{\omega}},K_{1:N}}
\left[
d_H\left(\hat{\boldsymbol{\omega}}(Z),\boldsymbol{\omega}\right)
\right],
\]
where \(d_H\) denotes the Hamming distance on \(\{-1,1\}^H\), 
the first infimum is over all \(\alpha\)-locally differentially private mechanisms on \(N\) coordinates, and, for each \(K_{1:N}\), the second is over all measurable functions from the output space of \(K_{1:N}\) to \(\{-1,1\}^H\).
\end{lemma}

Thus, under the two conditions in~\eqref{eq:lemma:condition-class-validity}, lower bounding the original
minimax risk reduces to lower bounding the Hamming risk for estimating the sign vector \(\boldsymbol{\omega}\) over the hypercube \(\{-1,1\}^H\).

%%%%%%%%%%%
\bigskip 
%%%%%%

\noindent\textit{Step 2: Assouad reduction to neighboring KL divergences.}
We now recall the following consequence of Assouad's lemma and
Pinsker's inequality; see, for instance, \cite[Theorem~2.12]{Tsybakov2009}.

\begin{lemma}[Assouad's lemma]
\label{assouad-lemma}
Let \(H\ge1\), and let
\(d_H(\boldsymbol{\omega},\boldsymbol{\omega}')
=
\sum_{h=1}^H
\mathbf 1\{\omega_h\ne\omega_h'\}\)
be the Hamming distance on \(\{-1,1\}^H\).
Let
\(\{\mu_{\boldsymbol{\omega}}:\boldsymbol{\omega}\in\{-1,1\}^H\}\)
be a family of probability measures on a common measurable space.
Then
\[
\inf_{\hat{\boldsymbol{\omega}}}
\sup_{\boldsymbol{\omega}\in\{-1,1\}^H}
\mathbb E_{\boldsymbol{\omega}}
\left[
d_H(\hat{\boldsymbol{\omega}},\boldsymbol{\omega})
\right]
\ \ge \ 
\frac{H}{2}
\min_{\substack{
\boldsymbol{\omega},\boldsymbol{\omega}'\in\{-1,1\}^H\\
d_H(\boldsymbol{\omega},\boldsymbol{\omega}')=1
}}
\left(
1-
\sqrt{
\frac12
\D\!\left(
\mu_{\boldsymbol{\omega}}
\,\middle\|\,
\mu_{\boldsymbol{\omega}'}
\right)
}
\right),
\]
where \(\mathbb E_{\boldsymbol{\omega}}\) denotes expectation under \(\mu_{\boldsymbol{\omega}}\), 
and the infimum is over all estimators of \(\boldsymbol{\omega}\), that is, 
all measurable maps from the common measurable space to \(\{-1,1\}^H\).
\end{lemma}

Applying Assouad's lemma with
\[
\mu_{\boldsymbol{\omega}}
\ = \
P_{f_{\boldsymbol{\omega}},K_{1:N}}^Z,
\qquad
\boldsymbol{\omega}\in\{-1,1\}^H,
\]
and then using the hypercube reduction from Step~1, we obtain, if
Lemma~\ref{lem:hypercube-reduction} applies,
\begin{equation}
\label{eq:assouad-reduction}
\mathcal R_{N,\alpha}(1,C_1,s)
\ \ge \
\frac{\varepsilon^2 H}{2\pi}
\left(
1-
\sqrt{
\frac{\mathcal V_{N,\alpha,H,\varepsilon}}{2}
}
\right),
\end{equation}
where
\begin{equation}
\label{eq:V-quantity:assouad-reduction}
\mathcal V_{N,\alpha,H,\varepsilon} 
\ := \ 
\sup_{K_{1:N}\in\mathcal M_\alpha^N} \
\max_{\substack{
\boldsymbol{\omega},\boldsymbol{\omega}'\in\{-1,1\}^H\\
d_H(\boldsymbol{\omega},\boldsymbol{\omega}')=1}} \
\D\!\left(
P_{f_{\boldsymbol{\omega}},K_{1:N}}^Z
\,\middle\|\,
P_{f_{\boldsymbol{\omega}'},K_{1:N}}^Z
\right).
\end{equation}

%%%%%%
\bigskip
%%%%%

\noindent\textit{Step 3: Control of the KL divergence.}
We now apply Theorem~\ref{thm:KL-bound-gaussian}.
By~\eqref{eq:hypercube-spectral-densities}, for any
\(\boldsymbol{\omega},\boldsymbol{\omega}'\in\{-1,1\}^H\) such that
\(d_H(\boldsymbol{\omega},\boldsymbol{\omega}')=1\), we have, for every
\(\lambda\in[-\pi,\pi]\),
\[
f_{\boldsymbol{\omega}}(\lambda)
\wedge
f_{\boldsymbol{\omega}'}(\lambda)
\ \geq \
\frac{1}{2\pi}\left(1-2\varepsilon H\right)
\ \geq \
\frac{1}{4\pi},
\]
provided that \(\varepsilon H\leq1/4\).
Moreover, for some \(h\in[H]\),
\[
f_{\boldsymbol{\omega}}(\lambda)
-
f_{\boldsymbol{\omega}'}(\lambda)
\ = \
\frac{2\omega_h\varepsilon}{\pi}\cos(h\lambda).
\]
Thus, the conditions in~\eqref{eq:min-f_pm} hold with
\[
f_+=f_{\boldsymbol{\omega}},
\qquad
f_-=f_{\boldsymbol{\omega}'},
\qquad
a=\frac13,
\qquad
\delta=12\omega_h\varepsilon.
\]
In particular, if \(\varepsilon\leq1/12\), then
\(|\delta|\leq1\).
Therefore, for every \(N\geq2\), \(\alpha>0\),
\(\varepsilon\in(0,1/12]\), and \(H\ge 1\) such that
\(\varepsilon H\leq1/4\), Theorem~\ref{thm:KL-bound-gaussian} gives
\begin{equation}
\label{eq:kl-bound-in-proof}
\mathcal V_{N,\alpha,H,\varepsilon}
\ \leq \
144C(\alpha^4\wedge1)N\varepsilon^2,
\end{equation}
for some universal constant \(C\). 

%%%%%
\bigskip
%%%%%%%

\noindent\textit{Final step: choice of \((\varepsilon, H)\) and conclusion.}
We set
\begin{align*}
\varepsilon_*
&\ := \
\left(
\frac1{12}
\wedge
\left(\frac{C_1^2}{2^{2s+1}}\right)^{1/2}
\wedge
2^{-(2s+1)}
\wedge
\frac1{12\sqrt{2C}}
\right)
\left(
1\wedge\frac1{(\alpha^2\wedge1)\sqrt N}
\right),
\\
H_*
&\ := \
\left\lfloor
\frac12
\left(
\frac{C_1^2\wedge2^{-(2s+1)}}{\varepsilon_*^2}
\right)^{1/(2s+1)}
\right\rfloor.
\end{align*}
The following technical lemma shows that this choice is compatible
with the conditions used in the previous steps.
Its proof is deferred to
Appendix~\ref{appendix:lower-bound-thm:proofs-lemmas}.

\begin{lemma}
\label{lem:choice-parameters-lower-bound-bis}
With \(\varepsilon_*\) and \(H_*\) as above, we have
\[
\varepsilon_*\leq 1/12,
\qquad
H_*\geq1,
\qquad
\varepsilon_*H_*\leq 1/4,
\qquad\text{and}\qquad
\varepsilon_*^2(2H_*)^{2s+1}\leq C_1^2.
\]
Moreover,
\[
144C(\alpha^4\wedge1)N\varepsilon_*^2
\ \leq \
1/2,
\qquad\text{and}\qquad
\frac{\varepsilon_*^2H_*}{4\pi}
\ \geq \
c_{C_1,s}
\left(
1\wedge
\left[(\alpha^4\wedge1)N\right]^{-\frac{2s}{2s+1}}
\right),
\]
where \(c_{C_1,s}>0\) depends only on \(C_1\) and \(s\).
\end{lemma}

The first part of Lemma~\ref{lem:choice-parameters-lower-bound-bis} shows that, with
\(\varepsilon=\varepsilon_*\) and \(H=H_*\), the conditions of Lemma~\ref{lem:hypercube-reduction} and of~\eqref{eq:kl-bound-in-proof}
are satisfied.
The first bound in the second part of Lemma~\ref{lem:choice-parameters-lower-bound-bis}, together
with~\eqref{eq:kl-bound-in-proof}, gives
\[
\mathcal V_{N,\alpha,H_*,\varepsilon_*}
\ \leq \
1/2.
\]
Therefore, \eqref{eq:assouad-reduction} yields
\[
\mathcal R_{N,\alpha}(1,C_1,s)
\ \geq \
\varepsilon_*^2H_*/(4\pi).
\]
The remaining bound in
Lemma~\ref{lem:choice-parameters-lower-bound-bis} now gives the claimed
lower bound for \(\mathcal R_{N,\alpha}(1,C_1,s)\).
Since
\(\mathcal F(1,C_1,s)\subset\mathcal F(C_0,C_1,s)\), the same lower
bound holds for \(\mathcal R_{N,\alpha}(C_0,C_1,s)\), which concludes
the proof of Theorem~\ref{thm-borne-inf:cas-non-interactif}.

%%% Local Variables:
%%% mode: latex
%%% TeX-master: "main.tex"
%%% ispell-local-dictionary: "american"
%%% End:

%%%%%%%

\section{Proof of the upper bound}
\label{section:proof-sketch-upper-bound-thm}

We prove Propositions~\ref{lem:variance-estimator} and~\ref{prop:fred-estimator}, and then deduce Theorem~\ref{thm:f-estimator}.
Throughout this section, when \((N, \alpha, f)\) is fixed, we simply write \(\mathbb E\), \(\operatorname{Var}\), and \(\operatorname{Cov}\) for expectation, variance, and covariance under \(P_{f,K_{1:N}^{\mathrm{Lap},\alpha}}\).

%%%%%

The following lemma provides \(\ell^2\)- and \(\ell^1\)-bounds that will be used repeatedly throughout this section.
Its proof follows directly from the Sobolev constraint defining \(\mathcal F(C_0,C_1,s)\), and is deferred to Appendix~\ref{appendix:upper-bound:additional-proofs}.

\begin{lemma}
\label{lem:rho-sum-uniform}
For every \(f\in\mathcal F(C_0,C_1,s)\), we have
\[
\sum_{h\in\mathbb Z\setminus\{0\}}|\rho_f(h)|^2
\ \leq \
C_1^2,
\qquad\text{and}\qquad
\sum_{h\in\mathbb Z\setminus\{0\}}|\rho_f(h)|
\ \leq \
\sqrt{2\zeta(2s)}\,C_1\;,
\]
where \(\zeta\) denotes the Riemann zeta function.
\end{lemma}

%%%%%%%%%
%Section 6.1. Proof Variance estimation
%%%%%%%%

\subsection{Proof of Proposition~\ref{lem:variance-estimator}}
\label{proof-section:-performance-estimator-variance}

We will use the following result due to Hans Gebelein (1941).

\begin{lemma}[Gebelein inequality \cite{Gebelein1941}]
\label{lem:gebelein}
Let \((U,V)\) be a \(\mathbb R^2\)-valued centered Gaussian vector
such that \(\mathbb E[U^2]=\mathbb E[V^2]=1\). 
Then
\begin{equation*}
\left|
\mathbb E[g(U)\tilde g(V)]
\right|
\le
|\mathbb E[UV]|
\left(\mathbb E|g(U)|^2\right)^{1/2}
\left(\mathbb E|\tilde g(V)|^2\right)^{1/2}
\end{equation*}
for all measurable functions \(g,\tilde g:\mathbb R\to\mathbb R\) such that
\(g(U),\tilde g(V)\in L^2\) and \(\mathbb E[g(U)]=\mathbb E[\tilde g(V)]=0\).
\end{lemma}

%%%%%%%
\medskip

Fix \(N\geq1\), \(\alpha>0\), and \(f\in\mathcal F(C_0,C_1,s)\).
Recall that \(Z^{(1)}:=(Z_{i,1})_{i\in[N]}\) is defined by~\eqref{defi-K_lap}.
Recall also that
\[
p_f
=
2\left(
1-\Phi\!\left(\frac{C_0}{\sqrt{\gamma_f(0)}}\right)
\right)
\in I_{C_0},
\qquad \quad 
I_{C_0}
=
\left[
2\bigl(1-\Phi(C_0^{3/2})\bigr),
2\bigl(1-\Phi(C_0^{1/2})\bigr)
\right],
\]
and that
\[
\hat p_f=\Pi_{I_{C_0}}(\overline Z^{(1)}) \qquad \text{ where } \quad  \overline Z^{(1)}:=\frac1N\sum_{i=1}^N Z_{i,1}\;. 
\]
The proof is organized into five steps. 
The first two reduce the estimation error of \(\hat\gamma(0)\) to \(|\overline Z^{(1)}-p_f|\), while the last three control this error using the Laplace noise variance, Gebelein's inequality, and the uniform summability bound of Lemma~\ref{lem:rho-sum-uniform}.

%%%%%%
\medskip 
%%%%%%

\noindent\textit{Step 1. Lipschitz reduction from \(\hat\gamma(0)\) to \(\hat p_f\).}
Let \(T:I_{C_0}\to[C_0^{-1},C_0]\) denote the inverse of the map
\(u\mapsto2\left(1-\Phi(C_0/\sqrt u)\right)\). Explicitly,
\(T(u)=\left(C_0/\Phi^{-1}(1-u/2)\right)^2\).
Since \(T\) is continuously differentiable on the compact interval
\(I_{C_0}\), it is Lipschitz. Thus, there exists a constant
\(L_{C_0}>0\), depending only on \(C_0\), such that
\[
|T(u)-T(v)|
\ \leq \
L_{C_0}|u-v|,
\qquad
u,v\in I_{C_0}.
\]
By~\eqref{eq:inverse-formula-motivatin-estimator}
and~\eqref{defi-variance-estimator},
\(\gamma_f(0)=T(p_f)\) and
\(\hat\gamma(0)=T(\hat p_f)\), respectively. We therefore obtain
\[
|\hat\gamma(0)-\gamma_f(0)|^2
\ \leq \
L_{C_0}^2|\hat p_f-p_f|^2.
\]

%%%%%%%%
\medskip 
%%%%%%

\noindent\textit{Step 2. Projection reduction from
\(\hat p_f\) to \(\overline Z^{(1)}\).}
By definition, \(\hat p_f\) is the projection of \(\overline Z^{(1)}\) onto \(I_{C_0}\).
Since \(p_f\in I_{C_0}\) and projection onto an interval can only
decrease the distance to points in that interval, we have
\[
|\hat p_f-p_f|
\ \leq \
|\overline Z^{(1)}-p_f|.
\]
Combining this inequality with the bound obtained in Step~1 yields
\begin{equation}
\label{eq:variance-estimator-reduction-to-zbar}
|\hat\gamma(0)-\gamma_f(0)|^2
\ \leq \
L_{C_0}^2|\overline Z^{(1)}-p_f|^2.
\end{equation}

%%%%%%%%
\medskip 
%%%%%%

\noindent\textit{Step 3. Quadratic risk bound for
\(\overline Z^{(1)}\).}
We now bound the second moment of \(\overline Z^{(1)}-p_f\).
By definition,
\[
\overline Z^{(1)}-p_f
\ = \
\frac1N\sum_{i=1}^N
\left(
\mathbf 1_{\{|X_i|>C_0\}}-p_f
\right)
+
\frac{3}{\alpha N}\sum_{i=1}^N W_{i,1}.
\]
Since the two sums are centered and independent, and since a standard Laplace
random variable has variance \(2\), we obtain
\begin{equation}
\label{eq:variance-estimator-zbar}
\mathbb E
\left[
|\overline Z^{(1)}-p_f|^2
\right]
\ = \ 
\mathbb E
\left[
\left|
\frac1N\sum_{i=1}^N
\left(
\mathbf 1_{\{|X_i|>C_0\}}-p_f
\right)
\right|^2
\right]
+
\frac{18}{N\alpha^2}.
\end{equation}

%%%%%%%%
\medskip 
%%%%%%

\noindent\textit{Step 4. Covariance of Gaussian threshold indicators.}
It remains to control the first term in
\eqref{eq:variance-estimator-zbar}. For \(i,j\in[N]\), set
\[
U_i
=
\frac{X_i}{\sqrt{\gamma_f(0)}},
\qquad
U_j
=
\frac{X_j}{\sqrt{\gamma_f(0)}}.
\]
Then \((U_i,U_j)\) is a centered Gaussian vector in \(\mathbb R^2\)
with unit marginal variances. Define
\[
g(u)
\ = \
\mathbf 1_{\{|u|>C_0/\sqrt{\gamma_f(0)}\}}-p_f.
\]
Then
\(\mathbb E[g(U_i)]
=\mathbb E[g(U_j)]=0\).
Applying Lemma~\ref{lem:gebelein} with \(\tilde g=g\), and using
\(|g|\leq1\), gives
\begin{equation}
\label{2nd-eq:proof-variance-estimator}
\left|
\mathbb E
[g(U_i)g(U_j)]
\right|
\ \leq \
\left|
\mathbb E[U_iU_j]
\right|.
\end{equation}
Since
\;\(g(U_i)=\mathbf 1_{\{|X_i|>C_0\}}-p_f\) \;and\;
\(g(U_j)=\mathbf 1_{\{|X_j|>C_0\}}-p_f\)\;, we have
\[
\mathbb E
[g(U_i)g(U_j)]
\ = \
\operatorname{Cov}
\left(
\mathbf 1_{\{|X_i|>C_0\}},
\mathbf 1_{\{|X_j|>C_0\}}
\right).
\]
Moreover,
\[
\mathbb E[U_iU_j]
=
\frac{\gamma_f(|i-j|)}{\gamma_f(0)}
=
\rho_f(|i-j|).
\]
Plugging these two identities into
\eqref{2nd-eq:proof-variance-estimator} yields
\[
\left|
\operatorname{Cov}
\left(
\mathbf 1_{\{|X_i|>C_0\}},
\mathbf 1_{\{|X_j|>C_0\}}
\right)
\right|
\ \leq \
|\rho_f(|i-j|)|.
\]

%%%%%%%%
\medskip 
%%%%%%

\noindent\textit{Step 5. Sobolev summability and conclusion.}
Using the covariance bound above, we obtain
\begin{align*}
\mathbb E
\left[
\left|
\frac1N\sum_{i=1}^N
\left(
\mathbf 1_{\{|X_i|>C_0\}}-p_f
\right)
\right|^2
\right]
\ \leq \
\frac1{N^2}
\sum_{i,j=1}^N
|\rho_f(|i-j|)|
\ \leq \
\frac1N
\sum_{h\in\mathbb Z}
|\rho_f(h)|.
\end{align*}
Since \(\rho_f(0)=1\), Lemma~\ref{lem:rho-sum-uniform} gives
\(\sum_{h\in\mathbb Z}|\rho_f(h)|
\leq1+\sqrt{2\zeta(2s)}\,C_1\).
Inserting this bound into~\eqref{eq:variance-estimator-zbar} yields
\[
\mathbb E
\left[
|\overline Z^{(1)}-p_f|^2
\right]
\ \leq \
\frac1N
\left(
1+\sqrt{2\zeta(2s)}\,C_1+\frac{18}{\alpha^2}
\right).
\]
Then, taking expectations in
\eqref{eq:variance-estimator-reduction-to-zbar} and using
\(1+\alpha^{-2}\leq2(\alpha^2\wedge1)^{-1}\), we obtain
\begin{equation}
\label{eq-profo-final-ineq-prop1}
\mathbb E
\left[
|\hat\gamma(0)-\gamma_f(0)|^2
\right]
\ \leq \
\frac{C_{C_0,C_1,s}}
{(\alpha^2\wedge1)N}.
\end{equation}

Since \(p_f,\hat p_f\in I_{C_0}\) and
\(\gamma_f(0)=T(p_f)\), \(\hat\gamma(0)=T(\hat p_f)\),  we have
\(\gamma_f(0),\hat\gamma(0)\in T(I_{C_0})=[C_0^{-1},C_0]\).
Hence,
\[
|\hat\gamma(0)-\gamma_f(0)|^2
\ \leq \
C_0^2.
\]
Taking the minimum of this bound and \eqref{eq-profo-final-ineq-prop1}, and then the supremum over \(f\in\mathcal F(C_0,C_1,s)\), concludes the proof of
Proposition~\ref{lem:variance-estimator}.

%%%%%%
%%%%%%%

%%%%%%%%%%%
% 6.2. Reduced spectral density 
%%%

\subsection{Proof of Proposition~\ref{prop:fred-estimator}}
\label{sec:proof-prop-fred-est}

The following lemma shows that the nonzero-lag autocorrelations are
uniformly bounded away from \(1\) in absolute value. 
Its proof is deferred to Appendix~\ref{appendix:upper-bound:additional-proofs}.

\begin{lemma}
\label{lem:rho-away-from-one}
There exists a constant
\(\varepsilon_{C_1,s}\in(0,1)\), depending only on \((C_1,s)\),
such that, for every \(f\in\mathcal F(C_0,C_1,s)\) and every
\(h\in\mathbb Z\setminus\{0\}\),
we have\; \(|\rho_f(h)| \leq 
1-\varepsilon_{C_1,s}\).
\end{lemma}

%%%%%%%
\medskip

Fix \(N\geq2\), \(\alpha>0\), and \(f\in\mathcal F(C_0,C_1,s)\).
Recall that \(Z^{(2)}:=(Z_{i,2})_{i\in[N]}\) is defined by~\eqref{defi-K_lap}.
For \(h\in[N-1]\), define the empirical lag-\(h\) autocovariance of
\(Z^{(2)}\) by
\begin{equation}
\label{eq:check_gamma}
\widehat\gamma^{(2)}(h)
\ := \
\frac1N\sum_{j=1}^{N-h}Z_{j,2}Z_{j+h,2}.
\end{equation}
The proof is organized into three steps.
We first control the mean squared error of \(\widehat\gamma^{(2)}(h)\), then transfer this bound to \(\hat\rho(h)\), and finally control the risk of \(\hat f_{\mathrm{red},H}\).

%%%%%
\medskip 
%%%%%%%

\noindent\textit{Step 1. Mean squared error of \(\widehat\gamma^{(2)}(h)\).}
Let \(h\in[N-1]\).
By Lemma~\ref{lem:sign-gaussian}, and the fact that
the variables \((W_{j,2})_{j\in[N]}\) are i.i.d., centered, and
independent of \(X\), we have
\[
\mathbb E\!\left[\widehat\gamma^{(2)}(h)\right]
=
\frac{N-h}{N}\,
\frac{2}{\pi}\arcsin\!\left(\rho_f(h)\right).
\]
Consequently, when \(\widehat\gamma^{(2)}(h)\) is used to estimate
\(\frac{2}{\pi}\arcsin(\rho_f(h))\), its squared bias satisfies
\begin{equation}
\label{eq:bias-check_gamma}
\operatorname{Bias}\!\left(\widehat\gamma^{(2)}(h)\right)^2
\ = \ 
\frac{h^2}{N^2}
\left(
\frac{2}{\pi}\arcsin\!\left(\rho_f(h)\right)
\right)^2
\ \leq \
\frac{h^2}{N^2},
\end{equation}
where the last inequality follows from
\(\left|\frac{2}{\pi}\arcsin(\rho_f(h))\right|\leq1\).

We now control the variance. Set
\[
S_j:=\operatorname{sgn}(X_j),
\qquad
\sigma_\alpha^2:=
\operatorname{Var}\!\left(\frac3\alpha W_{j,2}\right) = \frac{18}{\alpha^2}.
\]
Lemma~\ref{lem:variance-check-gamma} controls the contribution of the Laplace noise to the variance. 
Its proof is a direct variance expansion using the independence and centering of the Laplace noises, and is deferred to Appendix~\ref{appendix:upper-bound:additional-proofs}.

\begin{lemma}
\label{lem:variance-check-gamma}
For every \(h\in[N-1]\), we have
\[
\operatorname{Var}\!\left(\widehat\gamma^{(2)}(h)\right)
\ \leq \
\frac4{N^2}
\sum_{s,u=1}^{N-h}
\left|
\operatorname{Cov}\!\left(
S_sS_{s+h},
S_uS_{u+h}
\right)
\right|
+
\frac{4(N-h)}{N^2}
\left(
2\sigma_\alpha^2+\sigma_\alpha^4
\right).
\]
\end{lemma}

By Lemma~\ref{lem:rho-away-from-one}, the assumption of Proposition~\ref{thm:cov-bound-signs} holds with \(\epsilon=\varepsilon_{C_1,s}\). 
Applying this proposition with
\(g_1(x,y)=g_2(x,y):=\operatorname{sgn}(x)\operatorname{sgn}(y)\),
we obtain, for all \(j,k\in\mathbb Z\),
\[
\left|
\operatorname{Cov}\!\left(
S_jS_{j+h},
S_kS_{k+h}
\right)
\right|
\ \leq \ 
\frac{
|\rho_f(j-k)|
+
|\rho_f(j-k-h)|\vee|\rho_f(j-k+h)|
}{
\varepsilon_{C_1,s}
}.
\]
Here \(\varepsilon_{C_1,s}\in(0,1)\) depends only on \((C_1,s)\).
Combining with Lemma~\ref{lem:variance-check-gamma}, we obtain
\[
\begin{aligned}
\operatorname{Var}\!\left(\widehat\gamma^{(2)}(h)\right)
\ \leq & \ 
\frac{4}{\varepsilon_{C_1,s}N^2} 
\sum_{j,k=1}^{N-h}
\left(
|\rho_f(j-k)|
+
|\rho_f(j-k-h)|\vee|\rho_f(j-k+h)|
\right)
\\
&\quad+
\frac{4(N-h)}{N^2}
\left(
\frac{36}{\alpha^2}
+
\frac{18^2}{\alpha^4}
\right).
\end{aligned}
\]

For every nonnegative function \(\ell:\mathbb Z\to\mathbb R_+\), we have
\(\sum_{j,k=1}^{N-h}\ell(j-k)\leq(N-h)\sum_{r\in\mathbb Z}\ell(r)\).
Using this inequality, and Lemma~\ref{lem:rho-sum-uniform}, we obtain
\begin{align*}
\sum_{j,k=1}^{N-h}
\left(
|\rho_f(j-k)|
+
|\rho_f(j-k-h)|\vee|\rho_f(j-k+h)|
\right)
\ &\leq \
3(N-h)\sum_{r\in\mathbb Z}|\rho_f(r)|
\\
\ &\leq \
3(N-h)
\left(
1+\sqrt{2\zeta(2s)}\,C_1
\right).
\end{align*}
Consequently,
\begin{equation}
\label{eq:variance-check-gamma}
\operatorname{Var}\!\left(\widehat\gamma^{(2)}(h)\right)
\ \leq \
\frac{
12(N-h)\left(1+\sqrt{2\zeta(2s)}\,C_1\right)
}{
\varepsilon_{C_1,s}N^2
}
+
\frac{4(N-h)}{N^2}
\left(
\frac{36}{\alpha^2}
+
\frac{18^2}{\alpha^4}
\right).
\end{equation}
Combining this bound with~\eqref{eq:bias-check_gamma}, we obtain
\begin{equation}
\label{eq:mse-check-gamma}
\begin{aligned}
\mathbb E\!\left[
\left(
\widehat\gamma^{(2)}(h)
-
\frac2\pi\arcsin\!\left(\rho_f(h)\right)
\right)^2
\right]
\ \leq \ &
\frac{
12(N-h)\left(1+\sqrt{2\zeta(2s)}\,C_1\right)
}{
\varepsilon_{C_1,s}N^2
}
\\
& \qquad \quad +
\frac{4(N-h)}{N^2}
\left(
\frac{36}{\alpha^2}
+
\frac{18^2}{\alpha^4}
\right)
+
\frac{h^2}{N^2}.
\end{aligned}
\end{equation}

%%%%%
\medskip 
%%%%%

\noindent\textit{Step 2. Lipschitz reduction from
\(\widehat\gamma^{(2)}(h)\) to \(\hat\rho(h)\).}
Let \(h\in[N-1]\). 
The map \(x\mapsto\sin(\pi x/2)\) is \(\pi/2\)-Lipschitz. 
Therefore, by the definition of \(\hat\rho(h)\)
in~\eqref{eq:def-rho-hat},
\begin{equation}
\label{eq:check-gamma_vs_hat-gamma}
\left|\hat\rho(h)-\rho_f(h)\right|
\ \leq \
\frac{\pi}{2}
\left|
\widehat\gamma^{(2)}(h)
-
\frac2\pi\arcsin\!\left(\rho_f(h)\right)
\right|.
\end{equation}

%%%%%
\medskip 
%%%%%

\noindent\textit{Step 3. $L^2$ risk of \(\hat f_{\mathrm{red},H}\).}
Let \(H\in\{0,\ldots,N-1\}\). By the definitions of
\(f_{\mathrm{red}}\) and \(\hat f_{\mathrm{red},H}\)
in~\eqref{eq:fred-fourier-representation} and~\eqref{eq:def-fred-hat},
respectively,
\[
\hat f_{\mathrm{red},H}(\lambda)-f_{\mathrm{red}}(\lambda)
\ = \
\frac1\pi\sum_{h=1}^{H}
\bigl(\hat\rho(h)-\rho_f(h)\bigr)\cos(h\lambda)
-
\frac1\pi\sum_{h=H+1}^{\infty}
\rho_f(h)\cos(h\lambda).
\]
Since the family \((\cos(h\cdot))_{h\geq1}\) is orthogonal in \(L^2([-\pi,\pi])\), and each of its elements has squared norm \(\pi\),
it follows, after taking expectations, that
\begin{align}
\label{eq:fred-risk-decomposition}
\mathbb E\!\left[
\|\hat f_{\mathrm{red},H}-f_{\mathrm{red}}\|^2
\right]
\ = \ 
\frac1\pi
\sum_{h=1}^{H}
\mathbb E\!\left[
|\hat\rho(h)-\rho_f(h)|^2
\right]
+
\frac1\pi
\sum_{h=H+1}^{\infty}|\rho_f(h)|^2.
\end{align}

By~\eqref{eq:mse-check-gamma} and~\eqref{eq:check-gamma_vs_hat-gamma}, using
\(\sum_{h=1}^{H}(N-h)/N^2\leq H/N\), and
\(\sum_{h=1}^{H}h^2\leq H^3\), 
there exists a constant \(C_{C_1,s}\) such that
\begin{align*}
\sum_{h=1}^{H}
\mathbb E\!\left[
|\hat\rho(h)-\rho_f(h)|^2
\right]
\ &\leq \ 
C_{C_1,s}
\left(
\frac{H}{N}
+
\frac{H}{\alpha^2N}
+
\frac{H}{\alpha^4N}
+
\frac{H^3}{N^2}
\right)
\\
\ &\leq \
C_{C_1,s}
\left(
\frac{H}{(1\wedge\alpha^4)N}
+
\frac{H^3}{N^2}
\right),
\end{align*}
where the last inequality uses
\(1+\alpha^{-2}+\alpha^{-4}\leq3/(1\wedge\alpha^4)\).

For the truncation bias, since \(f\in\mathcal F(C_0,C_1,s)\), we have
\begin{align*}
\sum_{h=H+1}^{\infty}|\rho_f(h)|^2
\ \leq \
(1+H)^{-2s}
\sum_{h=H+1}^{\infty}h^{2s}|\rho_f(h)|^2
\ \leq \
C_1^2(1+H)^{-2s}.
\end{align*}
Combining the last two bounds with~\eqref{eq:fred-risk-decomposition}
yields, for some constant \(C_{C_1,s}\),
\begin{equation}
\label{eq:fred-risk-intermediate}
\mathbb E\!\left[
\|\hat f_{\mathrm{red},H}-f_{\mathrm{red}}\|^2
\right]
\ \leq \
C_{C_1,s}
\left(
\frac{H}{(1\wedge\alpha^4)N}
+
\frac{H^3}{N^2}
+
(1+H)^{-2s}
\right).
\end{equation}

We now choose
\(H^*:=\left\lfloor((1\wedge\alpha^4)N)^{1/(2s+1)}\right\rfloor\).
The definition of \(H^*\) yields
\[
\frac{H^*}{(1\wedge\alpha^4)N}
\leq
\bigl((1\wedge\alpha^4)N\bigr)^{-2s/(2s+1)},
\qquad
(1+H^*)^{-2s}
\leq
\bigl((1\wedge\alpha^4)N\bigr)^{-2s/(2s+1)}.
\]
Moreover, since \(s>1/2\) and \((1\wedge\alpha^4)N\leq N\), we have
\(H^*\leq N^{1/2}\), and therefore
\((H^*)^3/N^2\leq H^*/N\leq
H^*/((1\wedge\alpha^4)N)\).

Plugging the inequalities above into~\eqref{eq:fred-risk-intermediate}
gives 
\[
\mathbb E\!\left[
\|\hat f_{\mathrm{red},H^*}-f_{\mathrm{red}}\|^2
\right]
\ \leq \ C_{C_1,s}
\bigl((1\wedge\alpha^4)N\bigr)^{-2s/(2s+1)}.
\]
On the other hand, if \((1\wedge\alpha^4)N<1\), then \(H^*=0\), and
\eqref{eq:fred-risk-intermediate} shows that this risk is bounded by a constant
\(C_{C_1,s}\). Combining these estimates, and taking the
supremum over \(f\in\mathcal F(C_0,C_1,s)\), yields the bound stated in
Proposition~\ref{prop:fred-estimator}.

%%%%%%
%%%%%

%%%%%%
% Proof Upper bound THM
%%%%%

\subsection{Proof of Theorem~\ref{thm:f-estimator}} 

Fix \(N\geq2\), \(\alpha>0\), and
\(f\in\mathcal F(C_0,C_1,s)\).
Since \(f=\gamma_f(0)f_{\mathrm{red}}\) and
\(\hat f_{H^*}=\hat\gamma(0)\hat f_{\mathrm{red},H^*}\), we have
\[
\hat f_{H^*}-f
\ = \
\hat\gamma(0)
\bigl(\hat f_{\mathrm{red},H^*}-f_{\mathrm{red}}\bigr)
+
\bigl(\hat\gamma(0)-\gamma_f(0)\bigr)f_{\mathrm{red}}.
\]
Hence,
\[
\|\hat f_{H^*}-f\|^2
\ \leq \
2\hat\gamma(0)^2
\|\hat f_{\mathrm{red},H^*}-f_{\mathrm{red}}\|^2
+
2|\hat\gamma(0)-\gamma_f(0)|^2
\|f_{\mathrm{red}}\|^2.
\]
By construction of the variance estimator,
\(\hat\gamma(0)\in[C_0^{-1},C_0]\). Moreover, by
Lemma~\ref{lem:rho-sum-uniform},
\[
\|f_{\mathrm{red}}\|^2
=
\frac1{2\pi}
\sum_{h\in\mathbb Z}|\rho_f(h)|^2
\leq
\frac{1+C_1^2}{2\pi}.
\]
Taking expectations, we obtain
\[
\mathbb E
\left[
\|\hat f_{H^*}-f\|^2
\right]
\ \leq \
2C_0^2
\mathbb E
\left[
\|\hat f_{\mathrm{red},H^*}-f_{\mathrm{red}}\|^2
\right]
+
\frac{1+C_1^2}{\pi}
\mathbb E
\left[
|\hat\gamma(0)-\gamma_f(0)|^2
\right].
\]

The first term is controlled by
Proposition~\ref{prop:fred-estimator}, while the second term is
controlled by Proposition~\ref{lem:variance-estimator}. Furthermore,
since \((\alpha^2\wedge1)N\geq(\alpha^4\wedge1)N\) and
\(2s/(2s+1)<1\), we have
\[
1 \wedge \bigl((\alpha^2\wedge1)N\bigr)^{-1}
\ \leq \
1 \wedge \bigl((\alpha^4\wedge1)N\bigr)^{-1}
\ \leq \
1 \wedge \bigl((\alpha^4\wedge1)N\bigr)^{-2s/(2s+1)}.
\]
Combining these bounds and taking the supremum over \(f\in\mathcal F(C_0,C_1,s)\) yields the bound stated in Theorem~\ref{thm:f-estimator}.
 
%%% Local Variables:
%%% mode: latex
%%% TeX-master: "main.tex"
%%% ispell-local-dictionary: "american"
%%% End:

%%%%%%%%%%%%%%%%%

{\footnotesize
\bibliography{biblio}
\bibliographystyle{alpha}
}

%%%%%%%%%%%%

\appendix

%%%%%%

\section{Standard facts on conditional distributions and KL divergence}
\label{appendix:KL-properties}

We collect standard notation and identities for kernels and KL divergence used throughout the proofs.

\subsection{Random-variable formulation}
\label{appendix-notations:random-variables}

The notation and properties recalled below follow~\cite{polyanskiy2025information}.

\textit{Disintegration.}
Let \(P^{X,Y}\) be a probability measure on
\(\mathcal X\times\mathcal Y\).
We denote its marginal distributions on \(\mathcal X\) and
\(\mathcal Y\) by \(P^X\) and \(P^Y\), respectively; that is,
\(P^X=P^{X,Y}(\cdot\times\mathcal Y)\) and
\(P^Y=P^{X,Y}(\mathcal X\times\cdot)\).
We further denote by \(P^{Y\mid X}\) a Markov kernel from
\(\mathcal X\) to \(\mathcal Y\) representing the conditional
distribution of \(Y\) given \(X\), meaning that
\[
P^{X,Y}(A\times B)
=
\int_A P^{Y\mid X}(B\mid x)\,P^X(dx)
\]
for all measurable sets \(A\subseteq\mathcal X\) and
\(B\subseteq\mathcal Y\).
We write this identity more succinctly as
\(P^{X,Y}=P^X\otimes P^{Y\mid X}\).
When \(\mathcal Y\) is standard Borel, such a Markov kernel exists and
is unique \(P^X\)-almost surely.

\smallskip 
%%%%%%%%%%

\textit{Standard properties of the KL divergence.}
Recall that the conditional KL divergence was defined in Section~\ref{subsection:condtional-KL-contraction}.

\begin{lemma}
\label{lem:standard-properties-kl}
The statements below hold true, assuming that the conditional KL
divergence is well defined for the chain rule.
\begin{enumerate}[topsep=1pt,itemsep=1pt,label=(\roman*)]
    \item\label{item:chain-rule-va} \emph{Chain rule:}
    \[
    \D(P^{X,Y} \,\|\, Q^{X,Y})
    \ = \
    \D\!\left(P^{Y\mid X} \,\|\, Q^{Y\mid X} \mid P^X\right)
    +
    \D(P^X \,\|\, Q^X).
    \]

    \item\label{item:tensor-kl-va} \emph{Tensorization:}
    \[
    \D\!\left(
    \bigotimes_{j=1}^n P^{X_j}
    \,\middle\|\,
    \bigotimes_{j=1}^n Q^{X_j}
    \right)
    \ = \ 
    \sum_{j=1}^n
    \D(P^{X_j} \,\|\, Q^{X_j}).
    \]
\end{enumerate}
\end{lemma}

%%%%%%%%%%%
%%%%%%%%%%%

\subsection{Measure-kernel formulation}
\label{appendix-notations:measure-kernel}

\textit{Product of measures and kernels.}
If \(\mathcal X\) and \(\mathcal Z\) are measurable spaces, \(K\) is a Markov kernel from \(\mathcal X\) to \(\mathcal Z\), and \(P\) is a probability measure on \(\mathcal X\), we write \(P\otimes K\) for the probability measure on \(\mathcal X\times\mathcal Z\) defined by
\[
(P\otimes K)(C)
=
\int_{\mathcal X}\int_{\mathcal Z}
\mathbf 1_C(x,z)\,K(dz\mid x)\,P(dx),
\]
for every measurable set \(C\subseteq\mathcal X\times\mathcal Z\).
For products of measures and kernels, see, for instance,~\cite{kallenberg2021foundations}.

\smallskip 

\textit{Conditional KL divergences via kernels.} In
Section~\ref{subsection:condtional-KL-contraction}, we defined the
conditional KL through the disintegrations of two joint
distributions. However, given two kernels $K_1$ and $K_2$ from
a measurable space $\mathcal X$ to a standard Borel space
$\mathcal Z$, and a probability measure $P$ defined on $\mathcal X$,
we can directly define the conditional KL divergence between $K_1$ and
$K_2$ with respect to $P$ as
\[
\D\!\left(
K_1
\,\middle\|\,
K_2
\,\middle|\,
P
\right)
\ := \
\int_{\mathcal X}
\D\!\left(
K_1(\cdot\mid x)
\,\middle\|\,
K_2(\cdot\mid x)
\right)
P(dx).
\]
Then the chain rule property~\ref{item:chain-rule-va} in Lemma~\ref{lem:standard-properties-kl} takes the following form.

\begin{lemma}[Chain rule via kernels]
\label{lem:standard-properties-kl-kernels}
Let \(P_1\) and \(P_2\) be probability measures on a measurable space
\(\mathcal X\), and let \(K_1\) and \(K_2\) be Markov kernels from
\(\mathcal X\) to a standard Borel space \(\mathcal Z\).
Then we have
    \[
    \D(P_1\otimes K_1\|P_2\otimes K_2)
    \ = \
    \D(P_1\|P_2)+\D(K_1\|K_2\mid P_1).
    \]
\end{lemma}

%%% Local Variables:
%%% mode: latex
%%% TeX-master: "main.tex"
%%% ispell-local-dictionary: "american"
%%% End:

%%%%%%%

\section{Proof of Proposition~\ref{prop:shared-latent-reduction}}
\label{sec:proof-reduction-general}

We use throughout this appendix the measure-kernel notation introduced in Appendix~\ref{appendix-notations:measure-kernel}.

%%%%%%
%%%%%%%%

\subsection{Auxiliary lemmas}

The proof of Proposition~\ref{prop:shared-latent-reduction} relies on the following two elementary lemmas.
The first shows that averaging over the same input distribution cannot increase the KL divergence.

\begin{lemma}
\label{newlem-conditioning-shared-latent}
Let \(\nu\) be a probability measure on a measurable space \(\mathcal V\), and let
\(M_1\) and \(M_2\) be two Markov kernels from \(\mathcal V\) to a
standard Borel space \(\mathcal Z\). Then
\[
\D\!\left(
M_1\circ\nu
\,\middle\|\,
M_2\circ\nu
\right)
\ \leq \
\D\!\left(
M_1
\,\middle\|\,
M_2
\,\middle|\,
\nu
\right).
\]
\end{lemma}

\begin{proof}[Proof of Lemma~\ref{newlem-conditioning-shared-latent}]
If \(\D(M_1\|M_2\mid\nu)=+\infty\), the result is immediate.
We may therefore assume that
\(\D(M_1\|M_2\mid\nu)<+\infty\).
By
Lemma~\ref{lem:standard-properties-kl-kernels}
and since $\D(\nu\|\nu)=0$, this implies that
\(\D(\nu\otimes M_1\|\nu\otimes M_2)<+\infty\).

For \(i\in[2]\), the marginal on \(\mathcal Z\) of \(\nu\otimes M_i\) is \(M_i\circ\nu\).
Therefore,
\[
\D\!\left(
M_1\circ\nu
\,\middle\|\,
M_2\circ\nu
\right)
\ \leq \
\D\!\left(
\nu\otimes M_1
\,\middle\|\,
\nu\otimes M_2
\right)
\ = \
\D(\nu\|\nu)
+
\D\!\left(
M_1
\,\middle\|\,
M_2
\,\middle|\,
\nu
\right)
\ = \
\D\!\left(
M_1
\,\middle\|\,
M_2
\,\middle|\,
\nu
\right).
\]
The inequality follows from the data processing inequality in
Assertion~\ref{item:lem:contraction-key3} of
Lemma~\ref{lem:contraction-key}, applied to the deterministic
projection kernel from \(\mathcal V\times\mathcal Z\) onto
\(\mathcal Z\). 
The first equality follows from the chain rule in
Lemma~\ref{lem:standard-properties-kl-kernels}.
\end{proof}

%%%%%%%%%
\medskip 
%%%%%

The second lemma shows that composing the input of an \(\alpha\)-LDP
kernel with a measurable function preserves the \(\alpha\)-LDP
property.

\begin{lemma}
\label{newlem-induced-private-mechanism}
Let \(\mathcal U,\mathcal V,\mathcal X,\mathcal Z\) be measurable
spaces, let \(g:\mathcal U\times\mathcal V\to\mathcal X\) be
a measurable function, and let \(K\) be a Markov kernel from \(\mathcal X\) to
\(\mathcal Z\). For every \(v\in\mathcal V\), define $K_g^{(v)}$ by
\begin{equation}
\label{def-kernel_f}
K_g^{(v)}(A\mid u)\ :=\ K(A\mid g(u,v)),
\qquad \qquad 
u\in\mathcal U,\ A\subseteq\mathcal Z \text{ measurable}.
\end{equation}
If \(K\in\mathcal M_\alpha(\mathcal X)\), then
\[
K_g^{(v)}
\in
\mathcal M_\alpha(\mathcal U)
\quad\text{for all} \quad 
v\in\mathcal V.
\]

More generally, assume that \(\mathcal U=\mathcal U_1\times\cdots\times\mathcal U_N\),
\(\mathcal X=\mathcal X_1\times\cdots\times\mathcal X_N\),
\(\mathcal Z=\mathcal Z_1\times\cdots\times\mathcal Z_N\),
and
\[
g(u,v)
=
\bigl(
g_1(u_1,v),\ldots,g_N(u_N,v)
\bigr),
\qquad \qquad 
u=(u_1,\ldots,u_N)\in\mathcal U,\ v\in\mathcal V,
\]
where each
\(g_i:\mathcal U_i\times\mathcal V\to\mathcal X_i\) is measurable.
Let \(K_{1:N}=(K_1,\ldots,K_N)\), where each \(K_i\) is a Markov
kernel from \(\mathcal X_i\) to \(\mathcal Z_i\).
For every \(v\in\mathcal V\) and \(i\in[N]\), define \(K_{g,i}^{(v)}\) by
\[
K_{g,i}^{(v)}(A_i\mid u_i)
\ := \
K_i(A_i\mid g_i(u_i,v)),
\qquad \qquad
u_i\in\mathcal U_i,\
A_i\subseteq\mathcal Z_i \text{ measurable},
\]
and \(K_{g,1:N}^{(v)}
 := 
\bigl(
K_{g,1}^{(v)},\ldots,K_{g,N}^{(v)}
\bigr)\).
If
\(K_{1:N}\in \mathcal M_\alpha^N(\mathcal X)\),
then 
\[
K_{g,1:N}^{(v)}
\in
\mathcal M_\alpha^N(\mathcal U)
\quad\text{for all} \quad 
v\in\mathcal V.
\]
\end{lemma}

\begin{proof}[Proof of Lemma~\ref{newlem-induced-private-mechanism}]
Fix \(v\in\mathcal V\).
Since \(u\mapsto g(u,v)\) is measurable,
\(K_g^{(v)}\) is a Markov kernel from
\(\mathcal U\) to \(\mathcal Z\).
Moreover, for all \(u,\tilde u\in\mathcal U\) and all measurable
\(A\subseteq\mathcal Z\),
\[
K_g^{(v)}(A\mid u)
\ = \
K(A\mid g(u,v))
\ \leq \
e^\alpha K(A\mid g(\tilde u,v))
\ = \
e^\alpha K_g^{(v)}(A\mid\tilde u),
\]
because \(K\) is \(\alpha\)-locally differentially private.
Hence
\(K_g^{(v)}\in\mathcal M_\alpha(\mathcal U)\).

For the \(N\)-coordinate case, applying the preceding argument to
\(K_i\) and \(g_i\), for each \(i\in[N]\), yields
\(K_{g,i}^{(v)}\in\mathcal M_\alpha(\mathcal U_i)\). Therefore,
\(K_{g,1:N}^{(v)}\in\mathcal M_\alpha^N(\mathcal U)\).
\end{proof}

%%%%%%%%
%\medskip
%%%%%%%%

\subsection{Main proof}

We now combine the two lemmas above.

\begin{proof}[Proof of Proposition~\ref{prop:shared-latent-reduction}]
Fix \(K\in\mathcal M_\alpha(\mathcal X)\), and let
\(K_g^{(v)}\) be defined by~\eqref{def-kernel_f}.
Define \(M_P\) and \(M_Q\) from \(\mathcal V\) to \(\mathcal Z\) by
\[
M_P(\cdot\mid v)
\ := \
K_g^{(v)}\circ P^U,
\qquad
M_Q(\cdot\mid v)
\ := \
K_g^{(v)}\circ Q^U.
\]
These are Markov kernels.
Indeed, for every measurable \(A\subseteq\mathcal Z\),
\[
M_P(A\mid v)
\ = \
\int_{\mathcal U}
K(A\mid g(u,v))\,P^U(du).
\]
The map \((u,v)\mapsto K(A\mid g(u,v))\) is nonnegative and
measurable, so Tonelli's theorem shows that
\(v\mapsto M_P(A\mid v)\) is measurable.
The same argument applies to \(M_Q\).

For every measurable \(A\subseteq\mathcal Z\), the identity
\(P^{U,V}=P^U\otimes P^V\) gives
\[
\begin{aligned}
\left(
K\circ P^{g(U,V)}
\right)(A)
\ = \
\int_{\mathcal U\times\mathcal V}
K(A\mid g(u,v))\,
P^U(du)\,P^V(dv)
&\ = \
\int_{\mathcal V}
M_P(A\mid v)\,P^V(dv)
\\
&\ = \
\left(
M_P\circ P^V
\right)(A).
\end{aligned}
\]
Hence
\(K\circ P^{g(U,V)} = M_P\circ P^V\).
Similarly, we obtain
\(K\circ Q^{g(U,V)} = M_Q\circ P^V\).

Lemma~\ref{newlem-conditioning-shared-latent}, applied with
\(\nu=P^V\), gives
\[
\begin{aligned}
\D\!\left(
K\circ P^{g(U,V)}
\,\middle\|\,
K\circ Q^{g(U,V)}
\right)
&\ \leq \
\D\!\left(
M_P
\,\middle\|\,
M_Q
\,\middle|\,
P^V
\right)
\\
&\ = \
\int_{\mathcal V}
\D\!\left(
K_g^{(v)}\circ P^U
\,\middle\|\,
K_g^{(v)}\circ Q^U
\right)
P^V(dv).
\end{aligned}
\]
By Lemma~\ref{newlem-induced-private-mechanism},
\(K_g^{(v)}\in\mathcal M_\alpha(\mathcal U)\) for every
\(v\in\mathcal V\). Therefore,
\[
\D\!\left(
K\circ P^{g(U,V)}
\,\middle\|\,
K\circ Q^{g(U,V)}
\right)
\ \leq \
\sup_{\widetilde K\in\mathcal M_\alpha(\mathcal U)}
\D\!\left(
\widetilde K\circ P^U
\,\middle\|\,
\widetilde K\circ Q^U
\right).
\]
Since \(K\in\mathcal M_\alpha(\mathcal X)\) is arbitrary, taking the
supremum over \(K\) yields the first assertion of
Proposition~\ref{prop:shared-latent-reduction}.

The proof of the \(N\)-coordinate assertion is identical, replacing \(K\) and \(K_g^{(v)}\) by \(K_{1:N}\) and \(K_{g,1:N}^{(v)}\), respectively.
This completes the proof of Proposition~\ref{prop:shared-latent-reduction}.
\end{proof}

%%% Local Variables:
%%% mode: latex
%%% TeX-master: "main.tex"
%%% ispell-local-dictionary: "american"
%%% End:

%%%%%%%%

\section{Proof of Lemma~\ref{lem:oriented-kl-contraction-ldp}}
\label{appendix:proof-kl-contraction-ldp}

Unlike the proof in~\cite{duchi2018minimax}, our argument does not rely
on conditional densities.
We use the notation \(P\otimes K\) for products of probability measures
\(P\) and Markov kernels \(K\), as introduced in
Appendix~\ref{appendix-notations:measure-kernel}.

%%%%%%%

\paragraph{Setting.}
Let \(P\) and \(Q\) be two arbitrary probability measures on the same measurable space \(\mathcal X\). If \(\D(P\|Q)=+\infty\), the contraction inequality of
Lemma~\ref{lem:oriented-kl-contraction-ldp} is immediate. 
We may therefore assume that \(\D(P\|Q)<+\infty\), which in particular implies that \(P\ll Q\).
Let \(K\) be a Markov kernel from \(\mathcal X\) to another measurable space \(\mathcal Z\).
Let
\(X:\mathcal X\times\mathcal Z\to\mathcal X\) and
\(Z:\mathcal X\times\mathcal Z\to\mathcal Z\) denote the canonical
coordinate maps, defined by
\[
X(x,z)=x
\qquad\text{and}\qquad
Z(x,z)=z,
\qquad (x,z)\in\mathcal X\times\mathcal Z.
\]
We denote by \(\mathbb Q\) the probability measure \(Q\otimes K\) on \(\mathcal X\times\mathcal Z\), and by \(\mathbb E_{\mathbb Q}\) the corresponding expectation.
Finally, define \(\Psi:[0,\infty)\to\mathbb R\) by \(\Psi(x)=x\log x\), with \(\Psi(0)=0\). 
Since \(X\) has distribution \(Q\) under \(\mathbb Q\), we have
\begin{equation}
\label{eq:kullback-with-phi}
\begin{aligned}
\D(P\|Q)
\ = \
\int_{\mathcal X}
\log\left(\frac{dP}{dQ}(x)\right)\,P(dx)
\ = \
\int_{\mathcal X}
\Psi\left(\frac{dP}{dQ}(x)\right)\,Q(dx)
\ = \
\mathbb E_{\mathbb Q}\left[
\Psi\left(\frac{dP}{dQ}(X)\right)
\right].
\end{aligned}
\end{equation}

%%%%
%\medskip
%%%%

\subsection{Auxiliary lemmas}

We first present two auxiliary lemmas used in the proof of Lemma~\ref{lem:oriented-kl-contraction-ldp}.
The following lemma establishes two likelihood-ratio identities and the data processing inequality for a Markov kernel defined between general
measurable spaces.

\begin{lemma}
\label{lem:contraction-key}
The following statements hold.
\begin{enumerate}[label=(\roman*), series=kl-contraction-ass]
    
\item\label{item:lem:contraction-key1}
We have \(P\otimes K\ll Q\otimes K\), and the Radon--Nikodym
derivative satisfies
\[
\frac{d(P\otimes K)}{d(Q\otimes K)}(X,Z)
\ = \
\frac{dP}{dQ}(X),
\qquad \mathbb Q\text{-a.s.}
\]

\item\label{item:lem:contraction-key2}
We have \(K\circ P\ll K\circ Q\), and the Radon--Nikodym derivative
satisfies
\[
\frac{d(K\circ P)}{d(K\circ Q)}(Z)
\ = \
\mathbb E_{\mathbb Q}\left[
\frac{dP}{dQ}(X)\,\middle|\,Z
\right],
\qquad \mathbb Q\text{-a.s.}
\]

\item\label{item:lem:contraction-key3}
We have
\(\D(K\circ P\|K\circ Q)
\ \le \
\D(P\|Q)\).

\end{enumerate}
\end{lemma}

\begin{proof}
The three assertions above are proved sequentially:
Assertion~\ref{item:lem:contraction-key1} is used to prove Assertion~\ref{item:lem:contraction-key2}, which in turn yields
Assertion~\ref{item:lem:contraction-key3} through conditional Jensen's inequality.

Let \(A\subseteq\mathcal X\) and \(B\subseteq\mathcal Z\) be measurable.  
Then
\begin{align*}
(P\otimes K)(A\times B)
=
\int_{\mathcal X}
\mathbf 1_A(x)K(B\mid x)\,P(dx)
&=
\int_{\mathcal X}
\mathbf 1_A(x)K(B\mid x)
\frac{dP}{dQ}(x)\,Q(dx)
\\
&=
\int_{\mathcal X\times\mathcal Z}
\mathbf 1_{A\times B}(x,z)
\frac{dP}{dQ}(x)\,
(Q\otimes K)(dx,dz).
\end{align*}
Since probability measures on \(\mathcal X\times\mathcal Z\) are
uniquely determined by their values on products of measurable sets,
Assertion~\ref{item:lem:contraction-key1} follows.

Taking \(A=\mathcal X\) in the previous display, and recalling that
\(\mathbb Q:=Q\otimes K\), we have
\[
\begin{aligned}
(K\circ P)(B)
&\ = \
\mathbb E_{\mathbb Q}\left[
\mathbf 1_B(Z)\frac{dP}{dQ}(X)
\right] \\
&\ = \ \mathbb E_{\mathbb Q}\left[
\mathbf 1_B(Z)
\mathbb E_{\mathbb Q}\left[
\frac{dP}{dQ}(X)\,\middle|\,Z
\right]
\right]
\end{aligned}
\]
where the second equality follows by conditioning on \(Z\).
Since \(Z\) has distribution \(K\circ Q\) under \(\mathbb Q\),
this proves Assertion~\ref{item:lem:contraction-key2}.

To prove Assertion~\ref{item:lem:contraction-key3}, we first condition the
expectation in~\eqref{eq:kullback-with-phi} on \(Z\):
\[
\begin{aligned}
\D(P\|Q)
&\ = \
\mathbb E_{\mathbb Q}\left[
\mathbb E_{\mathbb Q}\left[
\Psi\left(\frac{dP}{dQ}(X)\right)
\,\middle|\,Z
\right]
\right]\\
&\ \ge \
\mathbb E_{\mathbb Q}\left[
\Psi\left(
\mathbb E_{\mathbb Q}\left[
\frac{dP}{dQ}(X)
\,\middle|\,Z
\right]
\right)
\right],
\end{aligned}
\]
where the inequality follows from conditional Jensen's inequality, since
\(\Psi\) is convex.
Assertion~\ref{item:lem:contraction-key2} then yields
\[
\begin{aligned}
\D(P\|Q)
\ \ge \
\mathbb E_{\mathbb Q}\left[
\Psi\left(
\frac{d(K\circ P)}{d(K\circ Q)}(Z)
\right)
\right].
\end{aligned}
\]
Using again that  \(Z\) has distribution \(K\circ Q\) under \(\mathbb
Q\), the right-hand side in this inequality is $\D(K\circ P\|K\circ
Q)$, which gives us
Assertion~\ref{item:lem:contraction-key3},
and completes the proof of Lemma~\ref{lem:contraction-key}.
\end{proof}

%%%%%%%
\bigskip 
%%%%%%%%%%%

The following lemma bounds the deviation from \(1\) of the likelihood
ratio \(d(K\circ P)/d(K\circ Q)\), in terms of the total variation
distance
\begin{equation}
    \label{defi-TV}
\|P-Q\|_{\mathrm{TV}}
\ := \
\sup_{\substack{A\subseteq\mathcal X\\ A\text{ measurable}}}
|P(A)-Q(A)|.
\end{equation}

\begin{lemma}
\label{lem:contraction-key-duchi-ingredient}
Suppose that there exists \(\alpha>0\) such that, for all
\(x,x'\in\mathcal X\) and all measurable \(A\subseteq\mathcal Z\),
\begin{equation}
  \label{eq:ldp-cond}
K(A\mid x)
\ \le \
e^\alpha \, K(A\mid x').  
\end{equation}
Then,
\[
\left|
\frac{d(K\circ P)}{d(K\circ Q)}(Z)-1
\right|
\ \le \
\left(e^\alpha-e^{-\alpha}\right)
\|P-Q\|_{\mathrm{TV}},
\qquad
\mathbb Q\text{-a.s.}
\]
\end{lemma}

\begin{proof}
Recall that, for every measurable \(A\subseteq\mathcal Z\),
\[
(K\circ Q)(A)
\ = \
\int_{\mathcal X}K(A\mid x')\,Q(dx').
\]
Integrating~\eqref{eq:ldp-cond} with respect to \(Q(dx')\) for fixed
\(x\), and with respect to \(Q(dx)\) for fixed \(x'\), yields, for all
\(x\in\mathcal X\) and measurable \(A\subseteq\mathcal Z\),
\begin{equation}
\label{eq:first-obs-key-duchi}
e^{-\alpha}(K\circ Q)(A)
\ \le \
K(A\mid x)
\ \le \
e^\alpha(K\circ Q)(A).
\end{equation}

Let \(g:\mathcal X\to\mathbb R_+\) be a measurable \(Q\)-integrable function, and let \(A\subseteq\mathcal Z\) be measurable. 
Using the definition of \(\mathbb Q=Q\otimes K\) and~\eqref{eq:first-obs-key-duchi}, we obtain
\[
\begin{aligned}
e^{-\alpha}(K\circ Q)(A)
\int_{\mathcal X}g(x)\,Q(dx)
&\ \le \
\mathbb E_{\mathbb Q}\left[g(X)\mathbf 1_A(Z)\right]
&\ \le \
e^\alpha(K\circ Q)(A)
\int_{\mathcal X}g(x)\,Q(dx).
\end{aligned}
\]
Since \(X\) has distribution \(Q\), and \(Z\) has distribution \(K\circ Q\) under \(\mathbb Q\), this can be rewritten as
\[
\begin{aligned}
e^{-\alpha}
\mathbb E_{\mathbb Q}[g(X)]
\mathbb E_{\mathbb Q}[\mathbf 1_A(Z)]
&\ \le \
\mathbb E_{\mathbb Q}[g(X)\mathbf 1_A(Z)]
&\ \le \
e^\alpha
\mathbb E_{\mathbb Q}[g(X)]
\mathbb E_{\mathbb Q}[\mathbf 1_A(Z)].
\end{aligned}
\]
Since this holds for every measurable \(A\subseteq\mathcal Z\), the characterization of conditional expectation yields
\begin{equation}
\label{eq:conditional-expectation-ldp-bound}
e^{-\alpha}\mathbb E_{\mathbb Q}[g(X)]
\ \le \
\mathbb E_{\mathbb Q}[g(X)\mid Z]
\ \le \
e^\alpha\mathbb E_{\mathbb Q}[g(X)],
\qquad
\mathbb Q\text{-a.s.}
\end{equation}
Denote by $g_+$ and $g_-$ the positive and negative parts of $dP/dQ-1$
so that by Assertion~\ref{item:lem:contraction-key2} in Lemma~\ref{lem:contraction-key}, we have, \(\mathbb Q\)-almost surely,
\[
\begin{aligned}
\frac{d(K\circ P)}{d(K\circ Q)}(Z)-1
&\ = \
\mathbb E_{\mathbb Q}\left[
\frac{dP}{dQ}(X)-1
\,\middle|\,Z
\right]
\\
&\ = \
\mathbb E_{\mathbb Q}[g_+(X)\mid Z]
-
\mathbb E_{\mathbb Q}[g_-(X)\mid Z].
\end{aligned}
\]
Applying~\eqref{eq:conditional-expectation-ldp-bound} to \(g_+\) and
\(g_-\), we obtain, \(\mathbb Q\)-almost surely,
\[
e^{-\alpha}\mathbb E_{\mathbb Q}[g_+(X)]
-
e^\alpha\mathbb E_{\mathbb Q}[g_-(X)]
\ \le \
\frac{d(K\circ P)}{d(K\circ Q)}(Z)-1
\leq
e^\alpha\mathbb E_{\mathbb Q}[g_+(X)]
-
e^{-\alpha}\mathbb E_{\mathbb Q}[g_-(X)].
\]
Since \(X\) has distribution \(Q\) under \(\mathbb Q\), we have $\mathbb E_{\mathbb Q}[g_+(X)]
-\mathbb E_{\mathbb Q}[g_-(X)]
=\mathbb E_{\mathbb Q}\left[\frac{dP}{dQ}(X)-1\right] = 0$
and thus
\[
\mathbb E_{\mathbb Q}[g_+(X)]
\ = \
\mathbb E_{\mathbb Q}[g_-(X)]
\ = \frac12\mathbb E_{\mathbb Q}\left[\left|\frac{dP}{dQ}(X)-1\right|\right] =
\|P-Q\|_{\mathrm{TV}}.
\]
Inserting this in the previous display, we get the claimed inequality,
which concludes the proof of Lemma~\ref{lem:contraction-key-duchi-ingredient}.
\end{proof}

%%%%%
%\bigskip
%%%%%%%

\subsection{Main proof}

We now combine the two lemmas above.

\begin{proof}[Proof of Lemma~\ref{lem:oriented-kl-contraction-ldp}]
First observe that, by definition of $\Psi$ and since $\log(x)\leq x-1$, we have for all $x\in\mathbb R_+$,
\begin{equation}
  \label{eq:Phi-bound-trivial}
\Psi(x) 
\ \leq \ 
x(x-1)
\ = \ 
x-1+(x-1)^2.  
\end{equation}
Now, by Assertion~\ref{item:lem:contraction-key2} in
Lemma~\ref{lem:contraction-key}, we have \(K\circ P\ll K\circ Q\) so
that the likelihood ratio
\[
L
\ := \
\frac{d(K\circ P)}{d(K\circ Q)}(Z)
\]
is well defined.
Since \(Z\) has distribution \(K\circ Q\) under \(\mathbb Q\), the same calculation as in~\eqref{eq:kullback-with-phi} gives
\(\D(K\circ P\|K\circ Q)
=
\mathbb E_{\mathbb Q}[\Psi(L)]\).
Using~\eqref{eq:Phi-bound-trivial} and
\(\mathbb E_{\mathbb Q}[L]=1\), we conclude that
\[
\D(K\circ P\|K\circ Q)
\ \leq \ 
\mathbb E_{\mathbb Q}\left[(L-1)^2\right]\;.
\]
Lemma~\ref{lem:contraction-key-duchi-ingredient} therefore yields
\[
\begin{aligned}
\D(K\circ P\|K\circ Q)
&\ \le \
\left(e^\alpha-e^{-\alpha}\right)^2
\|P-Q\|_{\mathrm{TV}}^2
\\
&\ \le \
\frac12\left(e^\alpha-e^{-\alpha}\right)^2
\D(P\|Q),
\end{aligned}
\]
where the second inequality follows from Pinsker's inequality.
Combining this bound with Assertion~\ref{item:lem:contraction-key3}, we obtain
\[
\D(K\circ P\|K\circ Q)
\ \le \
\left(
\frac12\left(e^\alpha-e^{-\alpha}\right)^2\wedge1
\right)
\D(P\|Q).
\]
This concludes the proof of Lemma~\ref{lem:oriented-kl-contraction-ldp}.
\end{proof}

%%% Local Variables:
%%% mode: latex
%%% TeX-master: "main.tex"
%%% ispell-local-dictionary: "american"
%%% End:

%%%%%%%%%%%%%%%%

\section{Proofs of Lemma~\ref{lem:general-contraction-with-contiionning} and Proposition~\ref{prop:abstract-kl-iteration}}
\label{appendix:proof-kl-iteration-1-dependent-process}

%%%%%
%Appendix C.1
%%%%%%%

\subsection{Proof of Lemma~\ref{lem:general-contraction-with-contiionning}}

By the disintegration result recalled in Appendix~\ref{appendix-notations:random-variables}, since \(\mathcal X\) and \(\mathcal Z\) are standard Borel spaces, there
exist regular versions (\textit{i.e.} Markov kernels)
\(P_\pm^{X\mid Y}\) and \(P_\pm^{Z\mid Y}\) of the conditional
distributions of \(X\) given \(Y\) and of \(Z\) given \(Y\) under
\(P_\pm\), respectively; these kernels are unique \(P_\pm^Y\)-almost surely.  
Writing \(\mathbb E_\pm\) for expectation under \(P_\pm\),
the disintegration formula gives, for every nonnegative measurable
function \(\varphi:\mathcal X\times\mathcal Y\to[0,\infty]\),
\begin{equation}
\label{desintegration}
\mathbb E_\pm\!\left[\varphi(X,Y)\right]
\ = \
\int_{\mathcal Y}\pa{\int_{\mathcal X}
\varphi(x,y)\,
P_\pm^{X\mid Y=y}(dx)}\,P_\pm^Y(dy).
\end{equation}

Let \(A\subseteq\mathcal Z\) and \(B\subseteq\mathcal Y\) be measurable.
We have
\begin{align*}
\int_B
\left(K\circ P_\pm^{X\mid Y=y}\right)(A)\,P_\pm^Y(dy)
&\ = \
\int_B \left(\int_{\mathcal X}
K(A\mid x)\,P_\pm^{X\mid Y=y}(dx)\right)P_\pm^Y(dy) \\
&\ = \
\mathbb E_\pm\!\left[
\mathbf 1_{\{Y\in B\}}K(A\mid X)
\right] \\
&\ = \
\mathbb E_\pm\!\left[
\mathbf 1_{\{Y\in B\}}
P_\pm(Z\in A\mid X,Y)
\right] \\
&\ = \
P_\pm(Z\in A,Y\in B),
\end{align*}
where the first equality follows from the definition of kernel composition, the second from~\eqref{desintegration} applied to \(\varphi(x,y)=\mathbf 1_B(y)K(A\mid x)\), 
the third from the assumption on the conditional distribution of \(Z\) given \((X,Y)\), and the last from the tower property.
The obtained equality shows that \(y\mapsto K\circ P_\pm^{X\mid Y=y}\) satisfies the disintegration identity for the conditional distribution of \(Z\) given \(Y\) under
\(P_\pm\); hence this kernel is a version of that conditional distribution.
By uniqueness of this representation,
\begin{equation}
\label{eq:condtional-ditrib:Z-Y}
P_\pm^{Z\mid Y=y}
\ = \
K\circ P_\pm^{X\mid Y=y}
\quad\text{for}\quad
P_\pm^Y\text{-almost every }y.
\end{equation}

For each \(y\in\mathcal Y\) satisfying the equalities
in~\eqref{eq:condtional-ditrib:Z-Y}, setting \(\mu_\pm^y :=
P_\pm^{X\mid Y=y}\), and using the \(\beta\)-contraction property of
\(K\) yields
\[
\D\!\left(
P_+^{Z\mid Y=y}
\,\middle\|\,
P_-^{Z\mid Y=y}
\right)
\ = \
\D\!\left(
K\circ\mu_+^y
\,\middle\|\,
K\circ\mu_-^y
\right)
\ \le \
\beta\,
\D\!\left(
\mu_+^y
\,\middle\|\,
\mu_-^y
\right)
\ = \
\beta\,
\D\!\left(
P_+^{X\mid Y=y}
\,\middle\|\,
P_-^{X\mid Y=y}
\right).
\]
Since \(P_+^Y\ll P_-^Y\), the identity
in~\eqref{eq:condtional-ditrib:Z-Y} for \(P_-^{Z\mid Y=y}\), which
holds \(P_-^Y\)-almost surely, also holds \(P_+^Y\)-almost surely.
Consequently, the identities for \(P_+^{Z\mid Y=y}\) and
\(P_-^{Z\mid Y=y}\) in~\eqref{eq:condtional-ditrib:Z-Y} hold
simultaneously for \(P_+^Y\)-almost every \(y\).
Hence, we
can integrate the previous inequality with respect to \(P_+^Y\), which
gives
\[
\D\!\left(
P_+^{Z\mid Y}
\,\middle\|\,
P_-^{Z\mid Y}
\,\middle|\,
P_+^Y
\right)
\ \le \
\beta\,
\D\!\left(
P_+^{X\mid Y}
\,\middle\|\,
P_-^{X\mid Y}
\,\middle|\,
P_+^Y
\right).
\]
This concludes the proof of Lemma~\ref{lem:general-contraction-with-contiionning}.

%%%%%%
%Appendix C.2
%%%%%%%

\subsection{Proof of Proposition~\ref{prop:abstract-kl-iteration}}

\noindent\textit{Shorthand notation.}
For the KL and conditional KL divergences defined at the beginning of Section~\ref{section:element_proofs} and in Section~\ref{subsection:condtional-KL-contraction}, respectively, we use the following shorthand: for generic random variables \(U\) and \(V\) defined under both \(P_+\) and
\(P_-\),
\begin{equation*}
\D(U)
\ := \
\D\!\left(
P_+^U
\,\middle\|\,
P_-^U
\right),
\qquad
\D(U\mid V)
\ := \
\D\!\left(
P_+^{U\mid V}
\,\middle\|\,
P_-^{U\mid V}
\,\middle|\,
P_+^V
\right),
\end{equation*}
where the conditional KL divergence is defined only when \(P_+^V\ll P_-^V\).
For \(i\le j\), we write
\(X_{i:j}:=(X_i,\ldots,X_j)\) and
\(Z_{i:j}:=(Z_i,\ldots,Z_j)\), with the convention that these blocks
are empty when \(i>j\), and that conditioning on an empty block amounts
to no conditioning.

%%%%%%
\medskip 

If the right-hand side of the inequality in
Proposition~\ref{prop:abstract-kl-iteration} is infinite, there is
nothing to prove. We therefore assume that it is finite.

%%%%%%%

\subsubsection{Auxiliary lemmas}

The proof of Proposition~\ref{prop:abstract-kl-iteration} relies on the
following three auxiliary lemmas.
The first ensures that the conditional KL divergences used below are
well defined.

\begin{lemma}
\label{lem:conditioning-domination}
For every \(1\le j\le t\le N\),
\[
P_+^{(Z_{1:(j-1)},X_{(j+1):t})}
\ll
P_-^{(Z_{1:(j-1)},X_{(j+1):t})}.
\]
\end{lemma}

\begin{proof}
  Since the right-hand side of the inequality in
  Proposition~\ref{prop:abstract-kl-iteration} is finite, and since
  conditional KL divergences are nonnegative, we have
  \(\D(X_0\mid X_{1:k})<\infty\) for every \(k=0,\ldots,N-1\).  By the
  chain rule property~\ref{item:chain-rule-va} in Lemma~\ref{lem:standard-properties-kl} and
  stationarity,
\[
\D(X_{1:N})
=
\sum_{j=1}^N
\D(X_j\mid X_{j+1:N})
=
\sum_{j=1}^N
\D(X_0\mid X_{1:(N-j)})
<\infty.
\]
Hence
\(P_+^{X_{1:N}}\ll P_-^{X_{1:N}}\).

Fix \(1\le j\le t\le N\).
In the given setting, there exists a Markov kernel \(L_{j,t}\) such that
\[
P_\pm^{(Z_{1:(j-1)},X_{(j+1):t})}
=
L_{j,t}\circ P_\pm^{X_{1:N}}.
\]
Therefore, since \(P_+^{X_{1:N}}\ll P_-^{X_{1:N}}\), Lemma~\ref{lem:contraction-key}~\ref{item:lem:contraction-key2} yields the desired domination.
\end{proof}

%%%%%%
\bigskip 

The second is an application of
Lemma~\ref{lem:general-contraction-with-contiionning} and does not use
the stationarity or \(1\)-dependence of \(X\).
By Lemma~\ref{lem:conditioning-domination}, all the conditional KL divergences considered below are well defined.

\begin{lemma}
\label{lem:conditional-contraction}
For any
\(1\le j\le t\le N\),
\[
\D\!\left(
Z_j\mid Z_{1:(j-1)},X_{(j+1):t}
\right)
\ \le \
\beta\,
\D\!\left(
X_j\mid Z_{1:(j-1)},X_{(j+1):t}
\right).
\]
\end{lemma}

\begin{proof}
Fix \(1\le j\le t\le N\), and set
\(W:=\left(Z_{1:(j-1)},X_{(j+1):t}\right)\).
Since
\(P_+^W\ll P_-^W\) by Lemma~\ref{lem:conditioning-domination}, and
\[
P_\pm^{Z_j\mid (X_j,W)=(x,w)}
=
K_j(\cdot\mid x),
\qquad
\textup{for }P_\pm^{(X_j,W)}\text{-almost every }(x,w),
\]
and \(K_j\) is \(\beta\)-contracting,
Lemma~\ref{lem:general-contraction-with-contiionning} applies with
\((X,Y,Z,K)=(X_j,W,Z_j,K_j)\), and gives
\[
\D(Z_j\mid W)
\le
\beta\,\D(X_j\mid W).
\]
By the definition of \(W\), this is precisely the inequality of
Lemma~\ref{lem:conditional-contraction}.
\end{proof}

%%%%%%
\bigskip 

The third result is a one-step backward inequality.
Its proof combines the chain rule for conditional KL divergence with Lemma~\ref{lem:conditional-contraction}, together with the \(1\)-dependence assumption on \(X\).

\begin{lemma}
\label{lem:one-step-iteration}
For any
\(2\le j\le t\le N\),
\[
\D\!\left(
X_j\mid Z_{1:(j-1)},X_{(j+1):t}
\right)
\ \le \
\beta\,
\D\!\left(
X_{j-1}\mid Z_{1:(j-2)},X_{j:t}
\right)
\ + \
\D\!\left(
X_j\mid X_{(j+1):t}
\right).
\]
\end{lemma}

\begin{proof}
  Fix \(2\le j\le t\le N\).  Applying the chain rule in
  Lemma~\ref{lem:standard-properties-kl}\ref{item:chain-rule-va} in
  two different orders, conditionally on
  \(\bigl(Z_{1:(j-2)},X_{(j+1):t}\bigr)\), gives
\[
\begin{aligned}
\D\!\left(
X_j,Z_{j-1}\mid Z_{1:(j-2)},X_{(j+1):t}
\right)
\ &= \
\D\!\left(
X_j\mid Z_{1:(j-1)},X_{(j+1):t}
\right)
+
\D\!\left(
Z_{j-1}\mid Z_{1:(j-2)},X_{(j+1):t}
\right)
\\
\ &= \
\D\!\left(
Z_{j-1}\mid Z_{1:(j-2)},X_{j:t}
\right)
+
\D\!\left(
X_j\mid Z_{1:(j-2)},X_{(j+1):t}
\right).
\end{aligned}
\]
Using the non-negativity of conditional KL divergence, we obtain
\[
\D\!\left(
X_j\mid Z_{1:(j-1)},X_{(j+1):t}
\right)
\ \le \
\D\!\left(
Z_{j-1}\mid Z_{1:(j-2)},X_{j:t}
\right)
+
\D\!\left(
X_j\mid Z_{1:(j-2)},X_{(j+1):t}
\right).
\]

Lemma~\ref{lem:conditional-contraction}, applied with \(j-1\) in place
of \(j\), gives
\[
\D\!\left(
Z_{j-1}\mid Z_{1:(j-2)},X_{j:t}
\right)
\ \le \
\beta\,
\D\!\left(
X_{j-1}\mid Z_{1:(j-2)},X_{j:t}
\right).
\]

Since \(X\) is \(1\)-dependent and \(Z_{1:(j-2)}\) is generated
coordinatewise from \(X_{1:(j-2)}\), the variables \(X_{j:t}\) and
\(Z_{1:(j-2)}\) are independent under both \(P_+\) and \(P_-\).
Therefore,
\[
\D\!\left(
X_j\mid Z_{1:(j-2)},X_{(j+1):t}
\right)
\ = \
\D\!\left(
X_j\mid X_{(j+1):t}
\right).
\]

Substituting the last two displays into the preceding bound proves Lemma~\ref{lem:one-step-iteration}.
\end{proof}

%%%%%

\subsubsection{Main proof}

We now combine the last two lemmas above.

\begin{proof}[Proof of Proposition~\ref{prop:abstract-kl-iteration}]
We first prove that, for every \(t\in[N]\),
\begin{equation}
\label{lem:backward-bound}
\D\!\left(
X_t\mid Z_{1:(t-1)}
\right)
\ \le \
\sum_{k=0}^{t-1}
\beta^k
\D\!\left(
X_0\mid X_{1:k}
\right).
\end{equation}
For \(t=1\), inequality~\eqref{lem:backward-bound} becomes \(\D(X_1)\le\D(X_0)\), which holds since stationarity gives \(\D(X_1)=\D(X_0)\).
Let \(t\ge2\).
Applying Lemma~\ref{lem:one-step-iteration} successively with
\(j=t,t-1,\ldots,2\) gives
\[
\begin{aligned}
\D\!\left(
X_t\mid Z_{1:(t-1)}
\right)
&\ \le \
\D(X_t)
+
\beta\,
\D\!\left(
X_{t-1}\mid Z_{1:(t-2)},X_t
\right)
\\
&\ \le \
\D(X_t)
+
\beta\,
\D\!\left(
X_{t-1}\mid X_t
\right)
+
\beta^2
\D\!\left(
X_{t-2}\mid Z_{1:(t-3)},X_{(t-1):t}
\right)
\\
&\ \le \
\sum_{k=0}^{t-1}
\beta^k
\D\!\left(
X_{t-k}\mid X_{(t-k+1):t}
\right).
\end{aligned}
\]
By stationarity,
\(\D\!\left(
X_{t-k}\mid X_{(t-k+1):t}
\right)
 = 
\D\!\left(
X_0\mid X_{1:k}
\right)
\) for every \(k\in\{0,\ldots,t-1\}\).
Substituting these identities into the preceding bound
proves~\eqref{lem:backward-bound}.

By repeated application of the chain rule in
Lemma~\ref{lem:standard-properties-kl}\ref{item:chain-rule-va},
Lemma~\ref{lem:conditional-contraction} applied with \(j=t\), and
\eqref{lem:backward-bound}, we obtain
\[
\begin{aligned}
\D\!\left(
Z_{1:N}
\right)
\ = \
\sum_{t=1}^N
\D\!\left(
Z_t\mid Z_{1:(t-1)}
\right)
&\ \le \
\beta
\sum_{t=1}^N
\D\!\left(
X_t\mid Z_{1:(t-1)}
\right)
\\
&\ \le \
\beta
\sum_{t=1}^N
\sum_{k=0}^{t-1}
\beta^k
\D\!\left(
X_0\mid X_{1:k}
\right)
\\
&\ = \
\beta
\sum_{k=0}^{N-1}
\beta^k(N-k)
\D\!\left(
X_0\mid X_{1:k}
\right),
\end{aligned}
\]
where the last equality follows by exchanging the order of summation.
This proves Proposition~\ref{prop:abstract-kl-iteration}.
\end{proof}

%%% Local Variables:
%%% mode: latex
%%% TeX-master: "main.tex"
%%% ispell-local-dictionary: "american"
%%% End:

%%%%%%%%%

\section{Proof of Lemma~\ref{lem:lemma-conditional-KL-gaussian-general-sp-densities}}
\label{appendix:gaussian-process-bounds}

%%%%%%%%%
%Appendix D
%%%%%%%%%%

For each sign \(\pm\), let \(\gamma_\pm\) denote the autocovariance function associated with \(f_\pm\), so that \(\gamma_\pm(h)=\int_{-\pi}^{\pi}e^{ih\lambda}f_\pm(\lambda)\,d\lambda\), for every \(h\in\mathbb Z\).
Since the processes have unit variance, \(\gamma_\pm(0)=1\).

Fix \(k\geq1\). Set
\begin{align*}
\boldsymbol\gamma_\pm
\ := \begin{pmatrix}
\gamma_\pm(1)&\ldots&\gamma_\pm(k)
     \end{pmatrix}
  ^\top
\quad \qquad  \textup{ and } \qquad \quad 
\Gamma_\pm
\ := \
\bigl[
\gamma_\pm(i-j)
\bigr]_{1\leq i,j\leq k}.
\end{align*}
Thus, \(\Gamma_\pm \in \mathbb{R}^{k \times k}\) is the covariance matrix of \(X_{1:k}\) under
\(P_\pm\), while \(\boldsymbol\gamma_\pm\) is the vector of covariances
between \(X_0\) and \(X_{1:k}\). 
Consequently, the covariance matrix of \((X_0,X_{1:k})\), denoted by
\(\widetilde\Gamma_\pm\), is
\[
\widetilde\Gamma_\pm
\ := \
\begin{pmatrix}
1 & \boldsymbol\gamma_\pm^\top\\
\boldsymbol\gamma_\pm & \Gamma_\pm
\end{pmatrix}.
\]
Note that $\widetilde\Gamma_\pm$ corresponds to $\Gamma_\pm$ but with
$k$ replaced by $k+1$.

%%%%
\smallskip 

\textit{Some notation for vectors and matrices.} We denote by \(|\cdot|\) the Euclidean norm of vectors, by
\(\|\cdot\|_{\mathrm{op}}\) the corresponding operator
norm on square matrices and by $\mathrm{Sp}(\cdot)$ the set of their eigenvalues.
Recall that for any real symmetric matrix \(A\), we have
\begin{align}
  \label{eq:opertor-quadratic}
  \begin{split}
  \min\pa{\mathrm{Sp}(A)}=\inf_{|\boldsymbol y|=1}\boldsymbol y^T A\boldsymbol y\;,
    \qquad  
    \max\pa{\mathrm{Sp}(A)}=\sup_{|\boldsymbol y|=1}\boldsymbol y^T A\boldsymbol y\;,\\
 \|A\|_{\mathrm{op}}=\max_{\lambda\in\mathrm{Sp}(A)}|\lambda|\;,
  \qquad  \|A^{-1}\|_{\mathrm{op}}=\pa{\min_{\lambda\in\mathrm{Sp}(A)}|\lambda|}^{-1}\;,    
  \end{split}
\end{align}
where the last identity applies when \(A\) is invertible. 

%%%%%%

\subsection{Auxiliary lemmas}

The proof of Lemma~\ref{lem:lemma-conditional-KL-gaussian-general-sp-densities}
relies on the following four auxiliary lemmas.
Introduce the notation
\begin{align*}
d_\infty := 
\sup_{\lambda\in[-\pi,\pi]}
\Big{|}f_+(\lambda)-f_-(\lambda)\Big{|}.
\end{align*}
The first lemma provides spectral bounds, uniform in \(k\), for the
covariance matrices \(\Gamma_\pm\) and \(\widetilde\Gamma_\pm\), and
the difference matrix \(\Gamma_+-\Gamma_-\).
It follows from the assumption \(a\leq f_\pm\leq b\) and the definition
of \(d_\infty\) above.
This is an elementary result for Toeplitz matrices; for completeness,
a proof is given in Appendix~\ref{technical-lemma-eigenvalues}.

\begin{lemma}
\label{lem:gaussian-covariance-matrix-bounds}
The matrices \(\Gamma_\pm\) and  \(\widetilde \Gamma_\pm\) have all
their eigenvalues in $[2\pi a,2\pi b]$, while all eigenvalues of the matrix
$\Gamma_+-\Gamma_-$ lie in $[-2\pi d_\infty,2\pi d_\infty]$.
\end{lemma}

By Lemma~\ref{lem:gaussian-covariance-matrix-bounds}, since $a>0$, we may define
\[
\boldsymbol\phi_\pm
\ := \
\Gamma_\pm^{-1}\boldsymbol\gamma_\pm,
\qquad \qquad 
\sigma_\pm^2
\ := \
1-
\boldsymbol\phi_\pm^\top\boldsymbol\gamma_\pm.
\]
The second lemma identifies the conditional distributions of \(X_0\) given \(X_{1:k}\), under \(P_+\) and \(P_-\). 
It also gives a uniform lower bound on their variances, ensuring that these distributions are nondegenerate.

\begin{lemma}
\label{lem:gaussian-conditional-laws}
For every \(\boldsymbol x\in\mathbb R^k\), we have
\(P_\pm^{X_0\mid X_{1:k}=\boldsymbol x}
 = 
\mathcal N\!\left(
\boldsymbol\phi_\pm^\top\boldsymbol x,
\sigma_\pm^2
\right)
\).
Moreover, \(\sigma_\pm^2 \ge 2\pi a\).
\end{lemma}

\begin{proof}
Under \(P_\pm\), the vector \((X_0,X_{1:k})\) is centered Gaussian with
covariance matrix \(\widetilde\Gamma_\pm\).
Since \(\Gamma_\pm\) is invertible, the Gaussian conditioning formula
(see, for instance, \cite[Chapter~5]{brockwell1991time}) gives
\[
P_\pm^{X_0\mid X_{1:k}=\boldsymbol x}
\ = \
\mathcal N\!\left(\,
\boldsymbol\gamma_\pm^\top
\Gamma_\pm^{-1}\boldsymbol x,
\, \, 1-
\boldsymbol\gamma_\pm^\top
\Gamma_\pm^{-1}\boldsymbol\gamma_\pm \, 
\right).
\]
Recalling the definitions
\[
\boldsymbol\phi_\pm
\ := \
\Gamma_\pm^{-1}\boldsymbol\gamma_\pm,
\qquad \qquad
\sigma_\pm^2
\ := \
1-
\boldsymbol\phi_\pm^\top\boldsymbol\gamma_\pm,
\]
this is precisely
\(P_\pm^{X_0\mid X_{1:k}=\boldsymbol x}
=
\mathcal N\!\left(
\boldsymbol\phi_\pm^\top\boldsymbol x,\,
\sigma_\pm^2
\right)\). 
Moreover, by straightforward algebra, setting
\(\boldsymbol y_\pm
\ := \
\begin{pmatrix}
1\\
-\boldsymbol\phi_\pm
\end{pmatrix}\),
we have
\[
\widetilde\Gamma_\pm\boldsymbol y_\pm
\ = \
\begin{pmatrix}
1-\boldsymbol\gamma_\pm^\top\boldsymbol\phi_\pm\\
\boldsymbol\gamma_\pm-\Gamma_\pm\boldsymbol\phi_\pm
\end{pmatrix}.
\]
Consequently, \(\boldsymbol y_\pm^\top\widetilde\Gamma_\pm \boldsymbol y_\pm
 = \sigma_\pm^2\). 
Since Lemma~\ref{lem:gaussian-covariance-matrix-bounds} and~(\ref{eq:opertor-quadratic}) give
\(\boldsymbol y_\pm^\top\widetilde\Gamma_\pm\boldsymbol y_\pm
\ge 2\pi a|\boldsymbol y_\pm|^2\), and
since \(|\boldsymbol y_\pm|\ge1\), it follows that
\(\sigma_\pm^2\ge2\pi a\).
\end{proof}

%%%%%
\bigskip
%%%%%%

Since \(a>0\), Lemma~\ref{lem:gaussian-covariance-matrix-bounds}
implies that \(\Gamma_+\) and \(\Gamma_-\) are positive definite, so both Gaussian laws
$P_+^{X_{1:k}}$ and $P_-^{X_{1:k}}$ admit strictly positive densities
with respect to Lebesgue measure. In particular we have 
\(P_+^{X_{1:k}}\ll P_-^{X_{1:k}}\) and the conditional KL divergence $\D\!\left(
P_+^{X_0\mid X_{1:k}}
\,\middle\|\,
P_-^{X_0\mid X_{1:k}}
\,\middle|\,
P_+^{X_{1:k}}
\right)
$ is well defined. Using Lemma~\ref{lem:gaussian-conditional-laws},
the definition of conditional KL divergence, and the usual formula for
the KL divergence between two univariate Gaussian distributions, we
further obtain
\[
\D\!\left(
P_+^{X_0\mid X_{1:k}}
\,\middle\|\,
P_-^{X_0\mid X_{1:k}}
\,\middle|\,
P_+^{X_{1:k}}
\right)
\ = \
\frac12
\left[
\log\left(
\frac{\sigma_-^2}{\sigma_+^2}
\right)
+
\frac{\sigma_+^2}{\sigma_-^2}
-1
+
\frac{
\mathbb E_+\left[
\left(
(\boldsymbol\phi_+-\boldsymbol\phi_-)^\top X_{1:k}
\right)^2
\right]
}{\sigma_-^2}
\right],
\]
where \(\mathbb E_+\) denotes expectation under \(P_+\).

The third lemma bounds the above conditional KL divergence using the
quantities we introduced.

\begin{lemma}
\label{lem:gaussian-conditional-kl-parameters}
We have
\[
\begin{aligned}
\D\!\left(
P_+^{X_0\mid X_{1:k}}
\,\middle\|\,
P_-^{X_0\mid X_{1:k}}
\,\middle|\,
P_+^{X_{1:k}}
\right)
\ &\leq \
\frac{1}{2\pi a}
\bigg(
\left|
\boldsymbol\phi_+
-
\boldsymbol\phi_-
\right|
\left|
\boldsymbol\gamma_+
\right|
+
\left|
\boldsymbol\gamma_+
-
\boldsymbol\gamma_-
\right|
\left|
\boldsymbol\phi_-
\right|
+
\pi b
\left|
\boldsymbol\phi_+
-
\boldsymbol\phi_-
\right|^2
\bigg).
\end{aligned}
\]
\end{lemma}

\begin{proof}%[Proof of Lemma~\ref{lem:gaussian-conditional-kl-parameters}]
Since \(X_{1:k}\) is centered with covariance matrix \(\Gamma_+\)
under \(P_+\), we have, by Lemma~\ref{lem:gaussian-covariance-matrix-bounds} and~(\ref{eq:opertor-quadratic}),
\[
\begin{aligned}
\mathbb E_+\left[
\left(
(\boldsymbol\phi_+-\boldsymbol\phi_-)^\top X_{1:k}
\right)^2
\right]
&\ = \
(\boldsymbol\phi_+-\boldsymbol\phi_-)^\top
\Gamma_+
(\boldsymbol\phi_+-\boldsymbol\phi_-)
\\
&\ \leq \
2\pi b
\left|
\boldsymbol\phi_+-\boldsymbol\phi_-
\right|^2.
\end{aligned}
\]
Moreover, using \(\log x\leq x-1\), we obtain
\[
\begin{aligned}
\log\left(
\frac{\sigma_-^2}{\sigma_+^2}
\right)
+
\frac{\sigma_+^2}{\sigma_-^2}
-1
&\ \leq \
\frac{\sigma_-^2-\sigma_+^2}{\sigma_+^2}
+
\frac{\sigma_+^2-\sigma_-^2}{\sigma_-^2}
\\
&\ \leq \
\left|
\sigma_+^2-\sigma_-^2
\right|
\left(
\frac{1}{\sigma_+^2}
+
\frac{1}{\sigma_-^2}
\right).
\end{aligned}
\]
Lemma~\ref{lem:gaussian-conditional-laws} gives \(\sigma_\pm^2\geq2\pi a\). Combining the preceding bounds yields
\[
\begin{aligned}
\D\!\left(
P_+^{X_0\mid X_{1:k}}
\,\middle\|\,
P_-^{X_0\mid X_{1:k}}
\,\middle|\,
P_+^{X_{1:k}}
\right)
\ \leq \
\frac{1}{4\pi a}
\left(
2\left|
\sigma_+^2-\sigma_-^2
\right|
+
2\pi b
\left|
\boldsymbol\phi_+-\boldsymbol\phi_-
\right|^2
\right).
\end{aligned}
\]
Finally, the definitions \(\sigma_\pm^2 := 1-\boldsymbol\phi_\pm^\top\boldsymbol\gamma_\pm\) give
\[
\begin{aligned}
\left|
\sigma_+^2-\sigma_-^2
\right|
&\ = \
\left|
-
(\boldsymbol\phi_+-\boldsymbol\phi_-)^\top
\boldsymbol\gamma_+
-
\boldsymbol\phi_-^\top
(\boldsymbol\gamma_+-\boldsymbol\gamma_-)
\right|
\\
&\ \leq \
\left|
\boldsymbol\phi_+-\boldsymbol\phi_-
\right|
\left|
\boldsymbol\gamma_+
\right|
+
\left|
\boldsymbol\gamma_+-\boldsymbol\gamma_-
\right|
\left|
\boldsymbol\phi_-
\right|.
\end{aligned}
\]
Substituting this estimate into the preceding bound proves Lemma~\ref{lem:gaussian-conditional-kl-parameters}.
\end{proof}

%%%%%
\bigskip
%%%%%%

The last lemma controls the parameters appearing in the preceding bound
in terms of the spectral densities.

\begin{lemma}
\label{lem:gaussian-conditional-parameter-bounds}
We have 
\begin{align*}
&\left|
\boldsymbol\gamma_\pm
\right|
\ \le \
\sqrt{\pi}\,
\left\|
f_\pm-\frac{1}{2\pi}
\right\|,
\qquad 
\left|
\boldsymbol\gamma_+
-
\boldsymbol\gamma_-
\right|
\ \le \
\sqrt{\pi}\,\left\|
f_+-f_-
\right\|
,
\qquad
\left|
\boldsymbol\phi_\pm
\right|
\ \le \
\frac{1}{2\sqrt{\pi}a}
\left\|
f_\pm-\frac{1}{2\pi}
\right\|,\\
& \qquad \qquad \quad  \text{ and } \quad
\left|
\boldsymbol\phi_+
-
\boldsymbol\phi_-
\right|
\ \le \frac1{2\sqrt{\pi}a}
\left(\left\|
f_+-f_-
\right\|
+\frac{d_\infty}{a}\,
\left\|
f_+-\frac{1}{2\pi}
\right\|
\right).
\end{align*}
\end{lemma}

\begin{proof}%[Proof of Lemma~\ref{lem:gaussian-conditional-parameter-bounds}]
Since \(\gamma_\pm(0)=1\) and
\(\gamma_\pm(-h)=\gamma_\pm(h)\), Parseval's identity gives
\[
\sum_{h=1}^{\infty}
|\gamma_\pm(h)|^2
\ = \
\pi
\left\|
f_\pm-\frac{1}{2\pi}
\right\|^2,
\qquad
\sum_{h=1}^{\infty}
|\gamma_+(h)-\gamma_-(h)|^2
\ = \
\pi
\left\|
f_+-f_-
\right\|^2\;.
\]
The first two bounds of the lemma follow. For the third one, by
definition of
$\boldsymbol\phi_\pm$,  Lemma~\ref{lem:gaussian-covariance-matrix-bounds} and~(\ref{eq:opertor-quadratic}), we obtain
$
\left|
\boldsymbol\phi_\pm
\right|
\ \leq \
\|\Gamma_\pm^{-1}\|_{\mathrm{op}}
\left|
\boldsymbol\gamma_\pm
\right|
\ \leq \
\frac{1}{2\sqrt{\pi}a}
\left\|
f_\pm-\frac{1}{2\pi}
\right\|$.

It remains to control
\(\left|\boldsymbol\phi_+-\boldsymbol\phi_-\right|\).
We have
\[
\begin{aligned}
\boldsymbol\phi_+-\boldsymbol\phi_-
\ = \
\Gamma_-^{-1}
\left[
\Gamma_-\boldsymbol\phi_+
-
\Gamma_-\boldsymbol\phi_-
\right]
&\ = \
\Gamma_-^{-1}
\left[
(\Gamma_--\Gamma_+)\boldsymbol\phi_+
+
\Gamma_+\boldsymbol\phi_+
-
\Gamma_-\boldsymbol\phi_-
\right]
\\
&\ = \
\Gamma_-^{-1}
\left[
(\Gamma_--\Gamma_+)\boldsymbol\phi_+
+
(\boldsymbol\gamma_+-\boldsymbol\gamma_-)
\right],
\end{aligned}
\]
where the last equality follows from
\(\Gamma_\pm\boldsymbol\phi_\pm=\boldsymbol\gamma_\pm\).
Therefore,
\begin{equation}\label{proof-small-lem-parameter-bounds:eq-phi}
\begin{aligned}
\left|
\boldsymbol\phi_+-\boldsymbol\phi_-
\right|
&\ \leq \
\|\Gamma_-^{-1}\|_{\mathrm{op}}
\Big{(}
\left|
\boldsymbol\gamma_+-\boldsymbol\gamma_-
\right|
+
\|\Gamma_--\Gamma_+\|_{\mathrm{op}}
\left|
\boldsymbol\phi_+
\right|
\Big{)}
\\
&\ \leq \frac1{2\pi a}
\Big{(}
\left|
\boldsymbol\gamma_+-\boldsymbol\gamma_-
\right|
+2\pi\,d_\infty\,
\left|
\boldsymbol\phi_+
\right|
\Big{)}
  \;,
\end{aligned}
\end{equation}
where we again used Lemma~\ref{lem:gaussian-covariance-matrix-bounds}
and~(\ref{eq:opertor-quadratic}).
Plugging the second and third bounds of Lemma~\ref{lem:gaussian-conditional-parameter-bounds} (already proved
above) into~\eqref{proof-small-lem-parameter-bounds:eq-phi} gives the
last bound.
\end{proof}

%%%%%%

\subsection{Main proof}

We now combine the last two auxiliary lemmas.
For convenience, introduce the notation
\begin{align*}
 r:= 
\max_{\pm\in\{+,-\}}
\left\|
f_\pm-\frac{1}{2\pi}
\right\|
\quad  \text{ and } \quad 
d_2 := 
\|f_+-f_-\|\;.
\end{align*}

\begin{proof}[Proof of Lemma~\ref{lem:lemma-conditional-KL-gaussian-general-sp-densities}]
Throughout the proof, \(C\), \(C_a\), \(C_b\) and \(C_{a,b}\) denote positive
constants that are universal, depend only on \(a\), depend
only on \(b\), and depend only on \(a\) and \(b\), respectively; their values may change from line to
line.
With this convention, the bounds in Lemma~\ref{lem:gaussian-conditional-parameter-bounds} can be written as
\[
\left|\boldsymbol\gamma_+\right|
 \le 
C\,r,
\qquad
\left|\boldsymbol\gamma_+-\boldsymbol\gamma_-\right| \le 
C\,d_2,
\qquad
\left|\boldsymbol\phi_-\right|
 \le 
C_a r,
\qquad 
\left|\boldsymbol\phi_+-\boldsymbol\phi_-\right|
 \le 
C_a(d_2+r d_\infty).
\]
Substituting these bounds into
Lemma~\ref{lem:gaussian-conditional-kl-parameters} further yields that 
\begin{align*}
\D\!\left(
P_+^{X_0\mid X_{1:k}}
\,\middle\|\,
P_-^{X_0\mid X_{1:k}}
\,\middle|\,
P_+^{X_{1:k}}
\right)
&\le C_a\left[
(d_2+r d_\infty)r
+d_2 r
+ b (d_2+r d_\infty)^2
\right].  
\end{align*}
Using
\((u+v)^2\leq 2u^2+2v^2\), this yields
\begin{equation}
    \label{eq:proof:kl-conditional-gaussian:bound-inter}
\D\!\left(
P_+^{X_0\mid X_{1:k}}
\,\middle\|\,
P_-^{X_0\mid X_{1:k}}
\,\middle|\,
P_+^{X_{1:k}}
\right)
\ \le \
C_{a,b}
\left(
r^2d_\infty
+
r d_2
+
d_2^2
+
r^2d_\infty^2
\right).
\end{equation}
Since \([-\pi,\pi]\) has length \(2\pi\), and by the triangle inequality,
\[
d_2 \le \sqrt{2\pi}\,d_\infty , \qquad  \quad 
d_2 
 \le 
\left\|
f_+-\frac{1}{2\pi}
\right\|
+
\left\|
f_--\frac{1}{2\pi}
\right\|
\le 
2r.
\]
Moreover, since \(0\leq f_\pm\leq b\), we have that \(r\) and \(d_\infty\) are bounded by constants depending only on \(b\). 
Combining these bounds, we obtain
\begin{equation}
\label{intermediates-estimates:proof:cond-gauss}
\begin{aligned}
r^2d_\infty
&\ \le \
C_b\,r d_\infty,
&\qquad
rd_2
&\ \le \
C\,r d_\infty,
\\
r^2d_\infty^2
&\ \le \
C_b\,r d_\infty,
&\qquad
d_2^2
&\ \le \
C\,r d_\infty.
\end{aligned}
\end{equation}

Plugging~\eqref{intermediates-estimates:proof:cond-gauss} into~\eqref{eq:proof:kl-conditional-gaussian:bound-inter} yields 
\[
\D\!\left(
P_+^{X_0\mid X_{1:k}}
\,\middle\|\,
P_-^{X_0\mid X_{1:k}}
\,\middle|\,
P_+^{X_{1:k}}
\right)
\ \le \
C_{a,b}\,r d_\infty.
\]
By the definitions of \(r\) and \(d_\infty\), this is precisely the inequality of
Lemma~\ref{lem:lemma-conditional-KL-gaussian-general-sp-densities}.
\end{proof}

%%%%%%

\subsection{Technical details}
\label{technical-lemma-eigenvalues}

\begin{proof}[Proof of Lemma~\ref{lem:gaussian-covariance-matrix-bounds}]
For every \(\boldsymbol x=(x_1,\ldots,x_k)^\top\in\mathbb R^k\),
\[
\begin{aligned}
\boldsymbol x^\top\Gamma_\pm\boldsymbol x
\ = \ 
\sum_{r,s=1}^k
x_rx_s\gamma_\pm(r-s)
\ &= \
\int_{-\pi}^{\pi}
\left(
\sum_{r,s=1}^k
x_rx_s e^{i(r-s)\lambda}
\right)
f_\pm(\lambda)\,d\lambda
\\
\ &= \
\int_{-\pi}^{\pi}
\left|
\sum_{r=1}^k
x_r e^{ir\lambda}
\right|^2
f_\pm(\lambda)\,d\lambda.
\end{aligned}
\]
By Parseval's identity, \(\int_{-\pi}^{\pi}
\left|
\sum_{r=1}^k
x_r e^{ir\lambda}
\right|^2
d\lambda
=
2\pi|\boldsymbol x|^2\).
Since \(a\leq f_\pm\leq b\), we obtain
\[
2\pi a|\boldsymbol x|^2
\ \leq \ 
\boldsymbol x^\top\Gamma_\pm\boldsymbol x
\ \leq \
2\pi b|\boldsymbol x|^2.
\]
Hence all eigenvalues of \(\Gamma_\pm\)  belong to \([2\pi a,2\pi b]\).
For every \(\boldsymbol y=(y_0,\ldots,y_k)^\top\in\mathbb R^{k+1}\), the same
argument gives
\[
2\pi a|\boldsymbol y|^2
\ \le \
\boldsymbol y^\top\widetilde\Gamma_\pm\boldsymbol y
\ \le \
2\pi b|\boldsymbol y|^2,
\]
and therefore all eigenvalues of \(\widetilde\Gamma_\pm\) belong to \([2\pi a,2\pi b]\).
Finally, for every \(\boldsymbol x\in\mathbb R^k\),
\[
\begin{aligned}
\left|
\boldsymbol x^\top(\Gamma_+-\Gamma_-)\boldsymbol x
\right|
&\ = \
\left|
\int_{-\pi}^{\pi}
\left|
\sum_{r=1}^k x_r e^{ir\lambda}
\right|^2
\bigl(f_+(\lambda)-f_-(\lambda)\bigr)\,d\lambda
\right|
\\
&\ \le \
d_\infty
\int_{-\pi}^{\pi}
\left|
\sum_{r=1}^k x_r e^{ir\lambda}
\right|^2\,d\lambda
\ = \
2\pi d_\infty|\boldsymbol x|^2.
\end{aligned}
\]
Hence all eigenvalues of \(\Gamma_+-\Gamma_-\) belong to
\([-2\pi d_\infty,2\pi d_\infty]\).
\end{proof}

%%% Local Variables:
%%% mode: latex
%%% TeX-master: "main.tex"
%%% ispell-local-dictionary: "american"
%%% End:

%%%%%%

\section{Proof of Proposition~\ref{thm:cov-bound-signs}}
\label{proof-kolmogorov}

We first recall the following consequence of~\cite[Theorem~1]{KolmogorovRozanov1960}.

\begin{lemma}[Kolmogorov--Rozanov]
\label{lem:kolmogorov-rozanov}
Let \((U_1,U_2,V_1,V_2)\) be a centered Gaussian vector with positive marginal variances.
Define
\[
\varrho\bigl((U_1,U_2),(V_1,V_2)\bigr)
\ := \
\sup_{\substack{
Z_1\in\operatorname{span}(U_1,U_2),\;
Z_2\in\operatorname{span}(V_1,V_2)
\\
\mathbb E[Z_1^2]=\mathbb E[Z_2^2]=1
}}
\left|
\mathbb E[Z_1Z_2]
\right|.
\]
Then, for all measurable functions
\(q_1,q_2:\mathbb R^2\to\mathbb R\) such that
\(q_1(U_1,U_2),q_2(V_1,V_2)\in L^2\), we have
\[
\left|
\operatorname{Cov}\left(
q_1(U_1,U_2),
q_2(V_1,V_2)
\right)
\right|
\ \leq \
\varrho\bigl((U_1,U_2),(V_1,V_2)\bigr)
\sqrt{
\operatorname{Var}\left(q_1(U_1,U_2)\right)
\operatorname{Var}\left(q_2(V_1,V_2)\right)
}.
\]
\end{lemma}

%%%%%%%%%
\bigskip 
%%%%%%%%

\begin{proof}[Proof of Proposition~\ref{thm:cov-bound-signs}]
We first reduce the proof to the case where \(X\) is centered and has unit variance.
Let \(\mu:=\mathbb E[X_0]\),
\(\sigma^2:=\operatorname{Var}(X_0)\). 
Since $\sigma^2>0$, we may define
\[
\widetilde X_t
\ := \
\frac{X_t-\mu}{\sigma},
\qquad
\widetilde g_i(x,y)
\ := \
g_i(\mu+\sigma x,\mu+\sigma y),
\qquad i\in\{1,2\}.
\]
The process \(\widetilde X\) is centered, stationary, Gaussian, has unit variance, and has the same autocorrelation function \(\rho\) as \(X\).
Moreover,
\(g_i(X_s,X_t) =  \widetilde g_i(\widetilde X_s,\widetilde X_t)\) for \(i\in\{1,2\}\) and \(s,t\in \mathbb{Z}\),
and the functions \(\widetilde g_i\) still take values in \([-1,1]\).
We therefore assume without loss of generality that \(X\) is centered and has unit variance.

Lemma~\ref{lem:kolmogorov-rozanov}, applied to the Gaussian vector \((X_s,X_t,X_u,X_v)\), gives
\begin{equation}
    \label{eq:applciation-kolmogorov-rozanov}
\left|
\operatorname{Cov}\left(
g_1(X_s,X_t),
g_2(X_u,X_v)
\right)
\right|
\ \leq \
\varrho\left(
(X_s,X_t),(X_u,X_v)
\right),
\end{equation}
because the two variances appearing in the upper bound of Lemma~\ref{lem:kolmogorov-rozanov} are at most \(1\).

It remains to bound the maximal correlation \(\varrho\left((X_s,X_t),(X_u,X_v)\right)\).
Consider variables
\[
Z_1=aX_s+bX_t,
\qquad
Z_2=cX_u+dX_v,
\]
such that \(\mathbb E[Z_1^2]=\mathbb E[Z_2^2]=1\).
If \(s=t\), we choose the representation of \(Z_1\) with \(b=0\);
similarly, if \(u=v\), we choose the representation of \(Z_2\) with
\(d=0\).

We first show that
\begin{equation}\label{eq:a2+b2:epsilon}
a^2+b^2\leq\frac{1}{\epsilon},
\qquad
c^2+d^2\leq\frac{1}{\epsilon}.
\end{equation}
For the first inequality, we distinguish two cases.
\begin{itemize}
\item If \(s\neq t\), then
\[
1
\ = \
a^2+b^2+2ab\,\rho(s-t)
\ \geq \
(a^2+b^2)\bigl(1-|\rho(s-t)|\bigr)
\ \geq \
\epsilon(a^2+b^2).
\]
Since \(s-t\neq0\), the last inequality follows from the assumption
\(|\rho(h)|\leq1-\epsilon\) for every
\(h\in\mathbb Z\setminus\{0\}\).

\item If \(s=t\), our choice \(b=0\), together with
\(\mathbb E[Z_1^2]=1\) and \(\mathbb E[X_s^2]=1\), gives
\(a^2+b^2=a^2=1\).
Since \(|\rho(1)|\leq1-\epsilon\) by assumption, we have \(\epsilon\leq1\), and hence \(a^2+b^2\leq1/\epsilon\).
\end{itemize}
The same argument gives \(c^2+d^2\leq1/\epsilon\). Therefore, \eqref{eq:a2+b2:epsilon} is proved.

Recalling that \(Z_1=aX_s+bX_t\) and \(Z_2=cX_u+dX_v\), and setting
\[
M_1
\ := \
|\rho(s-u)|\vee|\rho(t-v)|,
\qquad
M_2
\ := \
|\rho(s-v)|\vee|\rho(t-u)|,
\] 
we obtain
\begin{equation*}
\begin{aligned}
\left|\mathbb E[Z_1Z_2]\right|
&\ = \
\left|ac\,\rho(s-u)
+
bd\,\rho(t-v)
+
ad\,\rho(s-v)
+
bc\,\rho(t-u)\right|
\\
&\ \leq \
(|ac|+|bd|)M_1
+
(|ad|+|bc|)M_2
\\
&\ \leq \
\frac12
(a^2+b^2+c^2+d^2)(M_1+M_2)
\ \leq \
\frac{M_1+M_2}{\epsilon},
\end{aligned}
\end{equation*}
where the last inequality follows from \eqref{eq:a2+b2:epsilon}. 
Since this bound holds for arbitrary 
\(Z_1\in\operatorname{span}(X_s,X_t)\) and
\(Z_2\in\operatorname{span}(X_u,X_v)\) satisfying
\(\mathbb E[Z_1^2]=\mathbb E[Z_2^2]=1\), we obtain
\[
\varrho\left(
(X_s,X_t),(X_u,X_v)
\right)
\ \leq \
\frac{
|\rho(s-u)|\vee|\rho(t-v)|
+
|\rho(s-v)|\vee|\rho(t-u)|
}{\epsilon}.
\]
Plugging this bound into~\eqref{eq:applciation-kolmogorov-rozanov} proves Proposition~\ref{thm:cov-bound-signs}.
\end{proof}

%%% Local Variables:
%%% mode: latex
%%% TeX-master: "main.tex"
%%% ispell-local-dictionary: "american"
%%% End:

%%%%%%%%%%%

\section{Proofs of the results in Section~\ref{sec:related-problems}}
\label{appendix:proof_related_pbms}

\subsection{Lag-\(h\) autocovariance estimation}
\label{appendix:lag-h-autocovariance}

In this appendix, we prove
Propositions~\ref{prop:lag-h-autocovariance-upper-bound}
and~\ref{prop:lag-h-autocovariance-lower-bound}.

%%%%%

\begin{proof}[Proof of Proposition~\ref{prop:lag-h-autocovariance-upper-bound}]
Fix \(N\geq2\), \(\alpha>0\), \(h\in[N-1]\), and
\(f\in\mathcal F(C_0,C_1,s)\).
We use elements from the proof of
Proposition~\ref{prop:fred-estimator}, given in
Section~\ref{sec:proof-prop-fred-est}.

Recall that \(\widetilde\gamma(h)\) is based on the unbiased empirical
lag-\(h\) autocovariance of
\(Z^{(2)}:=(Z_{i,2})_{i\in[N]}\), which we denote by
\begin{equation}
\label{eq:check_gamma_unbiased}
\widetilde\gamma^{(2)}(h)
\ := \
\frac1{N-h}
\sum_{j=1}^{N-h}Z_{j,2}Z_{j+h,2}.
\end{equation}
Since
\(\widetilde\gamma^{(2)}(h)
=(N/(N-h))\widehat\gamma^{(2)}(h)\),
where \(\widehat\gamma^{(2)}(h)\) is defined
in~\eqref{eq:check_gamma}, and since
\(\widetilde\gamma^{(2)}(h)\) is an unbiased estimator of
\((2/\pi)\arcsin(\rho_f(h))\), the variance
bound~\eqref{eq:variance-check-gamma} gives
\begin{equation}
\label{eq:MSE-unbiased-gamma}
\mathbb E_{f,K_{1:N}^{\mathrm{Lap},\alpha}}
\left[
\left(
\widetilde\gamma^{(2)}(h)
-
\frac2\pi\arcsin\!\left(\rho_f(h)\right)
\right)^2
\right]
\ \leq \
\frac{C_{C_1,s}}{N-h}
\left(
1+\frac1{\alpha^2}+\frac1{\alpha^4}
\right).
\end{equation}
Here, we used the fact that \(\varepsilon_{C_1,s}\) appearing
in~\eqref{eq:variance-check-gamma} depends only on \((C_1,s)\), and
can therefore be absorbed into \(C_{C_1,s}\).

Using
\(\widetilde\gamma(h)
=\hat\gamma(0)\sin(\pi\widetilde\gamma^{(2)}(h)/2)\), and
\(\gamma_f(h)=\gamma_f(0)\rho_f(h)\), we write
\[
\widetilde\gamma(h)-\gamma_f(h)
\ = \
\sin\!\left(\frac\pi2\widetilde\gamma^{(2)}(h)\right)
\bigl(\hat\gamma(0)-\gamma_f(0)\bigr)
+
\gamma_f(0)
\left(
\sin\!\left(\frac\pi2\widetilde\gamma^{(2)}(h)\right)
-\rho_f(h)
\right).
\]
By definition of \(\mathcal F(C_0,C_1,s)\), we have
\(\gamma_f(0)\leq C_0\). Moreover, \(|\sin(x)|\leq1\), and
\(\rho_f(h)\in[-1,1]\), and the map
\(x\mapsto\sin(\pi x/2)\) is \(\pi/2\)-Lipschitz. Combining these
facts, we obtain
\[
|\widetilde\gamma(h)-\gamma_f(h)|
\ \leq\
|\hat\gamma(0)-\gamma_f(0)|
+
C_0
\left(
2\wedge
\frac\pi2
\left|
\widetilde\gamma^{(2)}(h)
-
\frac2\pi\arcsin\!\left(\rho_f(h)\right)
\right|
\right).
\]

Squaring the preceding inequality, using
\((a+b)^2\leq2a^2+2b^2\), and then taking expectations, we apply
Proposition~\ref{lem:variance-estimator}
and~\eqref{eq:MSE-unbiased-gamma}.
Using also
\(1+\alpha^{-2}+\alpha^{-4}
\leq3/(\alpha^4\wedge1)\),
we obtain, for some constant \(C_{C_0,C_1,s}\), depending only on
\((C_0,C_1,s)\),
\[
\mathbb E_{f,K_{1:N}^{\mathrm{Lap},\alpha}}
\left[
|\widetilde\gamma(h)-\gamma_f(h)|^2
\right]
\ \leq\
C_{C_0,C_1,s}
\left[
\left(
1\wedge\frac{1}{(\alpha^2\wedge1)N}
\right)
+
\left(
1\wedge\frac{1}{(\alpha^4\wedge1)(N-h)}
\right)
\right].
\]
Since the second term in brackets is always at least as large as the
first, this concludes the proof.
\end{proof}

%%%%%%

\medskip 

%%%%%%%%%%%%%%%

\begin{proof}[Proof of Proposition~\ref{prop:lag-h-autocovariance-lower-bound}]
We derive Proposition~\ref{prop:lag-h-autocovariance-lower-bound}
from Proposition~\ref{prop:local-lag-h-autocovariance-lower-bound-test}.

Fix a lag \(h\geq1\), a sample size \(N\geq h+1\), and \(\alpha>0\).
Let \(f(\lambda):=1/(2\pi)\) for all \(\lambda\), and, for a value
\(\varepsilon>0\) to be chosen below, define \(f_\pm\)
by~\eqref{eq:local-lag-h-epsilon-sp-density-choice}.
Then
\[
f_\pm(\lambda)
\ \geq \
\frac{1-2\varepsilon}{2\pi},
\qquad \lambda\in[-\pi,\pi].
\]
Moreover, \(\gamma_{f_\pm}(0)=1\),
\(\rho_{f_\pm}(h)=\pm\varepsilon\), and
\(\rho_{f_\pm}(k)=0\) for every positive lag \(k\neq h\).
Consequently, for every
\(\varepsilon\in\left(0,\frac12\wedge
\frac{C_1}{\sqrt2\,h^s}\right]\), the functions \(f_\pm\) are spectral
densities and satisfy
\[
\sum_{k\in\mathbb Z\setminus\{0\}}
|k|^{2s}|\rho_{f_\pm}(k)|^2
\ = \ 
2h^{2s}\varepsilon^2
\ \leq \ C_1^2.
\]
It follows that
\begin{equation}
\label{eq:fpm_in_F}
f_\pm\in\mathcal F(1,C_1,s)
\subseteq\mathcal F(C_0,C_1,s),
\qquad \quad 
\text{whenever}\quad
\varepsilon\in
\left(
0,
\frac12\wedge\frac{C_1}{\sqrt2\,h^s}
\right],
\end{equation}
where the inclusion follows directly from \(C_0\geq1\).

Let now \(K_{1:N}\in\mathcal M_\alpha^N\), and let
\(\hat g\in\mathcal E_{\mathbb R}(K_{1:N})\).
As in~\eqref{eq:estimator_to_test}, define
\(\hat\phi\in\mathcal E_{[0,1]}(K_{1:N})\) by
\(\hat\phi(Z)=1\) if \(\hat g(Z)\geq0\), and
\(\hat\phi(Z)=0\) otherwise.
Since \(\gamma_{f_\pm}(h)=\pm\varepsilon\), we have
\begin{equation}
\label{eq:L2_lb_test}
\max_{\pm\in\ac{-,+}}
\mathbb E_{f_\pm,K_{1:N}}
\left[
\left|\hat g(Z)-\gamma_{f_\pm}(h)\right|^2
\right]
\ \geq \
\varepsilon^2\,
\mathcal T_{N,\alpha}
\bigl(\hat\phi,f,\varepsilon,K_{1:N}\bigr).
\end{equation}

Set
\[
f_{\min,h}
:=
\frac1{2\pi}
\wedge
\frac{C_1}{\sqrt2\,\pi h^s}.
\]
The constant spectral density \(f=1/(2\pi)\) satisfies
\(f\geq f_{\min,h}\). 
Therefore, Proposition~\ref{prop:local-lag-h-autocovariance-lower-bound-test},
applied with \(r=1/4\) and \(f_{\min}=f_{\min,h}\), shows that, for
\begin{equation}
\label{eq:epsilon-star-lag-lower-bound}
\varepsilon_{\mathrm{low}}
:=
c_{1/4,f_{\min,h}}
\left(
1\wedge
\left[(\alpha^4\wedge1)(N-h)\right]^{-1/2}
\right),
\end{equation}
with a constant \(c_{1/4,f_{\min,h}}\in(0,\pi f_{\min,h}]\),
we have
\begin{equation}
\label{eq:testing-risk-epsilon-star}
\mathcal T_{N,\alpha}
\bigl(\hat\phi,f,\varepsilon_{\mathrm{low}},K_{1:N}\bigr)
\ \geq \
\frac14.
\end{equation}
 
Since
\[
\varepsilon_{\mathrm{low}}
\leq
c_{1/4,f_{\min,h}}
\leq
\pi f_{\min,h}
\leq
\frac12\wedge\frac{C_1}{\sqrt2\,h^s},
\]
we obtain from~\eqref{eq:fpm_in_F} that \(f_\pm\in\mathcal F(C_0,C_1,s)\) for \(\varepsilon=\varepsilon_{\mathrm{low}}\). 
Combining this with~\eqref{eq:L2_lb_test}, \eqref{eq:epsilon-star-lag-lower-bound} and~\eqref{eq:testing-risk-epsilon-star}, we obtain
\[
\sup_{\tilde f\in\mathcal F(C_0,C_1,s)}
\mathbb E_{\tilde f,K_{1:N}}
\left[
\left|\hat g(Z)-\gamma_{\tilde f}(h)\right|^2
\right]
\ \geq \
\frac{c_{1/4,f_{\min,h}}^2}{4}
\left(
1\wedge
\left[(\alpha^4\wedge1)(N-h)\right]^{-1}
\right).
\]
Since \(K_{1:N}\) and \(\hat g\) were arbitrary, taking the two
infima concludes the proof with
\(c_{C_1,s,h}:=c_{1/4,f_{\min,h}}^2/4\).
\end{proof}

%%%%%%

%%%%%%%%%

%%%%%%%%

\subsection{The local testing problem}
\label{appendix:local-testing}

In this appendix, we prove the bound
in~\eqref{eq:local-testing-upper-bound}, and
Proposition~\ref{prop:local-lag-h-autocovariance-lower-bound-test}.

%%%%%%%%%%%%%%%

\begin{proof}[Proof of~\eqref{eq:local-testing-upper-bound}]
Fix \(N\geq2\), \(\alpha>0\), \(h\in[N-1]\), \(r\in(0,1/2)\), and
\(f\in\mathcal F(C_0,C_1,s)\) satisfying the Sobolev constraint
strictly and such that \(f\geq f_{\min}\) for some \(f_{\min}>0\).
Let \(\varepsilon_{\mathrm{up}}\) be defined
by~\eqref{eq:local-testing-separation-rate}, and let \(f_+\) and \(f_-\)
denote the alternatives
in~\eqref{eq:local-lag-h-epsilon-sp-density-choice} corresponding to
\(\varepsilon=\varepsilon_{\mathrm{up}}\).

By construction, \(f_+\) and \(f_-\) are even and integrable, and
\[
f_\pm(\lambda)
\ \geq \
f_{\min}-\frac{\varepsilon_{\mathrm{up}}}{\pi},
\qquad \lambda\in[-\pi,\pi].
\]
Thus, when \(\varepsilon_{\mathrm{up}}\leq\pi f_{\min}\), the functions \(f_+\)
and \(f_-\) are nonnegative and therefore define spectral densities.
Moreover, since \(f\) satisfies the Sobolev constraint strictly, these
spectral densities belong to \(\mathcal F(C_0,C_1,s)\) when
\(\varepsilon_{\mathrm{up}}\) is sufficiently small.

Since \(\gamma_{f_\pm}(h)=\gamma_f(h)\pm\varepsilon_{\mathrm{up}}\), the definition
of \(\phi_f\) in~\eqref{eq:estimator_to_test} gives
\[
\begin{aligned}
\mathcal T_{N,\alpha}
\bigl(
\phi_f,f,\varepsilon_{\mathrm{up}},
K_{1:N}^{\mathrm{Lap},\alpha}
\bigr)
\ &= \
\max\pa{
\mathbb E_{f_-,K_{1:N}^{\mathrm{Lap},\alpha}}
\cro{\phi_f(Z)},
\mathbb E_{f_+,K_{1:N}^{\mathrm{Lap},\alpha}}
\cro{1-\phi_f(Z)}
}
\\
&\leq \
\max_{\pm\in\ac{-,+}}
P_{f_\pm,K_{1:N}^{\mathrm{Lap},\alpha}}
\left(
\left|
\widetilde\gamma(h)-\gamma_{f_\pm}(h)
\right|
\geq
\varepsilon_{\mathrm{up}}
\right).
\end{aligned}
\]
Then Markov's inequality yields
\[
\mathcal T_{N,\alpha}
\bigl(
\phi_f,f,\varepsilon_{\mathrm{up}},
K_{1:N}^{\mathrm{Lap},\alpha}
\bigr)
\ \leq \
\frac1{\varepsilon_{\mathrm{up}}^2}
\max_{\pm\in\ac{-,+}}
\mathbb E_{f_\pm,K_{1:N}^{\mathrm{Lap},\alpha}}
\left[
\left|
\widetilde\gamma(h)-\gamma_{f_\pm}(h)
\right|^2
\right].
\]
Since \(f_+\) and \(f_-\) belong to \(\mathcal F(C_0,C_1,s)\),
Proposition~\ref{prop:lag-h-autocovariance-upper-bound} then gives
\[
\mathcal T_{N,\alpha}
\bigl(
\phi_f,f,\varepsilon_{\mathrm{up}},
K_{1:N}^{\mathrm{Lap},\alpha}
\bigr)
\ \leq \
\frac{C_{C_0,C_1,s}}{\varepsilon_{\mathrm{up}}^2}
\left(
1\wedge
\left[(\alpha^4\wedge1)(N-h)\right]^{-1}
\right).
\]
Since
\(
\varepsilon_{\mathrm{up}}^2
=
C_{C_0,C_1,s,r}^2
\left(
1\wedge
\left[(\alpha^4\wedge1)(N-h)\right]^{-1}
\right)\),
we obtain
\[
\mathcal T_{N,\alpha}
\bigl(
\phi_f,f,\varepsilon_{\mathrm{up}},
K_{1:N}^{\mathrm{Lap},\alpha}
\bigr)
\ \leq \
\frac{C_{C_0,C_1,s}}{C_{C_0,C_1,s,r}^2}
\ \leq \
r,
\]
where the last inequality follows by choosing the constant
\(C_{C_0,C_1,s,r}\) sufficiently large.
This completes the proof of~\eqref{eq:local-testing-upper-bound}.
\end{proof}

%%%%%%%

\medskip

%%%%%%%%

\begin{proof}[Proof of Proposition~\ref{prop:local-lag-h-autocovariance-lower-bound-test}]
Fix \(r\in(0,1/2)\), \(h\geq1\), \(f_{\min}>0\), and a spectral
density \(f\) satisfying \(f\geq f_{\min}\).
For every \(\varepsilon\in(0,\pi f_{\min}]\), we have
\[
f_\pm(\lambda)
\ \geq \
f_{\min}-\frac{\varepsilon}{\pi}
\ \geq \
0,
\qquad \lambda\in[-\pi,\pi].
\]
Since \(f_\pm\) are also even and integrable, they are spectral
densities.

Fix \(N\geq h+1\), and \(\alpha>0\), and
\(K_{1:N}\in\mathcal M_\alpha^N\), and
\(\phi\in\mathcal E_{[0,1]}(K_{1:N})\).
By lower bounding the maximal risk by the Bayes risk with equal prior
probabilities, we obtain
\begin{align*}
\mathcal T_{N,\alpha}
\bigl(\phi,f,\varepsilon,K_{1:N}\bigr)
\ &\geq \
\frac12
\left(\mathbb E_{f_-,K_{1:N}}\cro{\phi(Z)}
+\mathbb E_{f_+,K_{1:N}}\cro{1-\phi(Z)}
\right)
\\
\ &= \
\frac12
\left(1-\int \phi(z) \; \pa{P_{f_+,K_{1:N}}^Z-P_{f_-,K_{1:N}}^Z}(d z)
\right).
\end{align*}
By the Hahn-Jordan decomposition, and since \(\phi\) takes values in
\([0,1]\), the integral above satisfies
\[
\int \phi(z)\,
\pa{P_{f_+,K_{1:N}}^Z-P_{f_-,K_{1:N}}^Z}(dz)
\ \leq \
\left\|
P_{f_+,K_{1:N}}^Z-P_{f_-,K_{1:N}}^Z
\right\|_{\mathrm{TV}},
\]
where the total variation distance is defined in~(\ref{defi-TV}).
Combining the last two displays and Pinsker's inequality, we obtain
\[
\mathcal T_{N,\alpha}
\bigl(\phi,f,\varepsilon,K_{1:N}\bigr)
\ \geq \
\frac12
\left(
1-
\frac1{\sqrt2}
\sqrt{
\D\!\left(
P_{f_+,K_{1:N}}^Z
\,\middle\|\,
P_{f_-,K_{1:N}}^Z
\right)
}
\right).
\]

We apply Theorem~\ref{thm:KL-bound-gaussian} with
\(a=\pi f_{\min}\) and
\(\delta=4\varepsilon/(\pi f_{\min})\).
For every \(\varepsilon\in(0,\pi f_{\min}/4]\), we have
\(\delta\in(0,1]\), and the conditions in~\eqref{eq:min-f_pm} hold.
Consequently,
\begin{equation}
\label{eq:local-testing-risk-lower-bound}
\mathcal T_{N,\alpha}
\bigl(\phi,f,\varepsilon,K_{1:N}\bigr)
\ \geq \
\frac12
\left(
1-
\frac{4\sqrt C\,\varepsilon}{\sqrt2\,\pi f_{\min}}
\sqrt{(\alpha^4\wedge1)(N-h)}
\right),
\end{equation}
under the condition that \(\varepsilon\in(0,\pi f_{\min}/4]\), 
where \(C\) is the universal constant in Theorem~\ref{thm:KL-bound-gaussian}.

Set
\[
c_{r,f_{\min}}
:=
\frac{\pi f_{\min}}4
\wedge
\frac{\sqrt2\,\pi f_{\min}(1-2r)}{4\sqrt C}.
\]
For 
\(\varepsilon_{\mathrm{low}} :=c_{r,f_{\min}} \bigl(1\wedge[(\alpha^4\wedge1)(N-h)]^{-1/2}\bigr)\),
we have
\(\varepsilon_{\mathrm{low}}\leq c_{r,f_{\min}}\leq\pi f_{\min}/4\), 
so the bound~\eqref{eq:local-testing-risk-lower-bound} applies.
Moreover,
\(\varepsilon_{\mathrm{low}}\sqrt{(\alpha^4\wedge1)(N-h)} \leq c_{r,f_{\min}}\), so the definition of \(c_{r,f_{\min}}\) and~\eqref{eq:local-testing-risk-lower-bound} yield
\[
\mathcal T_{N,\alpha}
\bigl(\phi,f,\varepsilon_{\mathrm{low}},K_{1:N}\bigr)
\ \geq \
r.
\]
Since \(K_{1:N}\) and \(\phi\) were arbitrary, taking the two infima
concludes the proof of Proposition~\ref{prop:local-lag-h-autocovariance-lower-bound-test}.
\end{proof}

%%%%%%%%%%%%%%%

%%%%%%%%%%%%%%

%%%%%%%%%%%%%%%%%%

\subsection{Asymptotic equivalence under local differential privacy}
\label{appendix:asymptotic-equivalence-ldp}

In this appendix, we first establish inclusions relating the parameter classes
\(\widetilde{\mathcal F}(M,s)\) to \(\mathcal F(C_0,C_1,s)\). 
We then prove Proposition~\ref{prop:golubev_etal-upm}, together with an auxiliary
lemma used in its proof, and finally Lemma~\ref{lem:TV-compatible}.

\begin{proof}[Comparison of parameter classes]
We show that \(\widetilde{\mathcal F}(M,s)\) is comparable, up to
changes in the constants, to uniformly lower-bounded subclasses of
\(\mathcal F(C_0,C_1,s)\) considered in the present paper. 
More precisely, for any \(C_0\geq1\) and \(C_1>0\) satisfying
\(C_0^2(1+C_1^2)\leq M\), the following inclusions hold:
\begin{equation}
\label{eq:different-classes}
\mathcal F(C_0,C_1,s)
\cap
\left\{f:f\geq\frac1M\right\}
\subset
\widetilde{\mathcal F}(M,s)
\subset
\mathcal F\left(
\frac{M}{2\pi}\vee\sqrt M,
\frac{M^{3/2}}{2\pi},
s
\right).
\end{equation}
Indeed, the first inclusion follows directly from
\[
\gamma_f^2(0)
\left(
1+\sum_{h\neq0}|h|^{2s}|\rho_f(h)|^2
\right)
\leq
C_0^2(1+C_1^2)
\leq M.
\]
Conversely, if \(f\in\widetilde{\mathcal F}(M,s)\), then
\(\frac{2\pi}{M}\leq\gamma_f(0)\leq\sqrt M\),
and
\[
\sum_{h\neq0}|h|^{2s}|\rho_f(h)|^2
\leq
\frac{M}{\gamma_f^2(0)}
\leq
\frac{M^3}{(2\pi)^2}.
\]
This proves the second inclusion in
\eqref{eq:different-classes}.
\end{proof}

%%%%%%
\medskip
%%%%%

%%%%%%%
%proof of main prop
%%%%%

\begin{proof}[Proof of Proposition~\ref{prop:golubev_etal-upm}]
Fix \(M>(2\pi)^{2/3}\) and \(s>1/2\). 
Choose
\[
c_M\in\left(\frac{1}{M},\frac{\sqrt M}{2\pi}\right),
\quad \text{ and } \quad
f(\lambda)=c_M,
\quad \lambda\in[-\pi,\pi].
\]
We consider the local testing construction
from~\eqref{eq:local-lag-h-epsilon-sp-density-choice} with lag \(h=1\), that is,
the two alternatives 
\[
f_\pm(\lambda) 
= 
f(\lambda)\pm\frac{\varepsilon}{\pi}\cos(\lambda)
=
c_M\pm\frac{\varepsilon}{\pi}\cos(\lambda).
\]
Let
\[\varepsilon_0:=
\pi\bigl(c_M-\frac{1}{M}\bigr)
\wedge
\sqrt{\frac{M-(2\pi c_M)^2}{2}}.\]
Then, for every \(\varepsilon\in(0,\varepsilon_0]\), the spectral
densities \(f_+\) and \(f_-\) belong to
\(\widetilde{\mathcal F}(M,s)\).
Indeed,
\[
f_\pm\geq c_M-\frac{\varepsilon}{\pi}\geq\frac1M,
\]
and
\[
\gamma_{f_\pm}^2(0)
\left(
1+\sum_{k\neq0}|k|^{2s}|\rho_{f_\pm}(k)|^2
\right)
=(2\pi c_M)^2+2\varepsilon^2
\leq M.
\]

The proof compares the difficulty of testing \(f_-\) against \(f_+\)
under LDP in the time series experiment \(\mathcal P_N\) and in the
independent experiment \(\mathcal Q_N\). We choose the separation so
that this testing problem remains difficult in \(\mathcal P_N\),
whereas the following lemma shows that it can be solved consistently
in \(\mathcal Q_N\). We will then translate this difference into a lower
bound on the deficiency.

\begin{lemma}
\label{lem:private-test-independent-model}
There exist constants \(C'_f>0\) and \(N_0\geq1\) such that, for every
\(N\geq N_0\), every \(\alpha>0\), and every
\(\varepsilon\in(0,\varepsilon_0]\), there exist
\(K_{1:N}\in\mathcal M_\alpha^N\) and
\(\phi_N\in\mathcal E_{[0,1]}(K_{1:N})\) satisfying
\[
\max\left\{
\int\phi_N\,d(K_{1:N}\circ Q_{f_-}^N),
\int(1-\phi_N)\,d(K_{1:N}\circ Q_{f_+}^N)
\right\}
\leq
\frac{C'_f}{\varepsilon^2(\alpha^2\wedge1)N}.
\]
\end{lemma}

The proof of Lemma~\ref{lem:private-test-independent-model} is given
after the present proof.

We now choose the separation so that the testing problem remains
difficult in the time series model. 
Fix \(r\in(0,1/2)\). Applying
Proposition~\ref{prop:local-lag-h-autocovariance-lower-bound-test}
with \(h=1\) and \(f_{\min}=1/M\), we may choose a constant
\(c_{r,M}\in(0,\varepsilon_0]\) such that, setting
\begin{equation}
\label{eq:epsilon-N-asymptotic-equivalence}
\varepsilon_N
:=
c_{r,M}
\left(
1\wedge(\alpha_N^4N)^{-1/2}
\right),
\end{equation}
we have, for all sufficiently large \(N\),
\begin{equation}
\label{eq:prelim-lem-golub-2}
\inf_{K'_{1:N}\in\mathcal M_{\alpha_N}^N}
\inf_{\widetilde\phi\in\mathcal E_{[0,1]}(K'_{1:N})}
\max\left\{
\int\widetilde\phi\,d(K'_{1:N}\circ P_{f_-}^N),
\int(1-\widetilde\phi)\,d(K'_{1:N}\circ P_{f_+}^N)
\right\}
\geq r,
\end{equation}
where \(f_-\) and \(f_+\) are taken with
\(\varepsilon=\varepsilon_N\).
Here, we used that \(\alpha_N<1\) for all sufficiently large \(N\)
and that \(N-1\geq N/2\), decreasing \(c_{r,M}\) if necessary.

For the same separation \(\varepsilon_N\) as in \eqref{eq:epsilon-N-asymptotic-equivalence}, we now consider the
independent model.
Applying Lemma~\ref{lem:private-test-independent-model} with
\(\alpha=\alpha_N\) and \(\varepsilon=\varepsilon_N\), we obtain that, for all
sufficiently large \(N\), there exist
\(K_{1:N}\in\mathcal M_{\alpha_N}^N\) and
\(\phi_N\in\mathcal E_{[0,1]}(K_{1:N})\) such that
\begin{equation*}
\max\left\{
\int\phi_N\,d(K_{1:N}\circ Q_{f_-}^N),
\int(1-\phi_N)\,d(K_{1:N}\circ Q_{f_+}^N)
\right\}
\leq
\frac{C'_f}
{\varepsilon_N^2\alpha_N^2N}.
\end{equation*}
Here again, we used that
\(\alpha_N<1\) for all sufficiently large \(N\).

We now compare the two experiments by transporting the test
\(\phi_N\) from the independent model to an arbitrary privatized
version of the time series model.
Fix \(K'_{1:N}\in\mathcal M_{\alpha_N}^N\), and let \(L'\) be any
Markov kernel from the output space of \(K'_{1:N}\) to that of
\(K_{1:N}\). We define
\[
\widetilde\phi_N(z')
:=
\int\phi_N(z)\,L'(dz\mid z').
\]
Since \(\phi_N\) takes values in \([0,1]\),
\(\widetilde\phi_N\in\mathcal E_{[0,1]}(K'_{1:N})\). Moreover, by the
definition of the composition of Markov kernels,
\[
\int\widetilde\phi_N(z')\,
(K'_{1:N}\circ P_{f_-}^N)(dz')
=
\int\phi_N(z)\,
(L'\circ K'_{1:N}\circ P_{f_-}^N)(dz),
\]
and the analogous identity holds with
\(\widetilde\phi_N\), \(f_-\), and \(\phi_N\) replaced by
\(1-\widetilde\phi_N\), \(f_+\), and \(1-\phi_N\), respectively.
Therefore, \eqref{eq:prelim-lem-golub-2} gives, for all sufficiently
large \(N\),
\[
\inf_{K'_{1:N}\in\mathcal M_{\alpha_N}^N}
\inf_{L'}\
\max\left\{
\int\phi_N(z)\,
(L'\circ K'_{1:N}\circ P_{f_-}^N)(dz),
\int(1-\phi_N(z))\,
(L'\circ K'_{1:N}\circ P_{f_+}^N)(dz)
\right\}
\geq r,
\]
where the second infimum is over all Markov kernels \(L'\) between the
corresponding output spaces.

Combining the last two bounds, and using
\(\max\{a_- - b_-,a_+ - b_+\}
\geq \max\{a_-,a_+\}-\max\{b_-,b_+\}\)
we obtain, for all sufficiently large \(N\),
\begin{align}
\label{eq:deficiency-testing-lower-bound}
&\inf_{K'_{1:N}\in\mathcal M_{\alpha_N}^N}
\inf_{L'}
\max\Bigg\{
\int \phi_N(z)\,
\bigl[
L'\circ K'_{1:N}\circ P_{f_-}^N
-
K_{1:N}\circ Q_{f_-}^N
\bigr](dz),
\nonumber\\
&\hspace{45mm}
\int \phi_N(z)\,
\bigl[
K_{1:N}\circ Q_{f_+}^N
-
L'\circ K'_{1:N}\circ P_{f_+}^N
\bigr](dz)
\Bigg\}
\nonumber\\
&\hspace{25mm}\geq
r-
\frac{C'_f}{\varepsilon_N^2\alpha_N^2N}
=
r-
\frac{C'_f}{c_{r,M}^2}
\left(
\frac{1}{\alpha_N^2N}\vee\alpha_N^2
\right),
\end{align}
where \(L'\) ranges over all Markov kernels from the output space of
\(K'_{1:N}\) to that of \(K_{1:N}\).

The right-hand side of \eqref{eq:deficiency-testing-lower-bound} converges to \(r\) as \(N\to\infty\) by the
assumptions on \((\alpha_N)_{N\geq1}\). On the other hand, since
\(\phi_N\) takes values in \([0,1]\), each integral on the left-hand
side is bounded above by
\(\|L'\circ K'_{1:N}\circ P_{f_\pm}^N-
K_{1:N}\circ Q_{f_\pm}^N\|_{\mathrm{TV}}\).
Hence, the left-hand side is bounded above by
\(\mathcal D_{\mathrm{LC},\alpha_N}
(\mathcal P_N\|\mathcal Q_N)\). Since \(r\) is arbitrary in
\((0,1/2)\), this proves Proposition~\ref{prop:golubev_etal-upm}.
\end{proof}

\medskip 

%%%%%

\begin{proof}[Proof of Lemma~\ref{lem:private-test-independent-model}]
We construct the mechanism \(K_{1:N}\) and the test
\(\phi_N\) explicitly. Set
\[
C_f
:=
M\vee\left(c_M+\frac{\varepsilon_0}{\pi}\right).
\]
For every \(\varepsilon\in(0,\varepsilon_0]\),
\[
C_f^{-1}
\leq
f(\lambda),f_-(\lambda),f_+(\lambda)
\leq
C_f,
\qquad \lambda\in[-\pi,\pi].
\]
Since \(J_{j,N}(f)\) and \(J_{j,N}(f_\pm)\) are local averages of
\(f\) and \(f_\pm\), respectively, it follows that, for
every \(N\geq1\) and \(j\in[N]\),
\[
C_f^{-1}
\leq
J_{j,N}(f),J_{j,N}(f_-),J_{j,N}(f_+)
\leq
C_f.
\]

Let \(\mathbb Q_\pm\) denote probability measures under which
\(Y_{1:N}\sim Q_{f_\pm}^N\) and \(W_1,\ldots,W_N\) are i.i.d.\
standard Laplace random variables independent of \(Y_{1:N}\).
We denote the corresponding expectation and variance by
\(\mathbb E_\pm\) and \(\operatorname{Var}_\pm\), respectively.

We define the privatized observations \(Z_1,\ldots,Z_N\) by
\begin{equation}
\label{defi-K_lap_asymp_equiv}
Z_j
=
\mathbf 1_{\{|Y_j|>C_f\}}
-
Q_f^N\bigl(|Y_j|>C_f\bigr)
+
\frac{3}{\alpha}W_j,
\qquad j\in[N],
\end{equation}
where \(\mathbf 1_A\) denotes the indicator function of an event
\(A\).
(Here, the non-optimal factor \(3\) is retained for consistency with the main
privacy mechanism introduced in Section~\ref{section-estimators-upper-bounds}.)
Using these privatized observations, we compute 
\[
S_N:=\sum_{j=1}^N w_{j,N}Z_j \;, \quad \text{ with } \quad 
w_{j,N}
:=
\frac{N}{2\pi^2}
\int_{2(j-1)\pi/N}^{2j\pi/N}
\cos(\lambda)\,d\lambda\;,
\qquad j\in[N],
\]
and define the test function
\begin{equation}
\label{eq:independent-model-test}
\phi_N(Z_{1:N})
:=
\mathbf 1_{\{S_N>0\}}.
\end{equation}

Let \(K_{1:N}\) denote the mechanism defined by
\eqref{defi-K_lap_asymp_equiv}. Since each \(Z_j\) depends only on
\(Y_j\), and the indicator is perturbed by Laplace noise with scale
\(3/\alpha\), we have
\(K_{1:N}\in\mathcal M_\alpha^N\). Moreover,
\begin{equation}
\label{eq:law-private-independent-model}
\mathbb Q_\pm^{Z_{1:N}}
=
K_{1:N}\circ Q_{f_\pm}^N.
\end{equation}
By construction,
\(\phi_N\in\mathcal E_{[0,1]}(K_{1:N})\).

For \(j\in[N]\), define
\[
q_{j,N}(f_\pm)
:=
Q_{f_\pm}^N\bigl(|Y_j|>C_f\bigr)
-
Q_f^N\bigl(|Y_j|>C_f\bigr).
\]
Using the definition \eqref{defi-K_lap_asymp_equiv} of \(Z_j\), the fact that a standard Laplace
random variable is centered with variance \(2\), and the independence
of \(Z_1,\ldots,Z_N\), we obtain
\begin{equation}
\label{eq:SN-esp_and_var}
\mathbb E_\pm[S_N]
=
\sum_{j=1}^N q_{j,N}(f_\pm)w_{j,N},
\qquad
\operatorname{Var}_\pm(S_N)
=
\sum_{j=1}^N w_{j,N}^2
\left[
\operatorname{Var}_\pm
\left(\mathbf 1_{\{|Y_j|>C_f\}}\right)
+
\frac{18}{\alpha^2}
\right].
\end{equation}

To control the expectation of \(S_N\), first observe that
\[
q_{j,N}(f_\pm)
=
2\left[
\Phi\left(\frac{C_f}{\sqrt{J_{j,N}(f)}}\right)
-
\Phi\left(\frac{C_f}{\sqrt{J_{j,N}(f_\pm)}}\right)
\right].
\]
Then, since the function
\(u\mapsto\Phi(C_f/\sqrt u)\) is strictly decreasing on
\((0,\infty)\), the quantity \(q_{j,N}(f_\pm)\) has the same sign as \(J_{j,N}(f_\pm)-J_{j,N}(f)\). By the linearity of
\(J_{j,N}\),
\[
J_{j,N}(f_\pm)-J_{j,N}(f)
=
\pm\varepsilon w_{j,N},
\]
and consequently,
\begin{equation}
\label{eq:q-sign-independent-model}
\operatorname{sgn}\bigl(q_{j,N}(f_\pm)\bigr)
=
\pm\operatorname{sgn}(w_{j,N}).
\end{equation}

Moreover, the derivative of
\(u\mapsto\Phi(C_f/\sqrt u)\) is continuous and strictly negative on
the compact interval \([C_f^{-1},C_f]\). Hence, by the mean value
theorem, there exists a constant \(c_f>0\) such that
\[
|q_{j,N}(f_\pm)|
\geq
c_f\left|J_{j,N}(f_\pm)-J_{j,N}(f)\right|
=
c_f\varepsilon|w_{j,N}|.
\]
Combining this bound with \eqref{eq:q-sign-independent-model} and
\eqref{eq:SN-esp_and_var}, we obtain
\begin{equation}
\label{eq:SN-expectation-separation}
\mathbb E_+[S_N]
\geq
c_f\varepsilon\sum_{j=1}^Nw_{j,N}^2,
\qquad
\mathbb E_-[S_N]
\leq
-c_f\varepsilon\sum_{j=1}^Nw_{j,N}^2.
\end{equation}

Define
\begin{equation}
\label{eq:CN-independent-model}
C_N
:=
\frac{2\pi}{N}\sum_{j=1}^Nw_{j,N}^2.
\end{equation}
The bounds \eqref{eq:SN-expectation-separation} on the expectations can then be
written as
\[
\mathbb E_+[S_N]
\geq
c_f\varepsilon\frac{NC_N}{2\pi},
\qquad
\mathbb E_-[S_N]
\leq
-c_f\varepsilon\frac{NC_N}{2\pi}.
\]
Moreover, since Bernoulli random variables have variance at most \(1/4\), the variance in
\eqref{eq:SN-esp_and_var} satisfies
\[
\operatorname{Var}_\pm(S_N)
\leq
\left(\frac14+\frac{18}{\alpha^2}\right)
\sum_{j=1}^Nw_{j,N}^2
=
\frac{NC_N}{2\pi}
\left(\frac14+\frac{18}{\alpha^2}\right).
\]
Therefore, since \(\mathbb E_+[S_N]>0\) for \(N\geq3\), Chebyshev's inequality and the definition \eqref{eq:independent-model-test} of \(\phi_N(Z_{1:N})\) give
\begin{align}
\label{eq:independent-model-type-II-error}
\mathbb Q_+\bigl(\phi_N(Z_{1:N})=0\bigr)
=
\mathbb Q_+(S_N\leq0)
&\leq
\mathbb Q_+\left(
\left|S_N-\mathbb E_+[S_N]\right|
\geq
\mathbb E_+[S_N]
\right)\nonumber\\
&\leq
\frac{\operatorname{Var}_+(S_N)}
     {\mathbb E_+[S_N]^2}\nonumber\\
&\leq
\frac{2\pi}{c_f^2C_N}
\frac{1}{\varepsilon^2N}
\left(\frac14+\frac{18}{\alpha^2}\right).
\end{align}
The same argument, using \(\mathbb E_-[S_N]<0\), gives
\begin{equation}
\label{eq:independent-model-type-I-error}
\mathbb Q_-\bigl(\phi_N(Z_{1:N})=1\bigr)
=
\mathbb Q_-(S_N>0)
\leq
\frac{2\pi}{c_f^2C_N}
\frac{1}{\varepsilon^2N}
\left(\frac14+\frac{18}{\alpha^2}\right).
\end{equation}

Finally, since \(w_{j,N}\) is the average of
\(\lambda\mapsto\cos(\lambda)/\pi\) over
\([2(j-1)\pi/N,2j\pi/N]\), there exists a universal constant \(C>0\)
such that
\[
\left|
w_{j,N}
-
\frac{\cos(2j\pi/N)}{\pi}
\right|
\leq
\frac{C}{N}.
\]
Since both terms in the difference are uniformly bounded,
the inequality \(|x^2-y^2|\leq|x-y|(|x|+|y|)\) implies that \(C_N\),
defined in \eqref{eq:CN-independent-model}, satisfies
\[
\left|
C_N-
\frac{2\pi}{N}\sum_{j=1}^N
\left(\frac{\cos(2j\pi/N)}{\pi}\right)^2
\right|
\leq
\frac{2\pi}{N}\sum_{j=1}^N
\left|
w_{j,N}^2-
\left(\frac{\cos(2j\pi/N)}{\pi}\right)^2
\right|
\leq\frac{C}{N}.
\]
Then
\[
\left|
C_N
-
\int_0^{2\pi}
\left(\frac{\cos(u)}{\pi}\right)^2du
\right|
\leq
\frac{C}{N},
\]
for some absolute constant \(C\).
Since the integral is equal to \(1/\pi\), we have
\(C_N\to1/\pi\) as \(N\to \infty\). In particular, \(C_N\geq1/(2\pi)\) for all
sufficiently large \(N\).
Combining this fact with \eqref{eq:independent-model-type-I-error} and \eqref{eq:independent-model-type-II-error}, and using
\[
\frac14+\frac{18}{\alpha^2}
\leq
\frac{C'}{\alpha^2\wedge1}
\]
for some absolute constant \(C'\),
we conclude that there exists a constant \(C'_f>0\) such that, for all
sufficiently large \(N\),
\[
\max\left\{
\mathbb Q_-\bigl(\phi_N(Z_{1:N})=1\bigr),
\mathbb Q_+\bigl(\phi_N(Z_{1:N})=0\bigr)
\right\}
\leq
\frac{C'_f}{\varepsilon^2(\alpha^2\wedge1)N}.
\]
Finally, recalling from \eqref{eq:law-private-independent-model} that
\(\mathbb Q_\pm^{Z_{1:N}}=K_{1:N}\circ Q_{f_\pm}^N\), the left-hand
side is equal to
\[
\max\left\{
\int\phi_N\,d(K_{1:N}\circ Q_{f_-}^N),
\int(1-\phi_N)\,d(K_{1:N}\circ Q_{f_+}^N)
\right\}.
\]
This proves Lemma~\ref{lem:private-test-independent-model}.
\end{proof}

\medskip 

%%%%%

\begin{proof}[Proof of Lemma~\ref{lem:TV-compatible}]
By the definition of
\(\mathcal D_{\mathrm{LC},\alpha}\), it is enough to show that, for
every \(K_{1:N}\in\mathcal M_\alpha^N\) and every coordinatewise
Markov kernel \(L_{1:N}\) on \(\mathbb R^N\), there exist
\(K'_{1:N}\in\mathcal M_\alpha^N\) and a Markov kernel \(L'\) from the
output space of \(K'_{1:N}\) to that of \(K_{1:N}\) such that
\[
\sup_{\theta\in\Theta}
\left\|
L'\circ K'_{1:N}\circ P_\theta^N
-
K_{1:N}\circ Q_\theta^N
\right\|_{\mathrm{TV}}
\leq
\sup_{\theta\in\Theta}
\left\|
L_{1:N}\circ P_\theta^N-Q_\theta^N
\right\|_{\mathrm{TV}}.
\]
Fix such \(K_{1:N}\) and \(L_{1:N}\). We then take
\(K'_{1:N}:=K_{1:N}\circ L_{1:N}\) and let \(L'\) be the identity
kernel on the output space of \(K_{1:N}\). It remains to verify that
\(K'_{1:N}\in\mathcal M_\alpha^N\) and to apply the contraction of
total variation under Markov kernels.

Write \(K_{1:N}=(K_1,\ldots,K_N)\). Since both \(K_{1:N}\) and
\(L_{1:N}\) act coordinatewise, we have
\[
K'_{1:N}
=
K_{1:N}\circ L_{1:N}
=
(K_1\circ L_1,\ldots,K_N\circ L_N).
\]
Fix \(i\in[N]\), \(x,x'\in\mathbb R\), and a measurable set \(A\) in
the output space of \(K_i\). Since \(K_i\) is \(\alpha\)-LDP,
\(K_i(A\mid y)\leq e^\alpha K_i(A\mid y')\) for every
\(y,y'\in\mathbb R\). Integrating this inequality with respect to
\(L_i(dy\mid x)\) and \(L_i(dy'\mid x')\), we obtain
\begin{align*}
(K_i\circ L_i)(A\mid x)
&=
\int K_i(A\mid y)\,L_i(dy\mid x)\\
&\leq
e^\alpha
\int K_i(A\mid y')\,L_i(dy'\mid x')\\
&=
e^\alpha(K_i\circ L_i)(A\mid x').
\end{align*}
Thus, each \(K_i\circ L_i\) is \(\alpha\)-LDP, and hence
\(K'_{1:N}\in\mathcal M_\alpha^N\).

We now prove that total variation contracts under Markov kernels.
For any two probability measures \(\mu\) and \(\nu\), any
Markov kernel \(K\), and any measurable set \(A\) in the output space
of \(K\), 
\[
(K\circ\mu)(A)-(K\circ\nu)(A)
=
\int K(A\mid x)\,(\mu-\nu)(dx).
\]
Since the function \(x\mapsto K(A\mid x)\) takes values in
\([0,1]\), we have
\[
\left\|K\circ\mu-K\circ\nu\right\|_{\mathrm{TV}}
\leq
\left\|\mu-\nu\right\|_{\mathrm{TV}}.
\]
Applying this inequality with
\(\mu=L_{1:N}\circ P_\theta^N\), and \(\nu=Q_\theta^N\), and
\(K=K_{1:N}\), we obtain, for every \(\theta\in\Theta\),
\[
\left\|
L'\circ K'_{1:N}\circ P_\theta^N
-
K_{1:N}\circ Q_\theta^N
\right\|_{\mathrm{TV}}
=
\left\|
K_{1:N}\circ L_{1:N}\circ P_\theta^N
-
K_{1:N}\circ Q_\theta^N
\right\|_{\mathrm{TV}}
\leq
\left\|
L_{1:N}\circ P_\theta^N-Q_\theta^N
\right\|_{\mathrm{TV}}.
\]
Taking the supremum over \(\theta\in\Theta\) and then the infimum over
all coordinatewise Markov kernels \(L_{1:N}\) proves Lemma~\ref{lem:TV-compatible}.    
\end{proof}

\section{Additional proofs for Theorem~\ref{thm-borne-inf:cas-non-interactif}}
\label{appendix:lower-bound-thm:proofs-lemmas}

This appendix contains the arguments deferred from Section~\ref{subsection:proof-lower-bound-thm}.
We first prove Lemma~\ref{lem:hypercube-reduction} and then Lemma~\ref{lem:choice-parameters-lower-bound-bis}.

%%%%%%%%

\begin{proof}[Proof of Lemma~\ref{lem:hypercube-reduction}]
We first verify that the functions \(f_{\boldsymbol{\omega}}\) belong to \(\mathcal F(1,C_1,s)\).
For every \(\lambda\in[-\pi,\pi]\),
\[
f_{\boldsymbol{\omega}}(\lambda)
\ \ge\
\frac{1}{2\pi}
\left(
1-2\varepsilon H
\right)
\ \ge\
\frac{1}{4\pi},
\]
where the last inequality follows from \(\varepsilon H\le1/4\).
Thus, \(f_{\boldsymbol{\omega}}\) is nonnegative.
It is also an even function, since it is a finite linear combination of cosine functions.
Moreover, the cosine terms integrate to zero, so
\[
\int_{-\pi}^{\pi}
f_{\boldsymbol{\omega}}(\lambda)\,d\lambda
\ =  \ 1.
\]
Hence, \(f_{\boldsymbol{\omega}}\) is a spectral density with variance \(\gamma_{f_{\boldsymbol{\omega}}}(0)=1\).

The autocovariance coefficients of \(f_{\boldsymbol{\omega}}\) are
\[
\gamma_{f_{\boldsymbol{\omega}}}(\pm h)
\ = \
\varepsilon\omega_h,
\qquad 1\le h\le H,
\]
and vanish at all lags \(|h|\ge H+1\).
Since the variance is \(1\), these coefficients are also the autocorrelation coefficients \(\rho_{f_{\boldsymbol{\omega}}}(h)\). 
Therefore,
\[
\sum_{h\in\mathbb Z}
|h|^{2s}
\left|
\rho_{f_{\boldsymbol{\omega}}}(h)
\right|^2
\ = \
2\varepsilon^2
\sum_{h=1}^H h^{2s}
\ \le\
2\varepsilon^2H^{2s+1}
\ \le\
\varepsilon^2(2H)^{2s+1}
\ \le\
C_1^2.
\]
This proves that \(f_{\boldsymbol{\omega}}\in\mathcal F(1,C_1,s)\).

We now prove the reduction to Hamming risk.
Since
\(\left\{
f_{\boldsymbol{\omega}}
:
\boldsymbol{\omega}\in\{-1,1\}^H
\right\}
\subseteq
\mathcal F(1,C_1,s)\),
the definition of the minimax risk gives
\[\mathcal R_{N,\alpha}(1,C_1,s)
\ \ge\
\inf_{K_{1:N}\in\mathcal M_\alpha^N}
\inf_{\hat f\in\mathcal E_{L^2}(K_{1:N})}
\sup_{\boldsymbol{\omega}\in\{-1,1\}^H}
\mathbb E_{f_{\boldsymbol{\omega}},K_{1:N}}
\left[
\left\|
\hat f(Z)-f_{\boldsymbol{\omega}}
\right\|^2
\right].\]

Fix \(K_{1:N}\in\mathcal M_\alpha^N\) and \(\hat f\in\mathcal E_{L^2}(K_{1:N})\).
For every \(z\) in the output space of \(K_{1:N}\), let \(\hat{\boldsymbol{\omega}}_{\hat f}(z)\) be a minimizer of
\[\boldsymbol{\omega}'
\longmapsto
\left\|
\hat f(z)-f_{\boldsymbol{\omega}'}
\right\|\]
over \(\boldsymbol{\omega}'\in\{-1,1\}^H\).
Note that \(\hat{\boldsymbol{\omega}}_{\hat f}\) is a measurable function from the output space of \(K_{1:N}\) to \(\{-1,1\}^H\).

For every \(\boldsymbol{\omega}\in\{-1,1\}^H\), the definition of \(\hat{\boldsymbol{\omega}}_{\hat f}(z)\) gives \(\|\hat f(z)-f_{\hat{\boldsymbol{\omega}}_{\hat f}(z)}\|
\le\|\hat f(z)-f_{\boldsymbol{\omega}}\|\).
Therefore, by the triangle inequality,
\[
\left\|
f_{\hat{\boldsymbol{\omega}}_{\hat f}(z)}
-f_{\boldsymbol{\omega}}
\right\|
\ \le \
\left\|
f_{\hat{\boldsymbol{\omega}}_{\hat f}(z)}
-\hat f(z)
\right\|
+
\left\|
\hat f(z)-f_{\boldsymbol{\omega}}
\right\|
\ \le \
2
\left\|
\hat f(z)-f_{\boldsymbol{\omega}}
\right\|.
\]
On the other hand, the orthogonality of the cosine functions yields
\[
\left\|
f_{\hat{\boldsymbol{\omega}}_{\hat f}(z)}
-f_{\boldsymbol{\omega}}
\right\|^2
\ = \
\frac{\varepsilon^2}{\pi}
\sum_{h=1}^H
\left|
\hat\omega_{\hat f,h}(z)-\omega_h
\right|^2
\ = \
\frac{4\varepsilon^2}{\pi}
d_H\!\left(
\hat{\boldsymbol{\omega}}_{\hat f}(z),
\boldsymbol{\omega}
\right).
\]
Consequently,
\[
\left\|
\hat f(z)-f_{\boldsymbol{\omega}}
\right\|^2
\ \ge\
\frac{\varepsilon^2}{\pi}
d_H\!\left(
\hat{\boldsymbol{\omega}}_{\hat f}(z),
\boldsymbol{\omega}
\right).
\]

Taking expectations and then the corresponding supremum and infima, we obtain
\[
\begin{aligned}
\mathcal R_{N,\alpha}(1,C_1,s)
\ &\ge \
\frac{\varepsilon^2}{\pi}
\inf_{K_{1:N}\in\mathcal M_\alpha^N}
\inf_{\hat f\in\mathcal E_{L^2}(K_{1:N})}
\sup_{\boldsymbol{\omega}\in\{-1,1\}^H}
\mathbb E_{f_{\boldsymbol{\omega}},K_{1:N}}
\left[
d_H\!\left(
\hat{\boldsymbol{\omega}}_{\hat f}(Z),
\boldsymbol{\omega}
\right)
\right]
\\
\ &\ge \
\frac{\varepsilon^2}{\pi}
\inf_{K_{1:N}\in\mathcal M_\alpha^N}
\inf_{\hat{\boldsymbol{\omega}}}
\sup_{\boldsymbol{\omega}\in\{-1,1\}^H}
\mathbb E_{f_{\boldsymbol{\omega}},K_{1:N}}
\left[
d_H\!\left(
\hat{\boldsymbol{\omega}}(Z),
\boldsymbol{\omega}
\right)
\right],
\end{aligned}
\]
where the second infimum in the last line is over all measurable functions from the output space of \(K_{1:N}\) to \(\{-1,1\}^H\).
Indeed, these estimators form a larger class than the estimators \(\hat{\boldsymbol{\omega}}_{\hat f}\) constructed from \(\hat f\), so the corresponding infimum can only be smaller.
This proves the lower bound of Lemma~\ref{lem:hypercube-reduction}.
\end{proof}

%%%%%
\bigskip
%%%%%%

\begin{proof}[Proof of Lemma~\ref{lem:choice-parameters-lower-bound-bis}]
We prove the claimed bounds in the following order:
\begin{align}
\varepsilon_*
&\leq \frac1{12},
\label{eq:choice-parameters-lower-bound-bis1}
\\
H_*
&\geq 1,
\label{eq:choice-parameters-lower-bound-bis2}
\\
\varepsilon_*H_*
&\leq\frac14,
\label{eq:choice-parameters-lower-bound-bis3}
\\
\varepsilon_*^2(2H_*)^{2s+1}
&\leq C_1^2,
\label{eq:choice-parameters-lower-bound-bis4}
\\
144C(\alpha^4\wedge1)N\varepsilon_*^2
&\leq\frac12,
\label{eq:choice-parameters-lower-bound-bis5}
\\
\frac{\varepsilon_*^2H_*}{4\pi}
&\geq
c_{C_1,s}
\left(
1\wedge
\left[(\alpha^4\wedge1)N\right]^{-\frac{2s}{2s+1}}
\right).
\label{eq:choice-parameters-lower-bound-bis6}
\end{align}
The bound~\eqref{eq:choice-parameters-lower-bound-bis1} is immediate
from the definition of \(\varepsilon_*\): its first factor is at most
\(1/12\), while its second factor is at most \(1\).

Similarly, from the definition of \(\varepsilon_*\), we have
\[
\varepsilon_*^2
\ \leq \
\left(\frac{C_1^2}{2^{2s+1}}\right)
\wedge
2^{-2(2s+1)}
\ = \
\frac{C_1^2\wedge2^{-(2s+1)}}{2^{2s+1}}.
\]
Applying this bound in the definition of \(H_*\) yields~(\ref{eq:choice-parameters-lower-bound-bis2}).

From the definition of \(H_*\), we have
\[
H_*
\ \leq \
\frac12
\left(
\frac{2^{-(2s+1)}}{\varepsilon_*^2}
\right)^{1/(2s+1)}
\ = \
\frac14\,
\varepsilon_*^{-2/(2s+1)}.
\]
Therefore,
\[
\varepsilon_*H_*
\ \leq \
\frac14\,
\varepsilon_*^{1-\frac{2}{2s+1}}
\ \leq \
\frac14,
\]
where the last inequality follows from
\(\varepsilon_*\leq1\) and \(s>1/2\).
This proves~\eqref{eq:choice-parameters-lower-bound-bis3}.

From the definition of \(H_*\), we have
\[
2H_*
\ \leq \
\left(
\frac{C_1^2}{\varepsilon_*^2}
\right)^{1/(2s+1)}.
\]
Raising both sides to the power \(2s+1\) and multiplying by
\(\varepsilon_*^2\), we obtain~\eqref{eq:choice-parameters-lower-bound-bis4}.

Bounding the first factor in the definition of \(\varepsilon_*\) by
\(1/(12\sqrt{2C})\) and the second one by
\(1/((\alpha^2\wedge1)\sqrt N)\), we obtain
\[
\varepsilon_*^2
\ \leq \
\frac{1}{288C(\alpha^4\wedge1)N},
\]
which proves~\eqref{eq:choice-parameters-lower-bound-bis5}.

We finally prove~\eqref{eq:choice-parameters-lower-bound-bis6}.
Throughout the remainder of the proof, \(c_{C_1,s}>0\) denotes a
constant depending only on \(C_1\) and \(s\), whose value may change
from line to line. Recall that \(C\) is universal and can therefore be
absorbed into this constant.
Set
\[
r_{\alpha,N}
\ := \
1\wedge\frac{1}{(\alpha^4\wedge1)N}.
\]
By the definition of \(\varepsilon_*\), we have
\(\varepsilon_*^2
 =
c_{C_1,s}\,r_{\alpha,N}.
\)
Moreover, since \(H_*\geq1\) and
\(\lfloor x\rfloor\geq x/2\) for every \(x\geq1\), the definition of
\(H_*\) gives
\(H_*
 \geq 
c_{C_1,s}\,
\varepsilon_*^{-2/(2s+1)}
\).
Consequently,
\[
\frac{\varepsilon_*^2H_*}{4\pi}
\ \geq \
c_{C_1,s}\,
\varepsilon_*^{\frac{4s}{2s+1}}
\ = \
c_{C_1,s}\,
r_{\alpha,N}^{\frac{2s}{2s+1}}.
\]
By the definition of \(r_{\alpha,N}\),
this proves~\eqref{eq:choice-parameters-lower-bound-bis6} and concludes the proof of Lemma~\ref{lem:choice-parameters-lower-bound-bis}.
\end{proof}

%%% Local Variables:
%%% mode: latex
%%% TeX-master: "main.tex"
%%% ispell-local-dictionary: "american"
%%% End:

%%%%%%%

\section{Additional proofs for Theorem~\ref{thm:f-estimator}}
\label{appendix:upper-bound:additional-proofs}

This appendix contains the proofs of Lemmas~\ref{lem:rho-sum-uniform}, \ref{lem:rho-away-from-one}, and~\ref{lem:variance-check-gamma}, which are used in
Section~\ref{section:proof-sketch-upper-bound-thm}.

\begin{proof}[Proof of Lemma~\ref{lem:rho-sum-uniform}]
The first inequality follows directly from the definition of
\(\mathcal F(C_0,C_1,s)\), since \(|h|^{2s}\geq1\) for every
\(h\in\mathbb Z\setminus\{0\}\).
For the second inequality, the Cauchy--Schwarz inequality gives
\[
\sum_{h\in\mathbb Z\setminus\{0\}}|\rho_f(h)|
\ \leq \
\left(
2\sum_{h=1}^{\infty}h^{-2s}
\right)^{1/2}
\left(
\sum_{h\in\mathbb Z\setminus\{0\}}
|h|^{2s}|\rho_f(h)|^2
\right)^{1/2}.
\]
Since \(s>1/2\), the first factor equals \(\sqrt{2\zeta(2s)}\), while
the second is at most \(C_1\). This proves the second inequality of Lemma~\ref{lem:rho-sum-uniform}.
\end{proof}

%%%%%
\medskip
%%%%%

\begin{proof}[Proof of Lemma~\ref{lem:rho-away-from-one}]
For \(f\in\mathcal F(C_0,C_1,s)\), set
\(\overline f := f / \gamma_f(0)\).
Then \(\overline f\in\mathcal F(1,C_1,s)\) and
\(\rho_{\overline f}=\rho_f\). It therefore suffices to prove the result
over \(\mathcal F(1,C_1,s)\).

For every \(f\in\mathcal F(1,C_1,s)\) and every
\(h\in\mathbb Z\setminus\{0\}\), the definition of the parameter class
gives
\[
|\rho_f(h)|
\ \leq \
C_1|h|^{-s}.
\]
Consequently, if \(|h|\geq(2C_1)^{1/s}\), then
\(|\rho_f(h)|\leq1/2\).

It remains to consider the finitely many nonzero lags satisfying
\(|h|<(2C_1)^{1/s}\). For each such fixed \(h\), we will show that
\begin{equation}\label{rho-smaller-than-1}
\sup_{f\in\mathcal F(1,C_1,s)}
|\rho_f(h)|
\ < \
1.
\end{equation}
The bound~\eqref{rho-smaller-than-1} is a consequence of the following three facts:
\begin{enumerate}[label=(\roman*)]
\item\label{item:rho-away-from-one1}
The map \(f\mapsto\gamma_f(h)\) is continuous on
\(L^2([-\pi,\pi])\).
\item\label{item:rho-away-from-one2}
The maps \(f\mapsto1-\gamma_f(h)\) and
\(f\mapsto1+\gamma_f(h)\) do not vanish on
\(\mathcal F(1,C_1,s)\).
\item\label{item:rho-away-from-one3}
The class \(\mathcal F(1,C_1,s)\) is a compact subset of
\(L^2([-\pi,\pi])\).
\end{enumerate}
Indeed, these three facts imply that the two maps
\(f\mapsto1-\gamma_f(h)\) and \(f\mapsto1+\gamma_f(h)\) attain strictly
positive minima on \(\mathcal F(1,C_1,s)\). 
Since \(\rho_f(h)=\gamma_f(h)\) for all \(f\in\mathcal F(1,C_1,s)\), this proves
\eqref{rho-smaller-than-1}.

To conclude the proof, we verify (i), (ii) and (iii).
For Assertion~\ref{item:rho-away-from-one1}, recall that
\begin{equation}
\label{reminder-gamma_f}
    \gamma_f(h) \ = \ \int_{-\pi}^{\pi}f(\lambda)e^{ih\lambda}\,d\lambda\;.
\end{equation}
Consequently, the map \(f\mapsto\gamma_f(h)\) is continuous, as the
Cauchy--Schwarz inequality yields, for all
\(f,g\in L^2([-\pi,\pi])\),
\[
|\gamma_f(h)-\gamma_g(h)|
\ \leq \
\sqrt{2\pi}\,\|f-g\|.
\]

For Assertion~\ref{item:rho-away-from-one2}, the relation~\eqref{reminder-gamma_f} and the evenness of \(f\) give
\[
\gamma_f(0)\pm\gamma_f(h)
\ = \
\int_{-\pi}^{\pi}
f(\lambda)\bigl(1\pm\cos(h\lambda)\bigr)\,d\lambda.
\]
For \(h\neq0\), we have \(1\pm\cos(h\lambda)>0\) for almost every
\(\lambda\). Since \(f\) is nonnegative, each of the two integrals can vanish
only if \(f=0\) almost everywhere, which would contradict
\(\gamma_f(0)=1\). Hence \(1\pm\gamma_f(h)>0\) on
\(\mathcal F(1,C_1,s)\).

For Assertion~\ref{item:rho-away-from-one3}, the constraint
\(\sum_{h\in\mathbb Z\setminus\{0\}}|h|^{2s}|\gamma_f(h)|^2\leq C_1^2\)
places \(\mathcal F(1,C_1,s)\) inside a Sobolev ball. Since \(s>0\), such balls are compact in \(L^2([-\pi,\pi])\).
Moreover, \(\mathcal F(1,C_1,s)\) is closed in
\(L^2([-\pi,\pi])\), since its defining conditions are preserved under
\(L^2\)-limits. Hence, \(\mathcal F(1,C_1,s)\) is compact.
\end{proof}

%%%%%
\medskip
%%%%%

\begin{proof}[Proof of Lemma~\ref{lem:variance-check-gamma}]
Fix \(h\in[N-1]\) and set
\(\eta_j:=(3/\alpha)W_{j,2}\) for \(j\in[N]\).
Then \(Z_{j,2}=S_j+\eta_j\), where
\(S_j=\operatorname{sgn}(X_j)\), and
\[
\mathbb E[\eta_j]=0,
\qquad
\mathbb E[\eta_j^2]=\sigma_\alpha^2.
\]
Define
\begin{align*}
A_h&:=\frac1N\sum_{j=1}^{N-h}S_jS_{j+h},
&
B_h&:=\frac1N\sum_{j=1}^{N-h}S_j\eta_{j+h},\\
C_h&:=\frac1N\sum_{j=1}^{N-h}\eta_jS_{j+h},
&
D_h&:=\frac1N\sum_{j=1}^{N-h}\eta_j\eta_{j+h}.
\end{align*}
We then have
\[
\widehat\gamma^{(2)}(h) \ := \
\frac1N\sum_{j=1}^{N-h}Z_{j,2}Z_{j+h,2} =A_h+B_h+C_h+D_h.
\]
Since the variables \(\eta_j\) are centered, i.i.d., and independent
of \((S_j)_{j\in[N]}\), the random variables \(B_h\), \(C_h\), and
\(D_h\) are centered. Therefore,
\[
\widehat\gamma^{(2)}(h)
-
\mathbb E\!\left[\widehat\gamma^{(2)}(h)\right]
\ = \
A_h-\mathbb E[A_h]+B_h+C_h+D_h.
\]
Using \((a+b+c+d)^2\leq4(a^2+b^2+c^2+d^2)\), we obtain
\begin{align}
\label{eq:variance-decomposition-check-gamma}
\operatorname{Var}\!\left(\widehat\gamma^{(2)}(h)\right)
\ &\leq \
4\operatorname{Var}(A_h)
+
4\mathbb E[B_h^2]
+
4\mathbb E[C_h^2]
+
4\mathbb E[D_h^2].
\end{align}

First,
\[
\operatorname{Var}(A_h)
\ = \
\frac1{N^2}
\sum_{s,u=1}^{N-h}
\operatorname{Cov}\!\left(
S_sS_{s+h},
S_uS_{u+h}
\right)
\ \leq \
\frac1{N^2}
\sum_{s,u=1}^{N-h}
\left|
\operatorname{Cov}\!\left(
S_sS_{s+h},
S_uS_{u+h}
\right)
\right|.
\]
Using the independence of \((S_j)_{j\in[N]}\) and
\((\eta_j)_{j\in[N]}\), we obtain
\begin{align*}
\mathbb E[B_h^2]
\ = \ 
\frac1{N^2}
\sum_{s,u=1}^{N-h}
\mathbb E[S_sS_u]\,
\mathbb E[\eta_{s+h}\eta_{u+h}]
\ = \ 
\frac{(N-h)\sigma_\alpha^2}{N^2},
\end{align*}
where the last equality follows because the variables \(\eta_j\) are
independent and centered, and because \(S_s^2=1\). The same argument
gives
\[
\mathbb E[C_h^2]
\ = \
\frac{(N-h)\sigma_\alpha^2}{N^2}.
\]
Finally, by the definition of \(D_h\),
\[
\mathbb E[D_h^2]
\ = \
\frac1{N^2}
\sum_{s,u=1}^{N-h}
\mathbb E[\eta_s\eta_{s+h}\eta_u\eta_{u+h}]
\ = \
\frac{(N-h)\sigma_\alpha^4}{N^2} .
\]
Indeed, when \(s\neq u\), at least one of the independent centered variables
in the product above appears exactly once, so its expectation
vanishes. Therefore, only the terms with \(s=u\) remain.

Plugging these bounds into
\eqref{eq:variance-decomposition-check-gamma} yields
\[
\operatorname{Var}\!\left(\widehat\gamma^{(2)}(h)\right)
\ \leq \ 
\frac4{N^2}
\sum_{s,u=1}^{N-h}
\left|
\operatorname{Cov}\!\left(
S_sS_{s+h},
S_uS_{u+h}
\right)
\right|
+
\frac{4(N-h)}{N^2}
\left(
2\sigma_\alpha^2+\sigma_\alpha^4
\right),
\]
which proves the lemma.
\end{proof}

%%% Local Variables:
%%% mode: latex
%%% TeX-master: "main.tex"
%%% ispell-local-dictionary: "american"
%%% End:

%%%%%

\end{document}